\documentclass[12pt]{amsart}
\usepackage{amscd,amsthm,amsfonts,amssymb,amsmath,euscript}
\usepackage[dvips]{graphicx}
\usepackage[dvips]{graphics}
\usepackage[matrix,arrow]{xy}
\usepackage{longtable}
\usepackage{color}
\usepackage{url}
\usepackage[T2A]{fontenc}
\usepackage{multirow}
\usepackage%[shortlabels]
{enumerate}
\makeatletter\@addtoreset{equation}{part} \makeatother

\makeatletter\@addtoreset{section}{part} \makeatother

\newtheorem{theorem}[equation]{Theorem}

\newtheorem{proposition}[equation]{Proposition}
\newtheorem{lemma}[equation]{Lemma}
\newtheorem{corollary}[equation]{Corollary}
\newtheorem{conjecture}[equation]{Conjecture}
\newtheorem{principle}[equation]{Principle}
\newtheorem{optimisticpicture}[equation]{Optimistic picture}
\newtheorem{question}[equation]{Question}
\newtheorem{problem}[equation]{Problem}
\newtheorem{fact}[equation]{Fact}

\theoremstyle{definition}
\newtheorem{example}[equation]{Example}
\newtheorem{definition}[equation]{Definition}

\newtheorem{compactificationconstruction}[equation]{Compactification construction}

\theoremstyle{remark}
\newtheorem{remark}[equation]{Remark}
\theoremstyle{remark}

\newcommand{\QQ}{{\mathbb Q}}
\newcommand{\ZZ}{{\mathbb Z}}
\newcommand{\NN}{{\mathbb N}}

\newcommand{\PP}{{\mathbb P}}
\newcommand{\CC}{{\mathbb C}}
\newcommand{\RR}{{\mathbb R}}

\newcommand{\TT}{{\mathbb T}}
\newcommand{\FF}{{\mathbb F}}
\newcommand{\Aff}{{\mathbb A}}

\def \P {\mathbb{P}}
\def \R {\mathbb{R}}

\newcommand{\bbA}{{\mathbb A}}
\newcommand{\bbC}{{\mathbb C}}
\newcommand{\bbH}{{\mathbb H}}
\newcommand{\bbP}{{\mathbb P}}

\def \C {\mathbb{C}}
\def \sA {\mathbb{A}}
\def \sD {\mathbb{D}}
\def \sE {\mathbb{E}}
\def \QQQ {\mathcal{Q}}
\def \TTT {\mathcal{T}}

\newcommand{\pic}{\mathrm{Pic}\,}
\newcommand{\Pic}{\mathrm{Pic}\,}

\newcommand{\virt}{\mathrm{virt}}
\newcommand{\Spec}{\mathrm{Spec}\,}

\newcommand{\cO}{\mathcal{O}}

\newcommand{\vol}{\mathrm{vol}\,}

\newcommand{\iX}{i_X}
\newcommand{\id}{\mathrm{id}}
\newcommand{\init}{\mathrm{init}\,}
\newcommand{\scr}[1]{\ensuremath{\mathcal{#1}}}
\def \Res {\mathrm{Res}\,}
\newcommand{\iY}{i_Y}

\def \G {\mathrm{Gr}}
\def \Ver {\mathrm{Ver}}
\def \Ar {\mathrm{Ar}}
\def \vv {\mathbf{v}}
\def \hh {\mathbf{h}}
\def \wt {\mathrm{wt}}
\def \O {\mathcal{O}}

\def \HB {\mathrm{HB}}
\def \MB {\mathrm{MB}}
\def \VB {\mathrm{VB}}
\def \X {\mathcal{X}}
\def \Y {\mathcal{Y}}

\newcommand{\arrow}[2]{\langle #1\to #2 \rangle}

\def \ge {\geqslant}
\def \le {\leqslant}
\def \geq {\geqslant }
\def \leq {\leqslant}

\begin{document}

\title{Toric Landau--Ginzburg models}

\author{Victor Przyjalkowski}

\thanks{The author was partially supported by Laboratory of Mirror Symmetry NRU
HSE, RF government grant, ag. No. 14.641.31.0001. % and by Young Russian Mathematics award.
He is a Young Russian Mathematics award winners and would like to thank
its sponsors and jury.}

\address{\emph{Victor Przyjalkowski}
\newline
\textnormal{Steklov Mathematical Institute of Russian Academy of Sciences, 8 Gubkina street, Moscow 119991, Russia.}
\newline
\textnormal{National Research University Higher School of Economics, Russian Federation, Laboratory of Mirror Symmetry, NRU HSE, 6 Usacheva str., Moscow, Russia, 119048.}
\newline
\textnormal{\texttt{victorprz@mi.ras.ru, victorprz@gmail.com}}}

%\keywords{Weak Landau--Ginzburg model, Fano threefold, toric degeneration, intermediate Jacobian}

\maketitle

\begin{abstract}
This is a review of the theory of toric Landau--Ginzburg models --- the effective approach to mirror symmetry for Fano varieties.
We mainly focus on the cases of dimensions $2$ and $3$, as well as on the case of complete intersections in weighted projective spaces
and Grassmannians. Conjectures that relate invariants of Fano varieties and their Landau--Ginzburg models,
such as Katzarkov--Kontsevich--Pantev conjectures, are also studied.
\end{abstract}

\tableofcontents

\part{Introduction.}
\label{part:introduction}

One of the most brilliant ideas in mathematics in the last three decades is Mirror Symmetry.
As it often happens, it came to mathematics from physics.
That is, Calabi--Yau threefolds (i.\,e. varieties of complex dimension $3$ with non-vanishing everywhere defined holomorphic $3$-form)
play a central role in elementary particles description in the string theory.
These varieties, enhanced by symplectic forms and complex structures, can be considered as symplectic or algebraic manifolds.
Physicists noticed that these varieties come in (non-uniquely defined) pairs such that symplectic properties
of Calabi--Yau manifold~$X$ (the so called branes of type~$A$)
correspond to algebraic properties of its pair~$Y$ (the so called branes of type~$B$)
and, vice-versa, symplectic properties for~$Y$ correspond to algebraic properties for~$X$.
One of numerical consequences of the correspondence is Mirror Symmetry of Hodge numbers.
It states that~$h^{i,j}(X)=h^{i,3-j}(Y)$.
One can say that putting a mirror to Hodge diamond for~$X$ one can see the Hodge diamond for~$Y$.
This justifies the term ``Mirror Symmetry''.

Just after making this breakthrough it was straightforwardly generalized to higher-dimensional Calabi--Yau varieties.
Some numerical consequences of the discovery were also formulated, which enabled one to
formulate the idea of Mirror Symmetry mathematically.
The first example of the phenomena is the famous paper~\cite{COGP91},
where the generic quintic threefold in $\PP^4$ was considered.
The certain series for the hypersurface was considered,
that is the one constructed by expected numbers of rational curves of given degree
lying on the variety (Clemens conjecture states that for very generic quintic the numbers are finite).
The certain one-dimensional family was considered. The period for the family, that is the function
given by integrals of fiberwise forms over fiberwise cycles, after certain transformation
coincides with the series for the quintic.
This principle of correspondence of the series constructed by numbers of rational curves lying on the
manifold and periods of the dual one-parameter family is the basement of the Mirror Symmetry conjecture of
variations of Hodge structures.

The following generalization of Mirror Symmetry is the one for Fano varieties,
that is varieties with ample anticanonical class.
Such varieties play an important role in algebraic geometry: for instance, they are
the main ``building bricks'' in Minimal Model Program.
Moreover, they have rich geometry; say, a lot of rational curves lie on them.
In opposite to the Calabi--Yau case, mirror partners for Fano varieties are not varieties of the same kind but
certain varieties together with complex-valued functions called \emph{superpotential}.
Such varieties are called Landau--Ginzburg models and they can be described as
one-dimensional families of fibers of superpotentials.
In particular, fibers of the families are Calabi--Yau varieties mirror dual to anticanonical sections
of the Fano varieties.
Mirror Symmetry conjecture of variations of Hodge structures claims the correspondence between
$I$-series that are constructed by Gromov--Witten invariants, that are expected numbers
of rational curves of given degree lying on the manifold (it's important here that it is Fano or close to be Fano
to have enough rational curves) and periods of the dual family.
In other words it claims the coincidence of the second Dubrovin's connection for the Fano manifold
and the Gauss--Manin connection for the dual Landau--Ginzburg model, or
coincidence of regularized quantum differential equation of the variety and Picard--Fuchs differential
equation of the dual model.

The first and the main example when Mirror Symmetry conjecture of variations of Hodge structures holds
was given by Givental (see~\cite{Gi97b}, and also~\cite{HV00})).
He constructed Landau--Ginzburg models for complete intersections in smooth toric varieties.
This construction can be generalized to complete intersections in singular toric varieties
and, more general, varieties admitting ``good'' toric degenerations, such as Grassmannians
or partial flag manifolds (see~\cite{BCFKS97} and~\cite{BCFKS98}).
Moreover, Givental's model for a toric variety $T$ can be simplified by expressing monomially some variables in terms of others
such that the superpotential becomes a Laurent polynomial in $\dim (T)$ variables.
The Newton polytope of the Laurent polynomial coincide with the fan polytope for $T$,
that is a convex hull of integral generators of rays of a fan for $T$.
For a complete intersection it is often possible to make one more birational change of variables
after which the superpotential remains being representable by a Laurent polynomial.
Moreover, this change of variables transforms Givental integral (that express periods) correctly.

Consider a Gorenstein toric variety $T$. Its fan polytope is reflexive, which means that the dual the polytope
is integral.
%Anticanonical linear system on $T$ can be described as the linear system of Laurent polynomials with support on the dual polytope.
%Thus Givental's Landau--Ginzburg model for an anticanonical section of the toric variety can be described
%as an anticanonical section of the dual toric variety.
%It turns out that Mirror Symmetry conjecture of Hodge numbers holds for this duality.
%That is, c
Consider the dual to $T$ toric variety $T^\vee$;
in other words the varieties
$T$ and $T^\vee$ are defined by dual polytopes.
Let $X$ be a Calabi--Yau complete intersection in $T$ of dimension $n$,
which is defined by some nef-partition.
Batyrev and Borisov (see~\cite{BB96}) defined the dual nef-partition,
which gives the dual Calabi--Yau variety $Y$.
According to Givental, Mirror Symmetry conjecture of variations
of Hodge structures holds for $X$ and $Y$.
In loc. cit it is shown that
$$
h_{st}^{p,q}(X)=h_{st}^{p,n-q}(Y),
$$
where $h_{st}^{p,q}(Y)$ are stringy Hodge numbers.
(In particular in our case they coincide with Hodge numbers of a crepant resolution of $Y$,
which, by Batyrev's theorem (see~\cite{Ba99}) do not depend on the particular resolution.)
Thus in our case Mirror Symmetry conjecture for Hodge numbers follows from Mirror Symmetry conjecture of variations
of Hodge structures.
In the Fano case one can't claim the correspondence of Hodge diamonds
because the dual objects are not varieties but families of varieties.
In~\cite{KKP17} the analogues of Hodge numbers for ``tame compactified Landau--Ginzburg models''
were defined (in three ways). The authors made a conjecture about mirror correspondence for them.
We in particular study these conjectures in this paper, correct them a bit and observe schemes of
their proofs for the two- and three-dimensional cases.

The next step is Kontsevich's Homological Mirror Symmetry conjecture.
It states mirror correspondence in terms of derived categories.
That is, considering Fano manifold $X$ as an algebraic variety one can construct the derived category of coherent sheaves $D^b(coh\ X)$,
and considering $X$ as a symplectic variety (with chosen symplectic form) one can construct the Fukaya category $Fuk (X)$,
whose objects are Lagrangian submanifolds for the symplectic form,
and morphisms are Floer homology.
On the other hand, similar categories can be defined for a Landau--Ginzburg model $w\colon Y\to \CC$.
Analogue of the derived category of coherent sheaves for the Landau--Ginzburg model is the derived category
of singularities $D^b_{sing}(Y,w)$, that is a product over all singular fibers
of quotients of categories of coherent sheaves by subcategories of perfect complexes.
Analogue of the Fukaya category is the Fukaya--Seidel category $FS(Y,w)$,
whose objects are vanishing to singularities Lagrangian cycles (for chosen symplectic form on the Landau--Ginzburg model).
Homological Mirror Symmetry conjecture states the equivalences
\begin{equation*}
Fuk(X)\cong D^b_{sing}(Y,w),\ \ \ \
D^b(coh\ X)\cong FS(Y,w).
\end{equation*}

Homological Mirror Symmetry conjecture is very powerful. For instance, the Bondal--Orlov theorem states that a Fano variety
can be reconstructed from its derived category of coherent sheaves.
However because of the deepness of the conjecture it is hard to prove it even for the simplest cases.
The positive examples are the partial proofs of the conjecture (that is, the proof of one of equivalences
in the conjecture) for del Pezzo surfaces (\cite{AKO06}),
toric varieties (\cite{Ab09}), and some of hypersurfaces (\cite{Sh15}).
Let us mention that Mirror Symmetry conjecture of variations of Hodge structures is claimed to be a numerical consequence of Homological
Mirror Symmetry conjecture, since the equivalence of categories implies the isomorphism of their Hochschild cohomologies,
which in our case correspond to quantum cohomology and variations of Hodge structures.

It is expected that different versions of mirror symmetry conjectures agree one with others.
This means that Givental's Landau--Ginzburg models satisfy Homological Mirror Symmetry conjecture.
More precise, the following compactification principle should hold:
there should exist fiberwise (log) compactification of a Landau--Ginzburg model
which, after choosing a symplectic form, satisfies Homological Mirror Symmetry conjecture.
In particular, fibers of the compactification should be Calabi--Yau varieties mirror dual
to anticanonical sections of the Fano variety. These three properties (correspondence
to Gromov--Witten invariants, the existence of compactification of a family of Calabi--Yau varieties,
and a connection with toric degenerations) justify the notion of toric Landau--Ginzburg model
which is central in this paper. Similarly to the case of smooth toric varieties
(but not complete intersections therein!), toric Landau--Ginzburg model is an algebraic torus
together with non-constant complex-valued function satisfying the properties discussed above.
Since the function on the torus (after choosing a basis) is nothing but a Laurent polynomial,
we call the Laurent polynomial (satisfying the properties) toric Landau--Ginzburg model.
See Part~\ref{part:toric LG} for the precise definition.

Strong version of Mirror Symmetry conjecture of variations of Hodge structures claims
the existence of toric Landau--Ginzburg model for each smooth Fano variety.

The notion of toric Landau--Ginzburg model turned out to be an effective tool for studying mirror symmetry conjectures.
This paper is a review of the theory of toric Landau--Ginzburg models.
In particular, we construct them for a large class of Fano varieties such as del Pezzo surfaces,
Fano threefolds, complete intersections in (weighted) projective spaces and Grassmannians.
We also construct their compactifications and study their properties, invariants,
and related conjectures.

\medskip

We present only sketches of proofs for a lot of results in the paper; one can find details
in the references. The paper is organized as follows.
Part~\ref{part:preliminaries} contains definitions and preliminaries needed for the following.
Part~\ref{part:toric LG} is devoted to the notion of toric Landau--Ginzburg models.
Del Pezzo surfaces are discussed in Part~\ref{part: del Pezzo surfaces}.
We present there a precise construction of toric Landau--Ginzburg model depending on a divisor
on a del Pezzo surface.

Part~\ref{part:Fano 3-folds}, central in the paper, is devoted to the threefold case.
Section~\ref{section: weak LG 3-folds} contains construction of weak Landau--Ginzburg models.
In Section~\ref{section: Calabi--Yau 3-folds} (log) Calabi--Yau com\-pac\-ti\-fi\-cations are constructed. In Section~\ref{section: toric LG 3-folds} we discuss toric degenerations of Fano threefolds that correspond to their weak Landau--Ginzburg models.
We also present a certain construction for the Picard rank one case.
In Section~\ref{section:Modularity} we compute Picard lattices of fibers of Landau--Ginzburg
models for the Picard rank one case and show that the fibers are Dolgachev--Nikulin mirrors to anticanonical sections of Fano varieties.

In Part~\ref{part:KKP} we study Katzarkov--Kontsevich--Pantev conjectures on Hodge numbers of Landau--Ginzburg models.
In Section~\ref{section: KKP} we, following~\cite{KKP17}, define and discuss Hodge numbers of Landau--Ginzburg models
and Katzarkov--Kontsevich--Pantev conjectures.
In Section~\ref{section: KKP for surfaces} we prove the conjectures for del Pezzo surfaces. Finally, in Section~\ref{subsection:KKP-3}
we present a scheme of the proof of the conjectures in the threefold case.

Part~\ref{part:complete intersections in Grass} is devoted to the higher-dimensional case,
that is the cases of (weighted) complete intersections and Grassmannians.
A general Givental's construction of Landau--Ginzburg models for complete intersections in smooth toric varieties
is presented in Section~\ref{section: complete Givental}.
The most of results in the rest sections are related to the problem of existence of ge\-ne\-ra\-lizations of such models
and to the question if they are birational to weak Landau--Ginzburg mo\-dels. In Section~\ref{section:complete intersections}
we consider the case of weighted complete intersections. We present there results on existence of nef-partitions that guarantee
the existence of weak Landau--Ginzburg models. In the case of complete intersections in the usual projective spaces
we show existence of Calabi--Yau compactifications and toric degenerations.
The rest of the section contains boundness results for families of smooth complete intersections.
More details on this part one can find in the review~\cite{PSh18} (in preparation).
Finally, in Section~\ref{section: CI in Grass} we consider the case of complete intersections in Grassmannians.
For each of such complete intersection we show the existence of Batyrev--Ciocan-Fontanine--Kim--van Straten construction
which is birationally equivalent to weak Landau--Ginzburg models.

\section*{Notation and conventions}

All varieties are considered over the field of complex numbers $\CC$.

We consider only~{genus zero} Gromov--Witten invariants.

Homology $H_*(X,\ZZ)$ and cohomology $H^*(X,\ZZ)$ we denote by $H_*(X)$ and $H^*(X)$ respectively.
Cohomology with compact support (of a variety $X$ with coefficients in the constant sheaf $\CC_X$) we denote by $H_c^* (X)$.
Poincare dual class to $\gamma\in H^*(X)$ we denote by $\gamma^\vee$.
The space $\Pic (X)\otimes \CC$ we denote by $\Pic (X)_\CC$.

For any two numbers $n_1$ and $n_2$ we denote the set $\{i\mid n_1\le i\le n_2\}$ by $[n_1,n_2]$.

Calabi--Yau variety for us is a projective variety with trivial canonical class.

We often denote a Cartier divisor on a variety $X$ and its class in $\pic(X)$ by the same symbol.

A smooth degree $d$ del Pezzo surface (except for the quadric surface) we denote by $S_d$.

A smooth Fano variety (considered as an element of a family of the varieties of its type)
of Picard rank $k$ and number $m$ in the lists in~\cite{IP99} we denote by $X_{k-m}$.

We use the notation $\PP(w_0,\ldots,w_n)$ for a weighted projective space with weights $w_0,\ldots,w_n$.
(Weighted) projective spaces with coordinates $x_0,\ldots,x_n$ we denote by $\PP[x_0:\ldots:x_n]$.
Affine space with coordinates $x_0,\ldots,x_n$ we denote by $\Aff[x_1,\ldots,x_n]$.

The ring $\CC[x_1^{\pm 1},\ldots,x_n^{\pm 1}]$ we denote by $\TT[x_1,\ldots, x_n]$.
The torus $\Spec \TT[x_1,\ldots, x_n]$ we denote by $\TTT[x_1,\ldots, x_n]$.

An integral polytope $\Delta\in \ZZ^n\otimes \RR$ for us is a polytope with integral vertices, that is ones lying in $\ZZ^n$.
An integral length of an integral segment is a number of integral points on it minus one.

We consider pencils in the birational sense. That is, a pencil for us is a family birational to a family of fibers of a map to $\PP^1$.

\part{Preliminaries.}
\label{part:preliminaries}

\section{Gromov--Witten invariants and $I$-series}
\label{section:GW invariants}

In this part we introduce notions and notation of Gromov--Witten theory we need. Details one can find, say, in~\cite{Ma02}.

\begin{definition}[{\cite[V--3.3.2]{Ma02}}]
\emph{The moduli space of stable maps} of rational curves of
class $\beta\in H_2(X)$ with $n$ marked points to $X$ is the
Deligne--Mumford stack {(}see~\cite{Ma02}, V--5.5{)}
of stable maps $f\colon C\to X$ of genus $0$ curves with $n$ marked points
such that что $f_*(C)=\beta$.
\end{definition}

Consider \emph{the evaluation maps}
$ev_i:\bar{M}_n(X,\beta)\rightarrow X$, given by
$ev_i(C;p_1,\ldots,p_n,f)=f(p_i)$. Let
$\pi_{n+1}:\bar{M}_{n+1}(X,\beta)\rightarrow \bar{M}_n(X,\beta)$
be \emph{the forgetful map} at the point $p_{n+1}$ which forget
this point and contract unstable component after it. Consider the
sections $$\sigma_i:\bar{M}_n(X,\beta)\rightarrow
\bar{M}_{n+1}(X,\beta)$$ defined as follows. The
image of a curve $(C;p_1,\ldots,p_n,f)$ under the map $\sigma_i$ is a
curve $(C';p_1,\ldots,p_{n+1},f')$. Here $C'=C\bigcup C_0$,
$C_0\simeq \PP^1$, $C_0$ and $C$ intersect at the non-marked
point $p_i$ on $C'$, and $p_{n+1}$ and $p_i$ lie on
 $C_0$. The map $f'$ contracts the curve $C_0$ to the point and $f'|_C=f$.

Consider the sheaf $L_i$ given by
$L_i=\sigma_i^*\omega_{\pi_{n+1}}$, where $\omega_{\pi_{n+1}}$ is a
relative dualizing sheaf of $\pi_{n+1}$. Its fiber over the point
$(C;p_1,\ldots,p_n,f)$ is $T^*_{p_i}C$.

\begin{definition}[{\cite[VI--2.1]{Ma02}}]
The \emph{cotangent line
class} is the class
$$
\psi_i=c_1(L_i)\in H^2(\bar{M}_n(X,\beta)).
$$
\end{definition}

\begin{definition}[{\cite[VI--2.1]{Ma02}}]  {
Consider
$$
\gamma_1,\ldots, \gamma_n\in H^*(X),
$$
let
$a_1,\ldots,a_n$ be non-negative integers, and let $\beta \in H_2(X)$. Then \emph{the
Gromov--Witten invariant with descendants} is the number given by
\begin{equation*}
\langle\tau_{a_1} \gamma_1,\ldots,
\tau_{a_n}\gamma_n\rangle_\beta=
ev_1^*\gamma_1\cdot\psi_1^{a_1}\cdot\ldots \cdot
ev_n^*\gamma_n\cdot \psi_n^{a_n}\cdot [\bar{M}_n(X,\beta)]^\virt
\end{equation*}
if $\sum \mathrm{codim}\, \gamma_i+\sum a_i=\mathrm{vdim}\,
\bar{M}_n(X,\beta)$ and 0 otherwise. The invariants with $a_i=0$, $i=1\ldots,n$, are called \emph{prime}. We omit symbols
$\tau_0$ in this case. }
\end{definition}

Gromov--Witten invariants are usually ``packed'' into different structures
for convenience.
The simplest ones are one-pointed ($n=1$); they are usually packed to $I$-series.

Let $X$ be a smooth Fano variety of dimension $N$ and Picard number $\rho$.
Choose a basis
$$%\mathcal H=
\{H_1,\ldots,H_{\rho}\}$$
in $H^2(X)$
so that for any $i\in [1,\rho]$ and any curve $\beta$ in the K\"ahler cone
$K$ of $X$ one has~\mbox{$H_i\cdot\beta\ge 0$}.
Introduce formal variables $q^{\tau_i}$, $i\in [1,\rho]$ and denote $q_i=q^{\tau_i}$.
For any~\mbox{$\beta\in H_2(X)$} denote
$$q^\beta=q^{\sum \tau_i (H_i\cdot \beta)}.$$
Consider the Novikov ring $\C_q$, i.\,e. a group ring for $H_2(X)$. We treat it as a ring of polynomials over $\C$ in formal variables
$q^\beta$, with relations
$$q^{\beta_1}q^{\beta_2}=q^{\beta_1+\beta_2}.$$
Note that for any $\beta\in K$ the monomial $q^\beta$ has non-negative degrees in $q_i$.

\begin{definition}[{details see in \cite{Ga00}, \cite{Prz07a}}]
\label{definition:I-series}
Let $\mu_1,\ldots, \mu_N$ be a basis in $H^*(X)$ and let
$\check{\mu}_1,\ldots,\check{\mu}_N$ be the dual basis.
The \emph{$I$-series} (or Givental \emph{$J$-series}) for $X$ is given by the following.
\begin{gather*}
I^X_\beta=ev_*\left(\frac{1}{1-\psi}\cdot
[\bar{M}_1(X,\beta)]^\virt\right)=\sum_{i,j\geqslant 0}
\langle\tau_i \mu_j\rangle_\beta\check{\mu}_j,\\
I^X(q_1,\ldots,q_{\rho})=1+\sum_{\beta\in K} I_\beta^X\cdot q^\beta.
\end{gather*}

The \emph{constant term} of $I$-series $I^X_{0}$
is
\begin{gather*}
I^X_{0}(q_1,\ldots,q_{\rho})=1+\sum_{\beta\in K}
\langle\tau_{(-K_X)\cdot\beta-2} {\mathbf{1}}\rangle_\beta
\cdot q^\beta,
\end{gather*}
where $\mathbf{1}$ is the fundamental class.
(The map $ev$ and the cotangent line class are unique for one-pointed invariants, so we omit indices.)
The series
$$
\widetilde{I}^X_{0}(q_1,\ldots,q_{\rho})=
1+\sum_{\beta \in K} (-K_X\cdot \beta)!\langle\tau_{-K_X\cdot\beta-2} \mathbf 1\rangle_{\beta}
\cdot q^\beta
$$
is called the \emph{constant term of regularized $I$-series} for~$X$.

\end{definition}

Consider the class of divisors $H=\sum \alpha_i H_i$.
One can restrict $I$-series, the usual and the regularized ones, to the direction corresponding to the divisor class putting
$\sigma_i=\alpha_i \sigma$ and~\mbox{$t=q^\sigma$}.
Let us fix a divisor class $D$.
We are interesting in restrictions of $I$-series on the \emph{anticanonical direction corresponding to
$D$}.
For this we change $q^\beta$ by $e^{-D\cdot \beta}t^{-K_X\cdot \beta}$.
In particular one can define a \emph{restriction of the constant term of regularized $I$-series to the anticanonical direction} (so that $D=0$).
It has the form
$$
\widetilde{I}^X_0(t)=1+a_1t+a_2t^2+\ldots,\ \ \ \ a_i\in \CC.
$$

\section{Toric geometry}
\label{section:toric}

The definition and the main properties of toric varieties see
in~\cite{Da78} or in~\cite{Fu93}. Let us just remind that toric variety is
a variety with an action of a torus $\Spec(\CC^*)^N$ such that
one of its orbits is a Zariski open set. Toric variety is determined
by its \emph{fan}, i. e. some collection of cones with vertices in
the points of lattice that is dual to the lattice of torus
characters. Moreover, algebro-geometric properties of toric
variety can be formulated in terms of properties of this fan.
Remind some of them.

Every cone of the fan $\kappa\subset \mathcal N_{\RR}=\mathcal N\otimes \RR$, $\mathcal N\simeq \ZZ^N$ of dimension $r$
corresponds to the orbit of the torus of dimension $N-r$. Thus, each edge (one-dimensional
cone) correspond to the (equivariant) Weil
divisor. That is, let $\Sigma\in \mathcal N$ be a
fan of the toric variety $X_\Sigma$ and let $\sigma\in \Sigma$ be
any cone. Let $\mathcal M$ be a lattice dual to $\mathcal N$ with respect to some
non-degenerate pairing $\langle \cdot,\cdot\rangle$ and
$\sigma^\vee$ be a dual cone for $\sigma$ (i.~e. $\sigma^\vee=
\{l\in \mathcal M| \forall k\in \sigma \ \ \langle l,k\rangle\geq 0\}$). Let
$U_\sigma=\mathrm{Spec}\, \CC [\sigma^\vee]$ correspond to
$\sigma$. The variety $X_\sigma$ is obtained from the affine
varieties $U_\sigma$, $\sigma\in \Sigma$, by gluing together
$U_\sigma$ and $U_\tau$ along $U_{\sigma\cap\tau}$, $\sigma,\tau\in
\Sigma$. %Thus, if $l\subset\sigma\in \Sigma$ is an edge of the fan,
%then the divisor that is correspond to $l$ restricted on $U_\sigma$
%as $U_l\subset U_\sigma$.
The divisors which correspond to the edges of the fan generate
divisor class group. A Weil divisor $D=\sum d_i M_i$, where $M_i$
corresponds to edges, is Cartier if for each cone of the fan
$\sigma$ there exist a vector $n_\sigma$ such that $\langle
n_\sigma, m_i\rangle=d_i$ where $m_i$ are primitive elements of
the edges of this cone. If such vector is the same for all cones,
then the divisor is principal. Hence if the toric variety is
$N$-dimensional and the number of the edges is $N+\rho$, then the rank of
the divisor class group is $\rho$.

\begin{definition}
The variety is called \emph{$\QQ$-factorial} if some multiple of each Weil
divisor is a Cartier
divisor.
\end{definition}

In particular, there exists an intersection theory for Weil divisors
on the $\mbox{$\QQ$-factorial}$ variety.
Toric variety is $\QQ$-factorial if and only if any cone of the
fan, which corresponds to this variety, is simplicial. In this
case the Picard group is generated (over $\QQ$) by divisors, which
correspond to the edges of the fan.

Consider a weighted projective space
$\PP=\PP(w_0,\ldots,w_N)$. The fan which corresponds to it is
generated by integer vectors $m_0,\ldots,m_N\in \RR^N$ such that
$\sum w_i m_i=0$. If $w_0=1$, then one can put $m_0=(-w_1,\ldots,
-w_N)$, $m_i=e_i$, where $e_i$ is a basis of $\RR^N$. The collection
$\{m_i\}$ corresponds to the collection of standard divisors--strata
$\{ D_i\in
H^0\left(\mathcal O_\PP(w_i)\right)\}$.

A toric variety is \emph{smooth}
if for any cone $\sigma$ in the fan that correspond to this
variety the subgroup $\sigma\cap \ZZ^N$ is generated by the subset
of the basis of the lattice $m_1^\sigma, \ldots, m_k^\sigma$.
Adding the edge $a=a_1 m_1^\sigma+\ldots+a_km_k^\sigma$, $a_i\in
\QQ$ to the cone (and connection it with ``neighbor'' faces)
corresponds to weighted blow up (along subvariety which correspond
to $\sigma$) with weights $1/r\cdot(\alpha_1,\ldots,\alpha_k)$,
where $\alpha_i\in \ZZ$ and $a_i=\alpha_i/r$. Thus we can get toric resolution of a
toric variety adding consecutively
edges to the fan in this way.

In particular, singular locus of $\PP$ is the union of
strata given by $x_{i_1}=\ldots=x_{i_j}=0$, where $x_{i_j}$ is the
coordinate of weight $w_{i_j}$ and $\{ i_1,\ldots,i_j \}$ is the
maximal set of indices such that greatest common factor of the
others is greater than $1$, see Lemma~\ref{lemma:singularities-of-P}.

Let $X$ be a factorial $N$-dimensional toric Fano variety of Picard rank $\rho$ corresponding to
a fan $\Sigma_X$ in the lattice $\mathcal{N}$. Let $D_1,\ldots, D_{N+\rho}$ be its
prime invariant divisors.
Let $\mathcal{M}=\mathcal{N}^{\vee}$, and let $\mathcal{D}\simeq\ZZ^{N+\rho}$ be a lattice with
the basis $\{D_1,\ldots, D_{N+\rho}\}$ (so that one has a natural identification $\mathcal{D}\simeq\mathcal{D}^{\vee}$).
By~\cite[Theorem~4.2.1]{CLS11} one has an exact sequence
$$
0\to \mathcal{M} \to \mathcal{D}
\to A_{N-1}(X)=\pic(X)\simeq\ZZ^{\rho}\to 0.
$$
We use factoriality of $X$ here to identify the class group $A_{N-1}(X)$ and the Picard
group~\mbox{$\pic(X)$}.
Dualizing this exact sequence, we obtain
an exact sequence
\begin{equation}
\label{sequence}
0\to \pic(X)^{\vee}\to \mathcal{D}\to \mathcal{N}\to 0.
\end{equation}
Thus $\pic(X)^{\vee}$ can be identified with the lattice of relations on primitive vectors on the
rays of $\Sigma_X$ considered as Laurent monomials in variables $u_i$.
On the other hand, as the basis in $\pic(X)$ is chosen we can identify $\pic(X)^{\vee}$ and $\pic(X)=H^2(X)$.
Hence we can choose a basis in the lattice of relations on primitive vectors on the
rays of $\Sigma_X$ corresponding to $\{H_i\}$ and, thus, to $\{q_i\}$.
We denote these relations by $R_i$, and interpret them as monomials in
the variables~\mbox{$u_1,\ldots,u_{N+\rho}$}.

Consider a toric variety $T$.
\emph{A fan} (or \emph{spanning}) \emph{polytope} $F(T)$ is a convex hull of integral generators
of fan's rays for $T$.
Let
$$
\Delta=F(T)\subset \mathcal N_\RR.
$$
Let
$$
\nabla=\{x\ |\ \langle x,y\rangle \geq -1 \mbox{ for all } y\in \Delta \}\subset \mathcal M_\RR=\mathcal N^\vee\otimes \RR.
$$
be the dual polytope.

For an integral polytope $\Delta$ we associate a (singular) toric Fano variety $T_{\Delta}$
defined by a fan whose cones are cones over faces of $\Delta$. We also associate a (not uniquely defined) toric variety $\widetilde{T}_\Delta$
with $F(\widetilde T_\Delta)=\Delta$ such that for any toric variety $T'$ with $F(T')=\Delta$ and for any
morphism $T'\to \widetilde T_\Delta$ one has $T'\simeq \widetilde T_\Delta$. In other words, $\widetilde T_\Delta$ is
given by ``maximal triangulation'' of $\Delta$.

\begin{definition}
\label{definition:reflexive}
The variety $T_\Delta$ and the polytope $\Delta$ are called \emph{reflexive} if $\nabla$ is integral.
\end{definition}

Let $T$ be reflexive.
Denote $T_\nabla$ by $T^\vee$ and $\widetilde T_\nabla$ by $\widetilde T^\vee$.

Finally summarize some facts related to toric varieties and their anticanonical sections. One can see, say,~\cite{Da78} for details.
It is more convenient to start from the toric variety $T^\vee$ for the following.

\begin{fact}
\label{fact: toric 1}
Let the anticanonical class $-K_{T^\vee}$ be very ample (in particular, this holds in reflexive threefold case, see~\cite{JR06} and~\cite{CPS05}).
One can embed $T^\vee$ to a projective space in the following way.
Consider a set $A\subset M$ of integral points in a polytope $\Delta$ dual to $F(T^\vee)$. Consider a projective space
$\PP$ whose coordinates $x_i$ correspond to elements $a_i$ of $A$.
Associate a homogenous equation $\prod x_i^{\alpha_i}=\prod x_j^{\beta_j}$ with any homogenous relation $\sum \alpha_i a_i=\sum \beta_j a_j$, $\alpha_i, \beta_j\in \ZZ_+$.
The variety $T^\vee$ is cut out in $\PP$ by equations associated to all homogenous relations on $a_i$.

\end{fact}

\begin{fact}
\label{fact: toric 2}
 The anticanonical linear system of $T^\vee$ is a restriction of $\cO_\PP(1)$. In particular,
it can be described as (a projectivisation of) a linear system of Laurent polynomials whose Newton
polytopes contain in $\Delta$.
\end{fact}

\begin{fact}
\label{fact: toric 3}
Toric strata of $T^\vee$ of dimension $k$ correspond to $k$-dimensional faces of $\Delta$.
Denote by $R_{f}$ an anticanonical
section corresponding to a Laurent polynomial $f\in \CC[N]$ and by $F_Q$ a stratum corresponding to a face $Q$ of $\Delta$.
Denote by $f|_Q$ a sum of those monomials of $f$ whose support lie in $Q$.
Denote by $\PP_Q$ a projective space whose coordinates correspond to $Q\cap N$.
(In particular, $Q$ is cut out in $\PP_Q$ by homogenous relations on integral points of $Q\cap N$.)
Then $R_{Q,f}=R_f|_{F_Q}=\{f|_Q=0\}\subset \PP_Q$.
\end{fact}

\begin{fact}
\label{fact: toric 4}
In particular, $R_{f}$ does not pass through
a toric point corresponding to a vertex of $\Delta$ if and only if its coefficient
at this vertex is non-zero. The constant Laurent polynomial corresponds to
the boundary divisor of $T^\vee$.
\end{fact}

\part{Toric Landau--Ginzburg models}
\label{part:toric LG}

Consider a smooth Fano manifold $X$ of dimension $n$ and a divisor $D$ on it. Consider the restriction
$$
\widetilde{I}^{X,D}_{0}(t)
=1+\sum_{\beta \in K,\ a\in \ZZ_{\geq 0}} (-K_X\beta)!\langle\tau_{a} \mathbf 1\rangle_{\beta}
\cdot e^{-\beta\cdot D}t^{-K_X\cdot \beta}
$$
of constant term of regularized $I$-series corresponding to $D$.

Consider the torus $%\TT_{LG}=\mathbb
G_{\mathrm{m}}^n=\prod_{i=1}^n \TTT[x_i]$ and a function $f$
on it. This function can be represented by a Laurent polynomial: $f=f(x_1^{\pm 1}\ldots,x_n^{\pm 1})$. Denote
the constant term (that is a coefficient at $x_1^0\cdot \ldots
\cdot x_n^0$) of the polynomial $f$ by $[f]_0$ and put
$$
\Phi_f=\sum_{i=0}^\infty [f^i]_0 t^i\in \CC[[t]].
$$

\begin{definition} The series $\Phi_f$ %=\sum_{i=0}^\infty \phi_f(i)\cdot t^i$
is called \emph{the constant terms series} for $f$.
\end{definition}

\begin{definition}
\label{definition: main integral}
let $f$ be a Laurent polynomial in $n$ variables $x_1,\ldots,x_n$.
The integral
\begin{multline*}
I_f(t)=
\frac{1}{(2\pi i)^n}\int\limits_{|x_i|=\varepsilon_i}\frac{dx_1}{x_1}\wedge\ldots\wedge \frac{dx_n}{x_n}\frac{1
}{1-tf} =\\= \frac{1}{(2\pi i)^n}\sum_{j=0}^\infty t^j \cdot \int\limits_{|x_i|=\varepsilon_i}
f^j\frac{dx_1}{x_1}\wedge\ldots\wedge \frac{dx_n}{x_n}
 \in\CC[[t]]
\end{multline*}
is called \emph{the main period} for $f$, where $\varepsilon_i$ are some positive real numbers.
\end{definition}

\begin{remark}
\label{remark: Picard--Fuchs}
One has $I_f(t)=\Phi_f$.
\end{remark}

The following theorem justifies this definition.

\begin{theorem}[{see \cite[Proposition 2.3]{Prz08b}}]
\label{theorem: Picard--Fuchs}
Let $f$ be a Laurent polynomial in $n$ variables.
Let $PF_f=PF_f\left(t,
\frac{\partial}{\partial t}\right)$ be a Picard--Fuchs differential operator of the pencil of hypersurfaces in torus given by $f$.
Then~\mbox{$PF_f[I_f(t)]=0$}.
\end{theorem}

%\begin{remark}
%Denote the order of $PF_f$ by $m$ and denote the degree with
%respect to $t$ by $r$. Put $F_\alpha=\{1-\alpha f=0\}$, and let $Z$ be a semistable compactification of
%$\{F_t\}$ (so we have the map $Z\to \PP^1$;
%denote it for simplicity by $f$). Denote the dimension of
%the transcendental part of $R^{n-1}f_!\,\ZZ_Y$ by $m_f$ (for an algorithm for
%computing it see~\cite{DH86}), and denote the number of
%singularities (in particular, fake) of $f$ counted with multiplicities by $r_f$.
%Then $m\leq
%m_f$ and $r\leq r_f$.
%So we can write a differential operator of bounded order by $t$ and $D$ as an operator
%with indeterminant coefficients. As $I_f$ annihilates it, we get a system of infinite number of linear equations.
%To check that $L_X=PF_f$ we need to solve this system (it has a unique solution, up to scaling, so
%we need to solve a finite system of linear equations).
%
%However in practice it is enough to compare the first few coefficients of the
%expansion of $I_f$ and $\widetilde{I}_{0}^X$.
%Indeed, the first several coefficients of $\widetilde{I}^X_{0}$ determine $L_X$.
%Thus if the first several coefficients coincide with the first several coefficients of
%$I_f$, then the differential operator vanishing $\Phi_f$, coincide, up to higher order coefficients, with $L_X$, so that $PF_f=L_X$.
%\end{remark}

Now let us give the central definition of the paper.

\begin{definition}[{see~\cite[\S6]{Prz13}}]
\label{definition: toric LG}
\emph{A toric Landau--Ginzburg model} for a pair of a smooth Fano variety $X$ of dimension $n$ and divisor $D$ on it is a Laurent polynomial $\mbox{$f\in \TT[x_1, \ldots, x_n]$}$ which satisfies the following.
\begin{description}
  \item[Period condition] One has $I_f=\widetilde{I}^{X,D}_0$.
  \item[Calabi--Yau condition] There exists a relative compactification of a family
$$f\colon (\CC^*)^n\to \CC$$
whose total space is a (non-compact) smooth Calabi--Yau
variety $Y$. Such compactification is called \emph{a Calabi--Yau compactification}.
  \item[Toric condition] There is a degeneration
   $X\rightsquigarrow T_X$ to a toric variety~$T_X$ such that $F(T_X)=N(f)$, where $N(f)$ is the Newton polytope for $f$.
\end{description}

Laurent polynomial satisfying the period condition is called a \emph{weak Landau--Ginzburg model}.
\end{definition}

\begin{definition}[\cite{Prz17}]
\label{definition: log CY}
A compactification of the family $f\colon (\CC^*)^n\to \CC$ to a family $f\colon Z\to \PP^1$, where $Z$ is smooth and
$-K_Z=f^{-1}(\infty)$, is called a \emph{log Calabi--Yau compactification} (cf. Definition~\ref{def-3}).
\end{definition}

Now discuss why the notion of toric Landau--Ginzburg model is natural.

The period condition is nothing but Mirror Symmetry conjecture of variations of Hodge structures for the case when the ambient space is
an algebraic torus. This condition relates Gromov--Witten invariants with periods of the dual model.
Periods are integrals of fiberwise forms over fiberwise fibers. This means that they preserve under birational transformations which are
biregular in the neighborhood of the cycles we integrating over.
For instance, Givental constructed Landau--Ginzburg models in smooth toric varieties as quasi-affine varieties with superpotentials
(see Section~\ref{section: complete Givental} and Paragraph~\ref{section: Grass construction}).
However in a lot of cases these models are birational to algebraic tori, and the main periods preserve under the corresponding
birational isomorphisms, see Paragraph~\ref{section: weak CI}, Part~\ref{part:complete intersections in Grass} and~\cite{DH15} and~\cite{CKP14}.

The Calabi--Yau condition is going back to the following principle.

\begin{principle}[{Compactification principle,~\cite[Principle 32]{Prz13}}]
\label{principle:compactification}
There exists a fiberwise compactification of family of fibers of ``good'' toric Landau--Ginzburg model (defined up to flops)
satisfying (B side of) Homological Mirror Symmetry conjecture.
\end{principle}

In particular this means that there should exist a fiberwise compactification to (open) smooth Calabi--Yau variety which is
a family of smooth compact Calabi--Yau varieties.
This condition is strong enough: say, if $f(x_1,\ldots, x_n)$ is a toric Landau--Ginzburg model for a variety  $X$,
then for $k>1$ the Laurent polynomial $f(x_1^k,\ldots, x_n)$ satisfies the period condition for $X$, but not the
the Calabi--Yau condition and, thus, it does not satisfy the compactification principle.

\begin{example}
\label{example:4cubic}
The polynomials
$$
\frac{(x+y+1)^3}{xyzw}+z+w
$$
and
$$
\left(x_1+x_2+\frac{1}{x_1x_2}\right)\left(y_1+y_2+\frac{1}{y_1y_2}\right)
$$
satisfy the period and Calabi--Yau conditions for a cubic fourfold (see~\cite{KP09}).
However they are not fiberwise birational (they have different numbers of components of central fibers, cf. Section~\ref{section:Modularity}).
It is expected that the first polynomial satisfy the compactification principle.
\end{example}

One can easily see that the second Laurent polynomial in Example~\ref{example:4cubic}
is not toric Landau--Ginzburg model for a cubic fourfold: the degree of the corresponding toric variety
differs from the degree of cubic.

Finally, the toric condition goes back to Batyrev--Borisov construction of mirror duality for Calabi--Yau complete intersections
in toric varieties via the duality of toric varieties, see~\cite{Ba93}.

We consider mirror symmetry as a correspondence between Fano varieties and Laurent polynomials.
That is, the strong version of Mirror Symmetry conjecture of variations of Hodge structures states the following.
\begin{conjecture}
\label{conjecture:MS}
Any pair of a smooth Fano manifold and a divisor on it has a toric Landau--Ginzburg model.
\end{conjecture}

\begin{corollary}
Any smooth Fano variety has a toric degeneration.
\end{corollary}

This lets to hope to have the following picture.

\begin{optimisticpicture}[{\cite[Optimistic Picture 38]{Prz13}}]
\label{optimisticpicture}
Toric degenerations of smooth Fano varieties are in one-to-one correspondence with toric Landau--Ginzburg models.
They satisfy the compactification principle.
\end{optimisticpicture}

\begin{question}
Does the opposite to the second part of the compactification principle hold? That is, is it true that all Landau--Ginzburg models
(in the sense of Homological Mirror Symmetry) of the same dimension as the initial Fano variety are compactifications of toric ones? In particular, is it true that all of them are rational?
\end{question}

\begin{question}
Does the Optimistic Picture~\ref{optimisticpicture} need any extra condition on toric varieties?
\end{question}

\part{Del Pezzo surfaces}
\label{part: del Pezzo surfaces}

We start the section by recalling well known facts about del Pezzo surfaces. We refer, say, to~\cite{Do12} as to one of huge amount
of references on del Pezzo surfaces.

The initial definition of del Pezzo surface is the following one given by P.\,del Pezzo himself.

\begin{definition}[\cite{dP87}]
\emph{A del Pezzo surface} is a non-degenerate (that is not lying in a linear subspace) irreducible linear normal (that is it is not a projection of degree $d$ surface in $\PP^{d+1}$) surface in $\PP^d$ of degree $d$ which is not a cone.
\end{definition}

In modern words this means that a del Pezzo surface is an (anticanonically embedded) surface with ample anticanonical class
and canonical (the same as du Val, simple surface, Kleinian, or rational double point) singularities. (Classes of canonical and
Gorenstein singularities for surfaces coincide.)
So we use the following more general definition.

\begin{definition}
\emph{A del Pezzo surface} is a complete surface with ample anticanonical class and canonical singularities.
\emph{A weak del Pezzo surface} is a complete surface with nef and big anticanonical class and canonical singularities.
\end{definition}

\begin{remark}
Weak del Pezzo surfaces are (partial) minimal resolutions of singularities of del Pezzo surfaces. Exceptional
divisors of the resolutions are $(-2)$-curves.
\end{remark}

\emph{A degree} of del Pezzo surface $S$ is the number $d=(-K_S)^2$. One has $1\leq d \leq 9$.
If $d>2$, then the anticanonical class of $S$ is very ample and it gives the embedding $S\hookrightarrow \PP^d$, so both definitions
coincide.
In this section we assume that $d>2$.

Obviously, projecting a degree $d$ surface in $\PP^d$ from a smooth point on it one gets degree $d-1$ surface in $\PP^{d-1}$.
This projection is nothing but blow up of the center of the projection and blow down all lines passing through the
point. (By adjunction formula these lines are $(-2)$-curves.) If we choose general (say, not lying on lines) centers of projections we get a classical description
of a smooth del Pezzo surface of degree $d$ as a quadric surface (with $d=8$) or a blow up of $\PP^2$ in $9-d$ points.
They degenerate to singular surfaces which are projections from non-general points (including infinitely close ones).
Moreover, all del Pezzo surfaces of given degree lie in the same irreducible deformation space except for degree $8$ when there are two
components (one for a quadric surface and one for a blow up $\FF_1$ of $\PP^2$). General elements of the families are smooth,
and all singular del Pezzo surfaces are degenerations of smooth ones in these families.
This description enables us to construct toric degenerations of del Pezzo surfaces.
That is, $\PP^2$ is toric itself.
Projecting from toric points one gets a (possibly singular) toric del Pezzo surfaces.

\begin{remark}
The approach to description of del Pezzo surfaces via their toric degenerations and the connection of the degenerations by elementary transformations (projections) can be generalized to the threefold case. That is, smooth Fano threefolds can be connected via toric degenerations
and \emph{toric basic links}. For details see~\cite{CKP13}.
\end{remark}

\begin{remark}
\label{remark: del Pezzo 1 and 2}
Del Pezzo surfaces of degree $1$ or $2$ also have toric degenerations.
Indeed, these surfaces can be described as hypersurfaces in weighted projective spaces, that is ones
of degree $4$ in $\PP(1,1,1,2)$ and of degree $6$ in $\PP(1,1,2,3)$ correspondingly,
so they can be degenerated to binomial hypersurfaces, cf. Example~\ref{ex:ci}.
However their singularities are worse then canonical.
\end{remark}

Let $T_S$ be a Gorenstein toric degeneration of a del Pezzo surface $S$ of degree $d$. Let $\Delta=F(T_S)\subset \mathcal N_\RR=\ZZ^2\otimes \RR$ be a fan polygon
of $T_S$. Let $f$ be a Laurent polynomial such that $N(f)=\Delta$.

Our goal now is to describe in details a way to construct a Calabi--Yau compactification for $f$.
More precise, we construct a commutative diagram
\[\xymatrix{
(\CC^*)^2\ar@{^{(}->}[r] \ar[rd]^f& Y\ar[d] \ar@{^{(}->}[r] & Z\ar[d]\\
& \Aff^1 \ar@{^{(}->}[r] & \PP^1,\\
}
\]
where $Y$ and $Z$ are smooth, fibers of maps $Y\to \Aff^1$ and $Z\to \PP^1$ are compact, and $-K_Z=f^{-1}(\infty)$;
we denote all ``vertical'' maps in the diagram by $f$ for simplicity.

The strategy is the following. First we consider a natural compactification of the pencil $\{f=\lambda\}$ to an elliptic pencil
in a toric del Pezzo surface $T^\vee$. Then we resolve singularities of $T^\vee$ and get a pencil in a smooth toric weak del Pezzo surface
$\widetilde{T}^\vee$. Finally we resolve a base locus of the pencil to get $Z$. We get $Y$ cutting out strict transform of the boundary divisor of $\widetilde{T}^\vee$.

The polygon $\Delta$ %is \emph{reflexive}, that is an polygon with
has integral vertices in
$\mathcal N_\RR$ and it has the origin as a unique strictly internal integral point.
A dual polygon $\nabla=\Delta^\vee \subset \mathcal M=\mathcal N^\vee$ has integral vertices and a unique strictly internal integral point as well.
Geometrically this means that singularities of $T$ and $T^\vee$ are canonical. %, or, the same, Gorenstein.

\begin{remark}
The normalized volume of $\nabla$ is given by
$$\mathrm{vol}\, \nabla= |\mbox{integral points in $\nabla$}|-1=(-K_{S})^2=d.$$
It is easy to see %(cf. Figure~\ref{figure:del Pezzo})
that $$|\mbox{integral points on the boundary of $\Delta$}|+|\mbox{integral points on the boundary of $\nabla$}|=12.$$
In particular, $\mathrm{vol}\, \Delta=12-d$.
\end{remark}

\begin{compactificationconstruction}[\cite{Prz17}]
\label{compactification construction}
By Fact~\ref{fact: toric 2}, the anticanonical linear system on $T^\vee$ can be described as a projectivisation of a linear space of Laurent polynomials whose
Newton polygons are contained in $\nabla^\vee=\Delta$. Thus the natural way to compactify the family is to do it using
embedding $(\CC^*)^2\hookrightarrow T^\vee$. Fibers of the family are anticanonical divisors in this (possibly singular)
toric variety. Two anticanonical sections intersect by $(-K_{T^\vee})^2= \mathrm{vol}\, \Delta=12-d$ points (counted with multiplicities),
so the compactification of the pencil in $T^\vee$ has $12-d$ base points (possibly with multiplicities).
The pencil $\{\lambda_0 f=\lambda_1\}$, $(\lambda_0:\lambda_1)\in \PP$,
is generated by its general member and a divisor corresponding to a constant Laurent polynomial, i.\,e. to the boundary
divisor of $T^\vee$. Let us mention that the torus invariant points of $T^\vee$ do not lie in the base locus of the family
by Fact~\ref{fact: toric 4}.

Let $\widetilde{T}^\vee\to T^\vee$ be a minimal resolution of singularities of $T^\vee$.
Pull back the pencil under consideration. We get an elliptic pencil with $12-d$ base points (with multiplicities),
which are smooth points of %one of members of the family, that is
the boundary divisor $D$ %=(D_1,\ldots,D_d)$
of the toric surface $\widetilde{T}^\vee$; this divisor is a wheel of $d$ smooth rational curves.
Blow up these base points and get an elliptic surface $Z$.
Let $E_1,\ldots,E_{12-d}$ be the exceptional curves of the blow up $\pi\colon Z\to \widetilde{T}^\vee$; in particular, $Z$ is not toric. Denote a strict transform
of $D$ by $D$ for simplicity.
Then one has
$$
-K_Z=\pi^*(-K_{\widetilde{T}^\vee})-\sum E_i=D+\sum E_i-\sum E_i=D.
$$
Thus the anticanonical class $-K_Z$ contains $D$ and consists of fibers of $Z$.
This, in particular, means that an open variety $Y=Z\setminus D$ is a Calabi--Yau compactification of the pencil provided by $f$.
This variety has $e>0$ sections, where $e$ is a number of base points of the pencil in $\widetilde{T}^\vee$ counted \emph{without} multiplicities.

Summarizing, we obtain an elliptic surface $f\colon Z\to \PP^1$ with smooth total space $Z$ and a wheel $D$ of $d$ smooth rational curves over $\infty$.
%and with a section.
\end{compactificationconstruction}

\begin{remark}
\label{remark: Hodge numbers for Z}
Let the polynomial $f$ be general among ones with the same Newton polygon. Then singular fibers of $Z\to \PP^1$ are either curves
with a single node or a wheel of $d$ rational curves over $\infty$. By Noether formula one has
$$
12\chi (\cO_Z)=(-K_Z)^2+e(Z)=e(Z),
$$
where $e(Z)$ is a topological Euler characteristic. Thus singular fibers for $Z\to \PP^1$
are $d$ curves with one node and a wheel of $d$ curves over $\infty$. This description is given in~\cite{AKO06}.
%\end{itemize}
\end{remark}

\begin{remark}
One can compactify all toric Landau--Ginzburg models for all del Pezzo surfaces of degree at least three simultaneously.
That is, all reflexive polygons are contained in the biggest polygon $B$, that has vertices $(2,-1)$, $(-1,2)$, $(-1,-1)$. Thus fibers of all toric Landau--Ginzburg models can be simultaneously compactified to (possibly
singular) anticanonical curves on $T_{B^\vee}=\PP^2$. Blow up the base locus to construct a base points free family. However in this
case a general member of the family can pass through toric points as it can happen that $N(f)\varsubsetneq B$.
This means that some of exceptional divisors of the minimal resolution are %not sections of the elliptic pencil but
extra curves in a wheel over infinity.

In other words, consider a triangle of lines %and an elliptic curve
on $\PP^2$.
A general member of the pencil given by $f$ is an elliptic curve on $\PP^2$.
The total space of the log Calabi--Yau compactification %(together with a fiber over infinity)
is a blow up of nine intersection points (counted with multiplicities) of the elliptic curve and the triangle of lines.
Exceptional divisors for points lying over vertices of the triangle are components of the wheel over infinity for the log
Calabi--Yau compactification; the others are either sections of the pencil or components of fibers over finite points.
\end{remark}

Now following~\cite{Prz17} describe toric Landau--Ginzburg models for del Pezzo surfaces and toric weak del Pezzo surfaces.
That is, for a del Pezzo surface $S$, its
Gorenstein toric degeneration $T$ with a fan polygon $\Delta$,
its crepant resolution $\widetilde{T}$ with the same fan polygon, and a divisor $D\in \Pic(S)_\CC\simeq \Pic(\widetilde T)_\CC$,
we construct two Laurent polynomials $f_{S,D}$ and $f_{\widetilde T,D}$, that are toric Landau--Ginzburg models
for $S$ and $\widetilde T$ correspondingly, by induction. For this use, in particular, Givental's construction of Landau--Ginzburg
models for smooth toric varieties, see Section~\ref{section: complete Givental}.

Let $S\simeq \PP^1\times \PP^1$ be a quadric surface, and let $D_S$ be an $(a,b)$-divisor on it.
Let $T_1=S$,
and let $T_2$ be a quadratic cone; $T_1$ and $T_2$ are the only Gorenstein toric degenerations of $S$.
The crepant resolution of $T_2$ is a Hirzebruch surface $\FF_2$, so let $D_{\FF_2}=\alpha s+\beta f$,
where $s$ is a section of $\FF_2$, so that $s^2=-2$, and $f$ is a fiber of the map $\FF_2\to \PP^1$.
Define
$$
f_{S,D_S}=f_{\widetilde T_1,D_S}=x+\frac{e^{-a}}{x}+y+\frac{e^{-b}}{y}
$$
for the first toric degeneration and
$$
f_{S,D_S}=y+e^{-a}\frac{1}{xy}+\left(e^{-a}+e^{-b}\right)\frac{1}{y}+e^{-b}\frac{x}{y},\ \ f_{\widetilde T_2,D_{\FF_2}}=y+\frac{e^{-\beta}}{xy}
+\frac{e^{-\alpha}}{y}+\frac{x}{y}
$$
for the second one.

Now assume that $S$ is a blow up of $\PP^2$.
First let $S=T=\widetilde T=\PP^2$, let $l$ be a class of a line on $S$, and let $D=a_0l$.
Then up to a toric change of variables one has
$$
f_{\PP^2,D}=x+y+\frac{e^{-a_0}}{xy}.
$$
Now let $S'$ be a blow up of $\PP^2$ in $k$ points with exceptional
divisors $e_1, \ldots, e_k$, let $S$ be a blow up of $S'$ in a point,
and let $e_{k+1}$ be an exceptional divisor for the blow up.
We identify divisors on $S'$ and their strict transforms on $S$,
so $\Pic(S')=\Pic (\widetilde T')=\ZZ l+\ZZ{e_1}+\ldots+\ZZ{e_k}$ and
$\Pic(S)=\ZZ l+\ZZ{e_1}+\ldots+\ZZ{e_k}+\ZZ{e_{k+1}}$.
Let $D'=a_0l+a_1e_1+\ldots+a_ke_k\in \Pic(S')_\CC$
and $D=D'+a_{k+1}e_{k+1}\in \Pic(S)_\CC$.
First describe the polynomial $f_{\widetilde T,D}$.
Combinatorially $\Delta=F(\widetilde T)$
is obtained from a polygon $\Delta'=F(\widetilde T')$ by adding
one integral point $K$ that corresponds to the exceptional
divisor $e_{k+1}$, and taking a convex hull. Let
$L$, $R$ be boundary points of $\Delta$ neighbor to $K$, left and right with respect to the clockwise order.
Let $c_L$ and $c_R$ be coefficients in $f_{\widetilde T',D'}$
at monomials corresponding to $L$ and $R$. Let $M\in \TT[x,y]$ be a monomial corresponding to $K$.
Then from Givental's description of Landau--Ginzburg models for toric varieties (see Section~\ref{section: complete Givental}) one gets
$$
f_{\widetilde T,D}=f_{\widetilde S',D'}+c_Lc_Re^{-a_{k+1}}M.
$$

The polynomial $f_{S,D}$ differs from $f_{\widetilde T,D}$ by coefficients at non-vertex boundary points.
For any boundary point $K\in \Delta$ define \emph{marking} $m_K$ as a coefficient of $f_{\widetilde T,D}$ at $K$.
Consider a facet of $\Delta$ and let $K_0,\ldots, K_r$ be integral points in clockwise order of this facet.
Then coefficient of $f_{S,D}$ at $K_i$ is a coefficient at $s^i$ in the polynomial
$$
m_{K_0}\left(1+\frac{m_{K_1}}{m_{K_0}}s\right)\cdot\ldots\cdot \left(1+\frac{m_{K_r}}{m_{K_{r-1}}}s\right).
$$

\begin{remark}
One has $\Pic(S)\simeq \Pic(\widetilde T)$.
That is, if $S$ is not a quadric, then both $S$ and $\widetilde T$
are obtained by a sequence of blow ups in points (the only difference is
that the points for $\widetilde T$ can lie on exceptional divisors of previous blow ups).
Thus in both cases Picard groups are generated by a class of a line on $\PP^2$ and exceptional divisors
$e_1,\ldots, e_k$. However an image of $e_i$ under the map of Picard groups
given by the degeneration of $S$ to $\widetilde T$ can be not equal to $e_i$ itself
but to some linear combination of the exceptional divisors.
In other words these bases do not agree with degenerations.
\end{remark}

\begin{remark}
The spaces parameterizing toric Landau--Ginzburg models for $S$ and for $\widetilde T$ are the same ---
they are the spaces of Laurent polynomials with Newton polygon $\Delta$ modulo toric rescaling.
Thus any Laurent polynomial corresponds to different elements of $\Pic(S)_\CC\simeq \Pic(\widetilde T)_\CC$.
This gives a map $\Pic(S)_{\CC}\to \Pic(\widetilde T)_{\CC}$. However this
map is transcendental because of exponential nature of the parametrization.
\end{remark}

\begin{proposition}[{\cite[Proposition 21]{Prz17}}]
\label{proposition: toric LG for del Pezzo}
The Laurent polynomial $f_{S,D}$ is a toric Landau--Ginzburg model for $(S,D)$.
\end{proposition}

\begin{proof}
It is well known that $S$ is either a smooth toric variety
or a complete intersection in a smooth toric variety.
This enables one to compute a series $\widetilde I^S$ and, since, $\widetilde I^{S,D}$ following~\cite{Gi97b}, see Theorem~\ref{theorem: Givental I-series toric}.
Using this it is straightforward to check that the period condition for $f_{S,D}$ holds.
The Calabi--Yau condition holds by Compactification Construction~\ref{compactification construction}. %from Section~\ref{subsection:LG for DP}.
Finally the toric condition holds by construction.
(See Example~\ref{example: S7}.)
\end{proof}

\begin{proposition}[{\cite[Proposition 22]{Prz17}}]
\label{proposition:mutations}
Consider two different Gorenstein toric degenerations $T_1$ and $T_2$ of a del Pezzo surface $S$. Let $\Delta_1=F(T_1)$ and $\Delta_2=F(T_2)$.
Consider families of Calabi--Yau compactifications of Laurent polynomials
with Newton polygons $\Delta_1$ and $\Delta_2$. Then there is a birational isomorphism of these families.
In other words, there is a birational isomorphism between affine spaces of Laurent polynomials
with supports in $\Delta_1$ and $\Delta_2$ modulo toric change of variables that preserves Calabi--Yau compactifications.
\end{proposition}

\begin{proof}
One can check that polygons $\Delta_1$ and $\Delta_2$ differ by (a sequence of) mutations (see, say,~\cite{ACGK12}).
These mutations agree with fiberwise birational isomorphisms of toric Landau--Ginzburg models
modulo change of basis in $H^2(S,\ZZ)$ by the construction. The statement follows from the fact that birational elliptic curves are isomorphic.
\end{proof}

\begin{remark}
\label{remark:canonical_the_same}
Let $D=0$. Then
%Given a polytope $\Delta$ let us call a polynomial $f_\Delta$ that have
the polynomial $f_{S,0}$ has
coefficients $1$ at vertices of its Newton polygon and $\binom{n}{k}$ at $k$-th
integral point of an edge of integral length $n$.
In other words, $f_{S,0}$ is binomial, cf. Section~\ref{section: weak LG 3-folds}. %it satisfies the \emph{binomial principle}, see~\cite{Prz13}.
\end{remark}

\begin{example}
\label{example: S7}
Let $S=S_7$. This surface has two Gorenstein toric degenerations: it is toric itself, and also it can be degenerated to
a singular surface which is obtained by a blow up of $\PP^2$, a blow up of a point on the exceptional curve, and a blow down
the first exceptional curve to a point of type $\sA_1$.

Let $\Delta_1$ be the polygon with vertices $(1,0)$, $(1,1)$, $(0,1)$, $(-1,-1)$, $(0,-1)$,
%rightmost polytope on the third line of Figure~\ref{figure:del Pezzo},
and let $D=a_0l+a_1e_1+a_2e_2$.
Then
$$
f_{\widetilde T_{\Delta_1},D}=f_{S,D}=x+y+e^{-a_0}\frac{1}{xy}+e^{-(a_0+a_1)}\frac{1}{y}+e^{-a_2}xy.
$$

Let $\Delta_2$ be the polygon with vertices $(1,0)$, $(0,1)$, $(-1,-1)$, and $(1,-1)$,
and let $\mbox{$D=a_0l+a_1e_1+a_2e_2$}$.
Then
$$
f_{\widetilde T_{\Delta_2},D}=x+y+e^{-a_0}\frac{1}{xy}+e^{-(a_0+a_1)}\frac{1}{y}+e^{-(a_0+a_1+a_2)}\frac{x}{y},
$$
$$
f'_{S,D}=x+y+e^{-a_0}\frac{1}{xy}+\left(e^{-(a_0+a_1)}+e^{-(a_0+a_2)}\right)\frac{1}{y}+e^{-(a_0+a_1+a_2)}\frac{x}{y}.
$$
(Here $f_{S,D}$ and $f'_{S,D}$ are toric Landau--Ginzburg models for $(S,D)$ in different bases in $(\CC^*)^2$.)
One can easily check that the mutation
$$
x\to x,\ \ \ y\to \frac{y}{1+e^{-a_2}x}
$$
sends $f_{S,D}$ to $f'_{S,D}$.

The surface $S$ is toric, so by Givental
$$
\widetilde I_0^{S,D}= \sum_{k,l,m}  \frac{(2k+3l+2m)!e^{-a_0(k+l+m)-a_1k-a_2m}t^{2k+3l+2m}}{(k+l)!(l+m)!k!l!m!}
$$
(see~\cite{fanosearch}).
One can check that $\widetilde I_0^{S,D}=I_{f_{S,D}}=I_{f'_{S,D}}$.
\end{example}

\part{Fano threefolds}
\label{part:Fano 3-folds}

This part is devoted to the most studied case of toric Landau--Ginzburg models --- that is, models for Fano threefolds.
We mainly focus on the Picard rank one case.

\section{Weak Landau--Ginzburg models}
\label{section: weak LG 3-folds}

Consider a smooth Fano threefold $X$ of Picard rank $\rho$ and a divisor $D$ on it.
Recall that we associate with them the regularized series $\widetilde I^{X,D}$ and, in particular, the constant term of this series $\widetilde I^{X,D}_0$.
These series are given by an intersection theory (of curves and divisors) on $X$ and the series $\widetilde I^X=\widetilde I^{X,0}$,
see the beginning of Section~\ref{part:toric LG}, or even the series $I^X_0$, see~\cite{Prz08a}.
Coefficients of the series depend on the even part of cohomology of $X$,
which is quite simple: $H^0(X,\ZZ)=H^6 (X,\ZZ)=\ZZ$, $H^2(X,\ZZ)=H^4(X,\ZZ)=\ZZ^\rho$.
The relations on Gromov--Witten invariants shows that the coefficients are given by finite (small)
number of three-pointed Gromov--Witten invariants, see details, say, in~\cite{Prz08a}.
In the case of Picard rank one Fano threefolds these three-pointed invariants
were found, using Givental, Fulton--Woodward, and others results, see~\cite{Prz07a},~\cite{Prz07b} and references therein.
Theorems~\ref{theorem: Givental I-series toric} and~\ref{theorem: Grass I-series} enable one to compute $\widetilde I^X_0$
for complete intersections in smooth toric varieties and Grassmannians. Fano threefolds with $\rho>1$
have descriptions of this type, and the corresponding series $\widetilde I^X_0$ are computed in~\cite{CCGK16}, see also~\cite{fanosearch}.
We need the series $\widetilde I^X_0$ unless otherwise stated,
so we do not need the intersection theory on $X$. From now on we assume the series $\widetilde I^X_0$ known.

We assume that $D=0$.
Recall that a Laurent polynomial $f_X$ is called a weak Landau--Ginzburg model for $X$ if it satisfies the period condition,
that is if its main period (see Definition~\ref{definition: main integral}) coincides with $I^X_0$.
There are $105$ families of smooth Fano threefolds, see, for instance,~\cite{IP99} and~\cite{MM82}.
Their anticanonical classes are very ample for $98$ of them.
Weak Landau--Ginzburg models are known for each of them (they are usually not unique), see~\cite{fanosearch} for the case of very ample anticanonical class and Proposition~\ref{proposition:CY hyperelliptic-Laurent} for the other case.

Smooth Fano threefolds with non-very ample anticanonical classes can be described as complete intersections in smooth toric varieties
or weighted projective spaces, so one can construct Givental's Landau--Ginzburg models (see Definition~\ref{definition: Givental LG})
satisfying the period condition.

\begin{proposition}[cf. Proposition~\ref{proposition:CY hyperelliptic-CY}]
\label{proposition:CY hyperelliptic-Laurent}
Fano threefolds $X_{1-1}$, $X_{1-11}$, $X_{2-1}$, $X_{2-2}$, $X_{2-3}$, $X_{9-1}$, and $X_{10-1}$
have toric Landau--Ginzburg models.
\end{proposition}

\begin{proof}
The Fano variety $X_{1-1}$  is a hypersurface section of degree $6$ in $\PP(1,1,1,1,3)$.
The Fano variety $X_{1-11}$ is a hypersurface section of degree $6$ in $\PP(1,1,1,2,3)$.
The Fano variety $X_{2-1}$ is a hypersurface section of type $(1,1)$ in $\PP^1\times X_{1-11}$ in an anticanonical embedding;
in other words, it is a complete intersection of hypersurfaces of types $(1,1)$ and $(0,6)$ in $\PP^1\times \PP(1,1,1,2,3)$.
The Fano variety $X_{2-2}$ is a hypersurface in a certain toric variety, see~\cite{CCGK16}.
The Fano variety $X_{2-3}$ is a hyperplane section of type $(1,1)$ in $\PP^1\times X_{1-12}$ in an anticanonical embedding;
in other words, it is a complete intersection of hypersurfaces of types $(1,1)$ and $(0,4)$ in $\PP^1\times \PP(1,1,1,1,2)$.
Finally one has $X_{9-1}=\PP^1\times S_2$ and $X_{10-1}=\PP^1\times S_1$.

For a variety $X_{i-j}$ construct its Givental's type Landau--Ginzburg models and then present it by Laurent polynomial $f_{i-j}$, see, for instance, formula~\eqref{formula: CI LG}.
It satisfies the period condition by~\cite{fanosearch}
Consider these cases one by one.

Givental's Landau--Ginzburg model for $X_{2-1}$ is
a complete intersection
$$
\left\{
  \begin{array}{ll}
    u+v_0=0,  \\
    v_1+v_2+v_3=0
  \end{array}
\right.
$$
in $\TTT[u,v_0,v_1,v_2,v_3]$
with a function
$$
u+\frac{1}{u}+v_0+v_1+v_2+v_3+\frac{1}{v_1v_2^2v_3^3},
$$
see Definition~\ref{definition: Givental LG}.
After the birational change of variables
$$
v_1=\frac{x}{x+y+1},\ \ v_2=\frac{y}{x+y+1},\ \ v_3=\frac{1}{x+y+1},\ \ u=\frac{z}{z+1},\ \ v_0=\frac{1}{z+1}
$$
one, up to an additive shift, gets a function
$$
f_{2-1}=\frac{(x+y+1)^6(z+1)}{xy^2}+\frac{1}{z}
$$
on a torus $\TTT[x,y,z]$.

In the similar way one gets Calabi--Yau compactifications for the other varieties.
One has
$$
f_{1-1}=\frac{(x+y+z+1)^6}{xyz},
$$
$$
f_{1-11}=\frac{(x+y+1)^6}{xy^2z}+z,
$$
$$
f_{2-2}=\frac{(x+y+z+1)^2}{x}+\frac{(x+y+z+1)^4}{yz},
$$
$$
f_{2-3}=\frac{(x+y+1)^4(z+1)}{xyz}+z+1,
$$
$$
f_{9-1}=x+\frac{1}{x}+\frac{(y+z+1)^4}{yz}.
$$
\begin{equation*}
f_{10-1}=\frac{(x+y+1)^6}{xy^2}+z+\frac{1}{z}.
\qedhere
\end{equation*}
\end{proof}

Thus one can assume that the anticanonical class of $X$ is very ample. To find a weak Landau--Ginzburg model for $X$ one can,
similarly to Proposition~\ref{proposition:CY hyperelliptic-Laurent}, construct Givental's Landau--Ginzburg model
and try to find birational isomorphisms of total spaces of such models with an algebraic torus (cf. Theorem~\ref{theorem: CI in Grass Laurent}).
However we use another approach. That is, we consider ``good'' three-dimensional polytopes and study ``correct''
Laurent polynomials supported on them (in particular, their coefficients are symmetric enough).
At the moment the most appropriate method of constructing weak Landau--Ginzburg models seems to be a generalization
of the one described below, see Remark~\ref{remark: maximally mutational}. However we do not need it here.

Weak Landau--Ginzburg models, ``guessed'' via the period condition (see~\cite{Prz08b}) or obtained from Landau--Ginzburg models
for complete intersections often first have reflexive Newton polytopes and, second, they often satisfy the \emph{binomial principle},
see~\cite{Prz13}.
It declares a way to put coefficients of Laurent polynomials with fixed Newton polytopes.
That is, one needs to put 1's at vertices of such polytope,
and $\binom{n}{i}$ on $i$-th (from any end) integral point of an edge of integral length $n$.
This principle can be applied in many cases (in other words, for Newton polytopes of toric varieties with cDV singularities,
that is ones whose integral points, except for the origin, lie on edges).
Most of smooth Fano threefolds have degenerations to toric varieties with cDV singularities,
but not all unfortunately.
Thus we use the following generalization of the binomial principle.

\begin{definition}[see~\cite{CCGK16}]
\label{definition: Minkowski}
An integral polygon is called \emph{of type $A_n$}, $n\geq 0$, if it is a triangle such that two its edges have integral length $1$ and
the rest one has integral length $n$.
(In other words, its integral points are $3$ vertices and $n-1$ points lying on the same edge.)
In particular, $A_0$ is a segment of integral length $1$.

An integral polygon $P$ is called \emph{Minkowski}, or \emph{of Minkowski type}, if it is a Minkowski sum of several polygons of type $A_n$,
that is
$$
P=\left\{p_1+\ldots +p_k |\, p_i\in P_i \right\}
$$
for some polygons $P_i$ of type $A_{k_i}$, and if the affine lattice generated by $P\cap \ZZ^2$ is a sum of affine lattices
generated by $P_i\cap \ZZ^2$. Such decomposition is called \emph{admissible lattice Minkowski decomposition} and
denoted by $P=P_1+\ldots +P_k$.

An integral three-dimensional polytope is called \emph{Minkowski} if it is reflexive and if all its facets
are Minkowski polygons.
\end{definition}

\begin{definition}[see~\cite{CCGK16}]
Let $P\in \ZZ^2\otimes \RR$ be an integral polygon of type $A_n$. Let $v_0,\ldots,v_n$ be consecutive integral points
on the edge of $P$ of integral length $n$ and let $u$ be the rest integral point of $P$.
Let $x=(x_1,x_2)$ be a multivariable that corresponds to an integral lattice $\ZZ^2\subset \RR^2$.
Put
$$
f_P=x^u+\sum \binom{n}{k} x^{v_k}.
$$
(In particular one has $f_P=x^u+x^{v_0}$ for $n=0$.)

Let $Q=Q_1+\ldots+Q_s$ be an admissible lattice Minkowski decomposition of an integral polygon $Q\subset \RR^2$.
Put
$$
f_{Q_1,\ldots,Q_s}=f_{Q_1}\cdot\ldots\cdot f_{Q_s}.
$$

A Laurent polynomial $f\subset \TT[x_1,x_2,x_3]$ is called \emph{Minkowski} if $N(f)$ is Minkowski and
for any facet $Q\subset N(f)$, %after consideration of it
as for an integral polygon, there exist an admissible lattice Minkowski decomposition $Q=Q_1+\ldots+Q_s$ such that $f|_Q=f_{Q_1,\ldots,Q_s}$.
\end{definition}

All $98$ families of smooth Fano threefolds that have very ample anticanonical classes have weak Landau--Ginzburg models of Minkowski type, see~\cite{CCGK16}.

\begin{remark}
\label{remark: maximally mutational}
There exist the following notion of \emph{maximally mutable polynomial}, see, for instance,~\cite{KT}).
A birational isomorphism $\phi\colon \TTT[x_1,\ldots,x_n]\to \TTT[y_1,\ldots,y_n]$ is called elementary mutation
of the polynomials
$f$ and $g$ if it is given by $y_1=r(x_1,\ldots,x_n)$, $y_i=x_i$ for $2\leq i\leq n$,
and $\phi^*(f)=g$.
The Laurent polynomials $f$ and $g$ in $n$ variables are called mutationally equivalent if there exists a sequence
of mutations transforming one to another.
On the other hand, if we have a polytope $\Delta$ and a vector $v$ in the dual space, one can define a mutation of $\Delta$ in $v$
(if it exists) multiplying $k$-slice for $v$ (that is a set $\{p\in \Delta\mid \langle p,n\rangle=k\}$)
on $k$-th power of some fixed polytope $\Delta_v$
(and dividing on it for $k<0$).
(Mutations of polytopes correspond to deformations of toric Fano varieties, whose fan polytopes they are, see~\cite{IV12}.)
It is clear that mutations of polynomials induce mutations of their Newton polytopes.
However the opposite is not true in general. There are strong restrictions to make the opposite true.
A Laurent polynomial is called maximally mutational if any mutation of its Newton polytope
is given by mutation of the polynomial and if it is true for all mutations of the polynomial.
Rigid (without parameters) maximally mutational Laurent polynomials form a class of weak Landau--Ginzburg
models fits now in the best way to the list of Fano varieties.
In dimension $2$ there are exactly $10$ such classes, and elements of each of them are weak Landau--Ginzburg
models for all of $10$ families of del Pezzo surfaces.
In dimension $3$ there are $105$ classes,
and each of them correspond to one of $105$ families of Fano threefolds (private communication with A.\,Kasprzyk).
\end{remark}

\begin{remark}
\label{remark:Minkowski_equivalence}
Minkowski decompositions of facets of Newton polytopes of Laurent polynomials of Minkowski type
naturally give mutations of the polynomials. It turns out (see~\cite{CCGK16}) that Minkowski type polynomials
are mutationally equivalent if and only if they have the same constant term series (and, thus, they are weak Landau--Ginzburg models
of the same Fano variety if such variety exists). Classes of Laurent polynomials of Minkowski type
which do not correspond to smooth Fano threefolds are expected to correspond to smooth orbifolds.
\end{remark}

\section{Calabi--Yau compactifications}
\label{section: Calabi--Yau 3-folds}

Let $f$ be a weak Landau--Ginzburg model for a smooth Fano threefold $X$ and a divisor $D$ on it.
Recall the notation from Section~\ref{section:toric}.
Let $\Delta=N(f)$, $\nabla=\Delta^\vee$, $T=T_\Delta$, $T^\vee=T_\nabla$.
In a lot of cases
polynomials satisfying the period and toric conditions satisfy
the Calabi--Yau condition as well. However it is not easy to check this condition:
there are no general enough approaches as for the rest two conditions are; usually one needs to
check the Calabi--Yau condition ``by hand''. The natural idea is to compactify the fibers
of the map $f\colon (\CC^*)^n\to \CC$ using the embedding $(\CC^*)^n\hookrightarrow T^\vee$.
Indeed, the fibers compactify to anticanonical sections in $T^\vee$, and, since, have trivial canonical classes. However, first,
$T^\vee$ is usually singular, and, even if we resolve it (if it has a crepant resolution!),
we can just conclude that its general anticanonical section is a smooth Calabi--Yau variety,
but it is hard to say anything about the particular sections we need.
Second, the family of anticanonical sections we are interested in has a base locus which
we need to resolve to construct a Calabi--Yau compactification; and this resolution can be non-crepant.

Coefficients of the polynomials that correspond to trivial divisors tend to have very symmetric
coefficients, at least for the simplest toric degenerations. In this case the base loci
are more simple and enable us to construct Calabi--Yau compactifications.

We will assume $f$ to be of Minkowski type. In particular, $\nabla$ is integral, in other words, $\Delta$ is reflexive,
and integral points of both $\Delta$ and $\nabla$ are either boundary points or the origins.

\begin{lemma}
\label{lemma: 3dim smoothness}
Let $T$ be a threefold reflexive toric variety. Then $\widetilde T^\vee$ is smooth.
\end{lemma}

\begin{proof}
Let $C$ be a two-dimensional cone of the fan of $\widetilde T^\vee$. It is a cone over an integral triangle $R$ without strictly internal integral points, such that $R$ lies in the affine plane $L=\{x|\ \langle x,y\rangle\geq -1\}$ for some $y\in \mathcal N$.
This means that is some basis $e_1,e_2,e_3$ in $\mathcal M$ one
gets $L=\{a_1e_1+a_2e_2+e_3\}$. Let $P$ be a pyramid over $R$ whose vertex is an origin. Then by Pick's formula
one has $\vol R=\frac{1}{2}$ and $\vol P=\frac{1}{6}$, which means that vertices of $R$ form basis in $\mathcal M$, so $\widetilde T^\vee$ is smooth.
\end{proof}

Unfortunately, Lemma~\ref{lemma: 3dim smoothness} does not hold for higher dimensions in general, because there are $n$-dimensional
simplices whose only integral points are vertices, such that their volumes are greater then $\frac{1}{n!}$.

\begin{lemma}[{\cite[Lemma 25]{Prz17}}]
\label{lemma:restrictions Minkowski}
Let $f$ be a Laurent polynomial of Minkowski type.
Then for any facet $Q$ of $\Delta$ the curve $R_{Q,f}$ is a union of (transversally intersecting) smooth rational curves (possibly with multiplicities).
\end{lemma}

\begin{proof}[Idea of the proof]
For non-Minkowski decomposable case this follows from Facts~\ref{fact: toric 1} and~\ref{fact: toric 3}.
In the decomposable case the curve $R_{Q,f}$ is a union of curves that correspond to Minkowski summands of $Q$.
\end{proof}

\begin{proposition}
\label{proposition:CY compactification}
Let $W$ be a smooth threefold. Let $F$ be a one-dimensional anticanonical linear system on $W$
with reduced fiber $D=F_\infty$.
Let a base locus $B\subset D$ be a union of smooth curves (possibly with multiplicities) such that for any two components
$D_1,D_2$ of $D$ one has $D_1\cap D_2\not\subset B$. Then there is a resolution of the base locus $f\colon Z\to \PP^1$
with a smooth total space $Z$ such that
$-K_Z=f^{-1}(\infty)$.
\end{proposition}

\begin{proof}[Proof (cf. Compactification Construction~\ref{compactification construction}).]
Let $\pi\colon W'\to W$ be a blow up of one component $C$ of $B$ on $W$.
Since $\pi$ is a blow up of a smooth curve on a smooth variety, $W'$ is smooth.
Let $E$ be an exceptional divisor of the blow up.
Let $D'=\cup D'_i$ be a proper transform of $D=\cup D_i$.
Since the multiplicity of $C$ in $D$ is $1$, one gets
$$
-K_{W'}=\pi^*(-K_{W})-E=D'+E-E=D'.
$$
Moreover, a base locus of the family on $W'$ is the same as $B$ or $B\setminus C$,
possibly together with a smooth curve $C'$ which is isomorphic to
$E\cap D'_i$;
in particular, $C$ is isomorphic to $\PP^1$.
(There are no isolated base points as the base locus is an intersection of two divisors on a smooth variety.)
Thus all conditions of the proposition hold for $W'$. Since $(W,F)$ is a canonical pair,
the base locus $B$ can be resolved in finite number of blow ups. This gives the required resolution.%~$\varphi$.
\end{proof}

\begin{theorem}
\label{theorem: Minkowski CY}
Any Minkowski Laurent polynomial in three variables admits a log Calabi--Yau compactification.
\end{theorem}

\begin{proof}
Let $f$ be a Minkowski Laurent polynomial.
Recall that the Newton polytope $\Delta$ of $f$ is reflexive,
and the (singular Fano) toric variety whose fan polytope is $\nabla=\Delta^\vee$ is denoted by $T^\vee$.
The family of fibers of the map given by $f$ is a family $\{f=\lambda\}$, $\lambda\in \CC$.
Members of this family have natural compactifications to anticanonical sections of $T^\vee$. This family
(more precise, its compactification to a family $\{\lambda_0f=\lambda_1\}$ over $\PP[\lambda_0:\lambda_1]$)
is generated by a general member and the member that corresponds to the constant Laurent polynomial. The latter is nothing but
the boundary divisor $D$ of $T^\vee$ by Fact~\ref{fact: toric 4}. Denote the obtained pencil on $T^\vee$ by $f\colon Z_{T^\vee}\dashrightarrow \PP^1$
(we use the same notation $f$ for the Laurent polynomial, the corresponding pencil, and resolutions of this pencil for simplicity). By Lemma~\ref{lemma:restrictions Minkowski},
the base locus of $f$ on $Z_{T^\vee}$
is a union of smooth (rational) curves (possibly with multiplicities).
By Lemma~\ref{lemma: 3dim smoothness}, the variety $\widetilde T^\vee$ is a crepant resolution of $T^\vee$.
By definition of a Newton polytope, coefficients of the Minkowski Laurent polynomial at vertices of $\Delta$ are non-zero.
This means that the base locus does not contain any torus invariant strata of $T^\vee$ since it does not contain torus invariant
points by Fact $4$. % and intersects transversally torus invariant curves.
Thus we get a family $f\colon Z_{\widetilde T^\vee}\dashrightarrow \PP^1$, whose total space is smooth and a base locus is
a union of (transversally intersecting) smooth curves (possibly with multiplicities) again.
By Proposition~\ref{proposition:CY compactification}, there is a resolution $f\colon Z\to \PP^1$ of the base locus on
$Z_{\widetilde T^\vee}$ such that $Z$ is smooth and $-K_Z=f^{-1}(\infty)$.
Thus $Z$ is the required log Calabi--Yau compactification, and
$Y=Z\setminus f^{-1}(\infty)$ is a Calabi--Yau compactification.
\end{proof}

\begin{remark}
\label{remark: codimension 1}
The construction of Calabi--Yau compactification is not canonical: it depends on an order of blow ups of base locus components.
However all log Calabi--Yau compactifications are isomorphic in codimension one.
\end{remark}

\begin{proposition}[cf. Proposition~\ref{proposition:CY hyperelliptic-Laurent}]
\label{proposition:CY hyperelliptic-CY}
Fano threefolds $X_{1-1}$, $X_{1-11}$, $X_{2-1}$, $X_{2-2}$, $X_{2-3}$, $X_{9-1}$, and $X_{10-1}$
have toric Landau--Ginzburg models.
\end{proposition}

\begin{proof}
By Proposition~\ref{proposition:CY hyperelliptic-Laurent} these varieties have weak Landau--Ginzburg models. By~\cite{IKKPS} and~\cite{DHKLP}
they satisfy the period condition.
In a spirit of~\cite{Prz13} compactify the family given by $f_{i-j}$ to a family of (singular)
anticanonical hypersurfaces in $\PP^1\times \PP^2$ or $\PP^3$
and then crepantly resolve singularities of a total space of the family.
Consider these cases one by one.

A weak Landau--Ginzburg model for $X_{2-1}$ is the polynomial
$$
f_{2-1}=\frac{(x+y+1)^6(z+1)}{xy^2}+\frac{1}{z},
$$
that is a function on $\TTT[x,y,z]$.

Consider a family $\{f_{2-1}=\lambda\}$, $\lambda\in \CC$.
Make a birational change of variables %(cf. the proof of~\cite[Theorem $18$]{Prz13})
$$
x=\frac{1}{b_1}-\frac{1}{b_1^2b_2}-1,\ \ y=\frac{1}{b_1^2b_2},\ \ z=\frac{1}{a_1}-1
$$
and multiply the obtained expression by the denominator. We see that the family is birational to
$$
\{(1-a_1)b_2^3=\left((1-a_1)\lambda-a_1\right)a_1(b_1b_2-b_1^2b_2-1)\}\subset \Aff[a_1,b_1,b_2]\times \Aff[\lambda].
$$
Now this family can be compactified to a family of hypersurfaces of bidegree $(2,3)$ in $\PP^1\times \PP^2$ using the embedding
$\TTT[a_1,b_1,b_2]\hookrightarrow \PP[a_0:a_1]\times \PP[b_0:b_1:b_2]$.
The (non-compact) total space of the family has trivial canonical class and its singularities are a union of (possibly) ordinary double points and rational
curves which are du Val along a line in general points. Blow up any of these curves. We get
singularities of the similar type again. After several crepant blow ups one approaches to a threefold with just ordinary double points;
these points admit a small resolution. This resolution completes the construction of the Calabi--Yau compactification.
%(Cf.~\cite{Prz13}.)
Note that the total space $(\CC^*)^3$ of the initial family is embedded to the resolution.

In the similar way one gets Calabi--Yau compactifications for the other varieties.
One has
$$
f_{1-1}=\frac{(x+y+z+1)^6}{xyz}.
$$
The change of variables
$$
x=ab,\ \ y=ac,\ \ z=a-ab-ac-1,
$$
applied to the family
$\{f_{1-1}=\lambda\}$, and the multiplication on the denominator give the family of quartics $$
a^4=\lambda bc(a-ab-ac-1).
$$
The embedding $\TTT[a,b,c]\hookrightarrow \PP[a:b:c:d]$ gives the compactification to the family of singular quartics over $\Aff^1$.

One has
$$
f_{1-11}=\frac{(x+y+1)^6}{xy^2z}+z.
$$
The change of variables
$$
x=a-ab-\frac{c}{b}-1,\ \ y=ab,\ \ z=\frac{c}{b},
$$
applied to the family $\{f_{1-11}=\lambda\}$ and the multiplication on the denominator give the family of quartics
$$
a^4=(\lambda b-c)(a-ab-1)c.
$$
The embedding $\TTT[a,b,c]\hookrightarrow \PP[a:b:c:d]$ gives the compactification to the family of singular quartics over $\Aff^1$.

One has
$$
f_{2-2}=\frac{(x+y+z+1)^2}{x}+\frac{(x+y+z+1)^4}{yz}.
$$
The change of variables
$$
x=ab,\ \ y=bc,\ \ z=c-ab-bc-1
$$
applied to the family $\{f_{2-2}=\lambda\}$ and the multiplication on a denominator give the family of singular quartics
$$
ac^3=(c-ab-bc-1)(\lambda ab-c^2).
$$
The embedding $\TTT[a,b,c]\hookrightarrow \PP[a:b:c:d]$ gives the compactification
to the family of singular quartics over $\Aff^1$.

One has %(up to shift $f_{2-1}\to f_{2-3}+\alpha$, $\alpha\in \CC$)
$$
f_{2-3}=\frac{(x+y+1)^4(z+1)}{xyz}+z+1.
$$
The change of variables
$$
x=ac,\ \ y=a-ac-1,\ \ z=\frac{b}{c}-1
$$
applied to the family $\{f_{2-3}=\lambda\}$ and the multiplication on the denominator give the family
$$
a^3b=(\lambda c-b)(b-c)(a-ac-1).
$$
The embedding $\TTT[a,b,c]\hookrightarrow \PP[a:b:c:d]$ gives the compactification
to the family of singular quartics over $\Aff^1$.

One has
$$
f_{9-1}=x+\frac{1}{x}+\frac{(y+z+1)^4}{yz}.
$$
The change of variables
$$
x=\frac{c}{b},\ \ y=ac,\ \ z=a-ac-1
$$
applied to the family $\{f_{9-1}=\lambda\}$ and the multiplication on the denominator give the family
$$
a^3b=(\lambda bc-b^2-c^2)(a-ac-1).
$$
The embedding $\TTT[a,b,c]\hookrightarrow \PP[a:b:c:d]$ gives the compactification
to the family of singular quartics over $\Aff^1$.

One has
$$
f_{10-1}=\frac{(x+y+1)^6}{xy^2}+z+\frac{1}{z}.
$$
The change of variables
$$
x=\frac{1}{b_1}-\frac{1}{b_1^2b_2}-1,\ \ y=\frac{1}{b_1^2b_2},\ \ z=a_1
$$
applied to the family $\{f_{10-1}=\lambda\}$ and a multiplication on the denominator give the family
$$
a_1b_2^3=(\lambda a_1-a_1^2-1)(b_1b_2-b_1^2b_2-1).
$$
The embedding $\TTT[a_1,b_1,b_2]\hookrightarrow \PP[a_0:a_1]\times\PP[b_0:b_1:b_2]$ gives the compactification
to the family of singular hypersurfaces of bidegree $(2,3)$ in $\PP^1\times \PP^2$ over $\Aff^1$.

In all cases total spaces of the families have crepant resolutions.
\end{proof}

In~\cite{DHKLP} and~\cite{IKKPS} it is shown that all Fano threefolds with very ample anticanonical classes have weak Landau--Ginzburg
models satisfying the toric condition.
Thus, summarizing Theorem~\ref{theorem: Minkowski CY}, Proposition~\ref{proposition:CY hyperelliptic-Laurent}, Proposition~\ref{proposition:CY hyperelliptic-CY}, and~\cite{DHKLP} with~\cite{IKKPS},
one gets the following assertion.

\begin{corollary}
\label{corollary: toric LG for threefolds}
A pair of a smooth Fano threefold $X$ and a trivial divisor on it has a toric Landau--Ginzburg model.
Moreover, if $-K_X$ is very ample, then any Minkowski Laurent polynomial satisfying the period condition for $(X,0)$
is a toric Landau--Ginzburg model.
\end{corollary}

\begin{remark}
\label{remark:Hodge purity}
The compactification construction implies $h^{i,0}(Z)=0$ for all $i>0$.
\end{remark}

\begin{remark}
\label{remark: generic 3-Laurent}
Let us recall that $\widetilde T$ is a smooth toric variety with $F(\widetilde T)=\Delta$.
Let $f$ be a \emph{general} Laurent polynomial with $N(f)=\Delta$. The Laurent polynomial $f$ is a toric Landau--Ginzburg model for a pair $(\widetilde T,D)$, where $D$ is a general $\CC$-divisor on $\widetilde T$.
Indeed, the period condition for it is satisfied by~\cite{Gi97b}.
Following the compactification procedure described in the proof of Theorem~\ref{theorem: Minkowski CY}, one can see that the base locus $B$ is a union of smooth transversally intersecting curves (not necessary rational). This means that in the same way as above the statement of Theorem~\ref{theorem: Minkowski CY} holds for $f$, so that $f$
satisfies the Calabi--Yau condition (cf.~\cite{Harder}). Finally the toric condition holds for $f$ tautologically. Thus $f$ is a toric Landau--Ginzburg model for $(\widetilde T,D)$.
\end{remark}

\begin{problem}
Prove this for smooth Fano threefolds and any divisor. A description of Laurent polynomials for all Fano threefolds and
any divisor is contained in~\cite{DHKLP}.
\end{problem}

\begin{question}
Is it true that the Calabi--Yau condition follows from the period and the toric ones?
If not, what conditions should be put on a Laurent polynomial to hold the implication?
\end{question}

Another advantage of the compactification procedure described in Theorem~\ref{theorem: Minkowski CY} is that
it enables one to describe
``fibers of compactified toric Landau--Ginzburg models over infinity''. These fibers play an important role, say,
for computing Landau--Ginzburg Hodge numbers, see Part~\ref{part:KKP}. We summarize these considerations in the following assertion.

\begin{corollary}[{cf.~\cite[Conjecture 2.3.13]{Harder}}]
\label{proposition: fibers over infinity}
Let $f$ be a Minkowski Laurent polynomial. %or simplicial generalized binomial Laurent polynomial.
%Let $\widetilde T^\vee$ be a (smooth) maximally triangulated toric variety such that $F(\widetilde T^\vee)=N(f)$,
%and let $D$ be a boundary divisor of $\widetilde T^\vee$.
There is a log Calabi--Yau compactification $f\colon Z\to \PP^1$ with
$-K_Z=f^{-1}(\infty)=D$, where $D$ consists of $\frac{\left(-K_{T_{N(f)}}\right)^3}{2}+2$ components combinatorially given by
a triangulation of a sphere. (This means that vertices of the triangulation correspond to components of $D$, edges correspond to intersections
of the components, and triangles correspond to triple intersection points of the components.)
\end{corollary}

\begin{proof}
Let $\widetilde T^\vee$ be a (smooth) maximally triangulated toric variety such that $F(\widetilde T^\vee)=N(f)$,
and let $D$ be a boundary divisor of $\widetilde T^\vee$. The numbers of components of $D$ and $D'$ coincide.
Let $v$ be a number of vertices in a triangulation of $\nabla$; in other words, $v$ is a number of integral points on the boundary of $\nabla$,
or, the same, the number of components of $D$. Let $e$ be a number of edges in the triangulation of $\nabla$, and let $f$ be a number
of triangles in the triangulation. As the triangulation is a triangulation of a sphere, one has $v-e+f=2$.
On the other hand one has $2e=3f$. This means that $v=\frac{f}{2}+2$. The assertion of the corollary follows from the fact that
both $\left(-K_{T_{N(f)}}\right)^3$ and $f$ are equal to a normalized volume of $\nabla$.
\end{proof}

\begin{remark}
  Let $g=\frac{\left(-K_{X}\right)^3}{2}+1$ be the genus of Fano threefold $X$; in particular, $D$ consists of $g+1$ components.
Then one has $g+1=\dim |-K_X|$.
\end{remark}

General fibers of compactified toric Landau--Ginzburg models are smooth K3 surfaces.
However some of them can be singular or even reducible. Our observations give almost no information about them.
However singular fibers are of special interest: they contain information about the derived category of singularities.
There is a lack of examples of descriptions of singular fibers. More computable invariant is the number of components of fibers, see Theorems~\ref{theotem:KKP conjecture-3} and~\ref{theorem: CI components}.

\section{Toric Landau--Ginzburg models}
\label{section: toric LG 3-folds}
As we have mentioned, in~\cite{DHKLP} and~\cite{IKKPS} the toric condition was proven for a huge number
of smooth Fano threefolds (in particular, for those we need).
The methods used in these papers are theory of toric degenerations and analysis of tangent bundles
to deformation spaces at the points on the spaces we need.
In this section we study in details toric degenerations of Picard rank one Fano threefolds.

Let us give examples of toric Landau--Ginzburg models (of Minkowski type) and prove the toric condition for them.

%\newpage

\begin{center}
\begingroup
\begin{longtable}{||c|c|c|p{5cm}|p{6cm}||}
%\caption{Weak Landau--Ginzburg models for Fano threefolds.}
  \hline
  Var. & Index & Degree &
\begin{minipage}[c]{5cm}
\centering
\vspace{.1cm}

  Description

\vspace{.1cm}

\end{minipage}
  &
\begin{minipage}[c]{5cm}
\centering
\vspace{.1cm}

  Weak LG model

\vspace{.1cm}

\end{minipage}
\\
  \hline
  \hline
%  \endfirsthead
 $X_{1-1}$ & 1 & 2 &
\begin{minipage}[c]{5cm}
\vspace{.1cm}

A hypersurface of degree $6$ in $\PP(1,1,1,1,3)$.

\vspace{.1cm}

\end{minipage}
&
\begin{minipage}[c]{6cm}
\vspace{.1cm}
\centering

$\frac{(x+y+z+1)^6}{xyz}$

\vspace{.1cm}

\end{minipage}
    \\
  \hline
  $X_{1-2}$ & 1 & 4 &
\begin{minipage}[c]{5cm}
\vspace{.1cm}

  A general element of the family is quartic.

\vspace{.1cm}

\end{minipage}
&
\begin{minipage}[c]{6cm}
\vspace{.1cm}
\centering

  $\frac{(x+y+z+1)^4}{xyz}$

\vspace{.1cm}

\end{minipage}
  \\
  \hline
  $X_{1-3}$ & 1 & 6 &
\begin{minipage}[c]{5cm}
\vspace{.1cm}

A smooth complete intersection of quadric and cubic.

\vspace{.1cm}

\end{minipage}
&
\begin{minipage}[c]{6cm}
\vspace{.1cm}
\centering

$\frac{(x+1)^2(y+z+1)^3}{xyz}$

\vspace{.1cm}

\end{minipage}
\\
  \hline
  $X_{1-4}$ & 1 & 8 &
\begin{minipage}[c]{5cm}
\vspace{.1cm}

A smooth complete intersection of three quadrics.

\vspace{.1cm}

\end{minipage}
&
\begin{minipage}[c]{6cm}
\vspace{.1cm}
\centering

$\frac{(x+1)^2(y+1)^2(z+1)^2}{xyz}$

\vspace{.1cm}

\end{minipage}
\\
  \hline
  $X_{1-5}$ & 1 & 10 &
\begin{minipage}[c]{5cm}
\vspace{.1cm}

  A general element is a section of $G(2,5)$ by 2 hyperplanes in Pl\"{u}cker embedding and quadric.

%\vspace{.1cm}

\end{minipage}

&
%  $\frac{(x^2+x+z+zx+y+yx+yz)^2}{xyz}$ \\
\begin{minipage}[c]{6cm}
\vspace{.1cm}
\centering

$\frac{(1+x+y+z+xy+xz+yz)^2}{xyz}$

%\vspace{.1cm}

\end{minipage}
\\
  \hline
  $X_{1-6}$ & 1 & 12 &
\begin{minipage}[c]{5cm}
\vspace{.1cm}

  A linear section of the orthogonal Grassmannian $OG (5,10)$ of codimension $7$.

\vspace{.1cm}

\end{minipage}
&
\begin{minipage}[c]{6cm}
\vspace{.1cm}
\centering

  $\frac{(x+z+1)(x+y+z+1)(z+1)(y+z)}{xyz}$

\vspace{.1cm}

\end{minipage}
\\
  \hline
  $X_{1-7}$ & 1 & 14 &
\begin{minipage}[c]{5cm}
\vspace{.1cm}

  A section of $G(2,6)$ by 5 hyperplanes in Pl\"{u}cker embedding.

\vspace{.1cm}

\end{minipage}
  &

\begin{minipage}[c]{6cm}
\vspace{.1cm}
\centering

  $\frac{(x+y+z+1)^2}{x}$

  $+\frac{(x+y+z+1)(y+z+1)(z+1)^2}{xyz}$

\vspace{.1cm}

\end{minipage}
\\
  \hline
  $X_{1-8}$ & 1 & 16 &
\begin{minipage}[c]{5cm}
\vspace{.1cm}

  A linear section of symplectic Grassmannian $SGr(3,6)$ of codimension $3$.

  \vspace{.1cm}

\end{minipage}
  &
  \begin{minipage}[c]{6cm}
\vspace{.1cm}
\centering

$\frac{(x+y+z+1)(x+1)(y+1)(z+1)}{xyz}$

\vspace{.1cm}

\end{minipage}
\\
  \hline
  $X_{1-9}$ & 1 & 18 &
\begin{minipage}[c]{5cm}
\vspace{.1cm}

  A linear section of Grassmannian of the group $G_2$ of codimension $2$.

\vspace{.1cm}

\end{minipage}
  &
\begin{minipage}[c]{6cm}
\vspace{.1cm}
\centering

%  $\frac{((y+1)^2+z+yz+xy)(xy^2+yz+xz+xyz+x^2y)}{xyz}$
$\frac{(x+y+z)(x+xz+xy+xyz+z+y+yz)}{xyz}$
\vspace{.1cm}

\end{minipage}
\\
  \hline
  $X_{1-10}$ & 1 & 22 &
\begin{minipage}[c]{5cm}
\vspace{.1cm}

  A section of the vector bundle $\Lambda^2 \mathcal U^*\oplus \Lambda^2 \mathcal U^*\oplus \Lambda^2 \mathcal U^*$ on the Grassmannian $\G(3,7)$, where$\mathcal U$ is the tautological bundle

\vspace{.1cm}

\end{minipage}
  &
\begin{minipage}[c]{6cm}
\vspace{.1cm}
\centering

% $\frac{xy}{z}+\frac{y}{z}+\frac{x}{z}+x+y+\frac{1}{z}+4$ $+\frac{1}{x}+\frac{1}{y}+z+\frac{1}{xy}+\frac{z}{x}+\frac{z}{y}+
%\frac{z}{xy}$

$\frac{(z + 1) (x + y + 1) (x y + z)}{x y z} + \frac{x y}{z} + z + 3$

\vspace{.1cm}

\end{minipage}
\\
  \hline
  $X_{1-11}$ & 2 & $8\cdot 1$ &
\begin{minipage}[c]{5cm}
\vspace{.1cm}

 A hypersurface of degree $6$ in $\PP(1,1,1,2,3)$.

\vspace{.1cm}

\end{minipage}
&
\begin{minipage}[c]{6cm}
\vspace{.1cm}
\centering

%  $\frac{(x^2+y^2+z^2+1)^3}{xyz}$ \\
  $\frac{(x+y+1)^6}{xy^2z}+z$

\vspace{.1cm}

\end{minipage}
\\
  \hline
  $X_{1-12}$ & 2 & $8\cdot 2$ &
  \begin{minipage}[c]{5cm}
\vspace{.1cm}

 A hypersurface of degree $4$ in $\PP(1,1,1,1,2)$.

\vspace{.1cm}

\end{minipage}
&
\begin{minipage}[c]{6cm}
\vspace{.1cm}
\centering

$\frac{(x+y+1)^4}{xyz}+z$

\vspace{.1cm}

\end{minipage}
\\
  \hline
  $X_{1-13}$ & 2 & $8\cdot 3$ &
  \begin{minipage}[c]{5cm}
\vspace{.1cm}

Smooth cubic.

\vspace{.1cm}

\end{minipage}
&
\begin{minipage}[c]{6cm}
\vspace{.1cm}
\centering

$\frac{(x+y+1)^3}{xyz}+z$

\vspace{.1cm}

\end{minipage}
\\
  \hline
  $X_{1-14}$ & 2 & $8\cdot 4$ &
\begin{minipage}[c]{5cm}
\vspace{.1cm}

  Smooth intersection of two quadrics.

\vspace{.1cm}

\end{minipage}
&
\begin{minipage}[c]{6cm}
\vspace{.1cm}
\centering

  $\frac{(x+1)^2(y+1)^2}{xyz}+z$

\vspace{.1cm}

\end{minipage}
\\
  \hline
  $X_{1-15}$ & 2 & $8\cdot 5$ &
\begin{minipage}[c]{5cm}
\vspace{.1cm}

  A section of $G(2,5)$ by 3 hyperplanes in Pl\"{u}cker embedding.

\vspace{.1cm}

\end{minipage}
  &
\begin{minipage}[c]{6cm}
\vspace{.1cm}
\centering

  $x+y+z+\frac{1}{x}+\frac{1}{y}+\frac{1}{z}+xyz$

\vspace{.1cm}

\end{minipage}
\\
  \hline
  $X_{1-16}$ & 3 & $27\cdot 2$ &
\begin{minipage}[c]{5cm}
\vspace{.1cm}

  Smooth quadric.

\vspace{.1cm}

\end{minipage}
&
\begin{minipage}[c]{6cm}
\vspace{.1cm}
\centering

  $\frac{(x+1)^2}{xyz}+y+z$

\vspace{.1cm}

\end{minipage}
\\
  \hline
  $X_{1-17}$ & 4 & $64$ &
\begin{minipage}[c]{5cm}
\vspace{.1cm}

  $\PP^3$.

\vspace{.1cm}

\end{minipage}
&
\begin{minipage}[c]{6cm}
\vspace{.1cm}
\centering

  $x+y+z+\frac{1}{xyz}$

\vspace{.1cm}

\end{minipage}
\\
  \hline
  \hline

\caption[]{\label{table: Fano rank 1} Toric Landau--Ginzburg models for smooth Fano threefolds of Picard rank one}
\end{longtable}
\endgroup
\end{center}

Consider a projective variety $X\subset \PP^n$. Let it be defined by some homogeneous ideal $I\subset S=\CC[x_0,\ldots x_n]$. If $\prec$ is some monomial order for $S$, then there is a flat family degenerating $X$ to $X_\prec=V(\init_\prec (I))$, where $\init_\prec (I)$ is the initial ideal of $I$ with respect to the monomial order $\prec$. This is not of immediate help in finding toric degenerations of $X$, since in general, $X_\prec$ is highly singular with multiple components and thus cannot be equal to or degenerate to a toric variety.

Instead, the point is to consider toric varieties embedded in $\PP^n$ which also degenerate to $X_\prec$.
Consider such a toric variety $Z$, and let $\mathcal H$ be the Hilbert scheme of subvarieties of $\PP^n$ with Hilbert polynomial equal to that of $X$. If $X$ corresponds to a sufficiently general point of a component of $\mathcal H$ and $X_\prec$ lies only on this component, then $X$ must degenerate to $Z$. This is the geometric background for the following theorem; the triangulations which appear correspond to degenerations of toric varieties to certain special monomial ideals with unobstructed deformations.

\begin{theorem}[{\cite[Corollary 3.4]{CI14}}]
\label{theorem: triangulation-Ilten}
Consider a three-dimensional reflexive polytope $\nabla$ with $m$ lattice points, $7\leq m \leq 11$, which admits a regular unimodular triangulation with the origin contained in every full-dimensional simplex, and every other vertex having valency $5$ or $6$. Then the smooth Fano threefold of index $1$ and degree $2m-6$ admits a degeneration to $T_\Delta$, where $\Delta=\nabla^\vee$.
\end{theorem}

\begin{example}[$X_{1-6}$]\label{ex:mon12}
Consider the Laurent polynomial $f$ from Table \ref{table: Fano rank 1} for the Fano threefold $V_{12}$.  The dual of the Newton polytope $\nabla=\Delta_f^*$ is the convex hull of the vectors $\pm e_1$, $\pm e_2$, $e_3$, $-e_1-e_2$, $e_2+e_3$, and $-e_1-e_2-e_3$, see Figure \ref{fig:v12poly}. The polytope $\nabla$ has only one non-simplicial facet, a parallelogram. Subdividing this facet by either one of its diagonals gives a triangulation of $\partial \nabla$, which naturally induces a triangulation of $\nabla$ with the origin contained in every full-dimensional simplex.  It is not difficult to check that this triangulation is in fact regular and unimodular; furthermore, all vertices (with the exception of the origin) have valency $5$ or $6$. Thus, by Theorem~\ref{theorem: triangulation-Ilten}, the variety $X_{1-6}$ degenerates to $T_{\Delta_f}$.
\end{example}

\begin{figure}
\begin{center}
%\dualpolyfig
\includegraphics{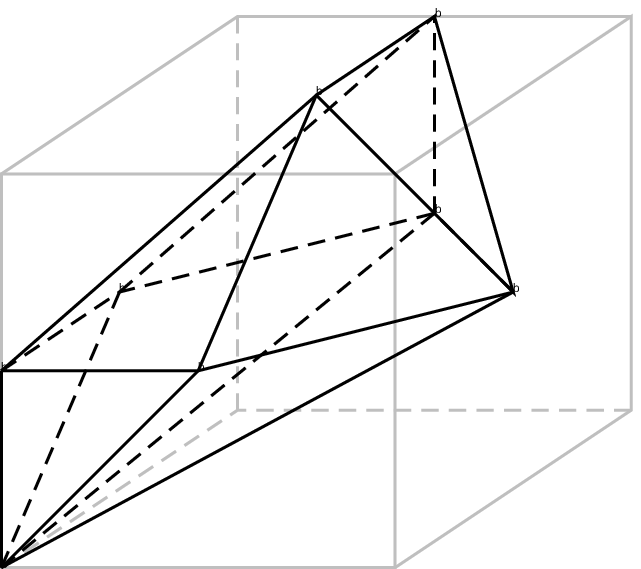}
\end{center}
\caption{The polytope $\Delta_f^\vee$ for $X_{1-6}$}\label{fig:v12poly}
\end{figure}

\begin{example}[$X_{1-4}$, $X_{1-5}$, $X_{1-7}$, and $X_{1-8}$]\label{ex:mongen}
Consider the Laurent polynomial $f$ from Table \ref{table: Fano rank 1} for $X_{1-i}$, $i\in\{4,5,7,8\}$.  Similar to the above example for $i=6$,
one can check by hand that the polytope
$\Delta_f^*$ satisfies the conditions of Theorem~\ref{theorem: triangulation-Ilten}. Thus, there is a degeneration of
$X_{1-i}$ to the toric variety $T_{\Delta}$ corresponding to the Landau--Ginzburg
model given by $f$.
\end{example}

\begin{example}[$X_{1-9}$]\label{ex:v18}
	Consider the Laurent polynomial $f$ from Table~\ref{table: Fano rank 1} for $X_{1-9}$. Here, $\nabla=\Delta_f^*$ has $12$ lattice points, so we cannot apply Theorem~\ref{theorem: triangulation-Ilten}, but similar techniques may be used to show the existence of the desired degeneration. Indeed, the dimension of the component $U$ corresponding to $X_{1-9}$ in the Hilbert scheme $\mathcal{H}_{X_{1-9}}$ of its anticanonical embedding is $153$, see~\cite[Proposition 4.1]{CI14}. The variety $T_\Delta$, where $\Delta=N(f)$, corresponds to a point $[T_\Delta]$ in $\mathcal{H}_{X_{1-9}}$ since its Hilbert polynomial agrees with that of $X_{1-9}$. A standard deformation-theoretic calculation shows that $[T_\Delta]$ is a smooth point on a component of dimension 153. It remains to be shown that this component is in fact $U$.

	Now, $\nabla=\Delta^\vee$ admits a regular unimodular triangulation such that the origin is contained in every full-dimensional simplex, one boundary vertex has valency $6$, and every other vertex has valency $4$ or $5$. The boundary of this triangulation is in fact the unique triangulation of the sphere with these properties. In any case, $T_\Delta$ degenerates to the Stanley--Reisner scheme $R$ corresponding to this triangulation, and $X_{1-9}$ does as well, see~\cite[Corollary 3.3]{CI14}. Furthermore, a standard deformation-theoretic calculation
%using \cite{ilten:11d}
shows that at the point $[R]$,
	$\mathcal{H}_{X_{1-9}}$ has only one $153$-dimensional component. Thus, $[T_\Delta]$ must lie on $U$, and $X_{1-9}$ must degenerate to $T_\Delta$.
\end{example}

Thus, independently from~\cite{DHKLP} and \cite{IKKPS}, we proved the following theorem (cf. Corollary~\ref{corollary: toric LG for threefolds}).

\begin{theorem}
\label{theorem: rank 1 toric 3-folds}
Each Fano threefold of rank 1 has a toric weak Landau--Ginzburg model.
\end{theorem}

\begin{proof}
According to~\cite{CCGK16}, Laurent polynomials from Table~\ref{table: Fano rank 1} are weak Landau--Ginzburg models
of the corresponding Fano varieties. According to Theorem~\ref{theorem: Minkowski CY}, they satisfy the Calabi--Yau condition.
Thus the last thing needed to check is the toric condition.
The varieties $X_{1-i}$, $i\in \{1,2,3,4,11,12,13,14\}$, are complete intersections in weighted projective spaces,
so the toric condition for them follows from Theorem~\ref{theorem: CI are toric LG}. The varieties $X_{1-10}$ and $X_{1-15}$
have small toric degenerations (i,\,e. degenerations to terminal Gorenstein toric varieties), so the toric condition
for them follows from~\cite{Gal08}. The toric condition for $X_{1-i}$, $i\in \{5,6,7,8\}$, follows from Examples~\ref{ex:mon12} and~\ref{ex:mongen}. The toric condition for $X_{1-9}$ follows from Example~\ref{ex:v18}. Finally, $X_{1-17}=\PP^3$ is toric.
\end{proof}

\section{Modularity}
\label{section:Modularity}
In this section we present results from~\cite{DHKLP}, see also~\cite{ILP13}.

Mirror Symmetry predicts that fibers of Landau--Ginzburg model for a Fano variety are Calabi--Yau varieties.
More precise, it is expected that these fibers are mirror dual to anticanonical sections of the Fano variety.
In the threefold case this duality is nothing but Dolgachev--Nikulin duality of K3 surfaces.

Let $H$ be a hyperbolic lattice, $\ZZ\oplus \ZZ$ with intersection form
$$
\left(
  \begin{array}{cc}
    0 & 1 \\
    1 & 0 \\
  \end{array}
\right).
$$
The intersection lattice on the second cohomology on any K3 surface is
$$N=H\oplus H\oplus H\oplus E_8(-1)\oplus E_8(-1).
$$
Consider a family $U_K$ of K3 surfaces whose lattice of algebraic cycles contains $K\subset N$ (and coincides with $K$ for general K3 surface).
Consider a lattice $L'=K^\bot$, the orthogonal to $K$ in $N$. Let $L'=H\oplus L$.

\begin{definition}[see~\cite{Do01}]
\label{definition: Dolgachev-Nikulin}
The family of K3 surfaces $U_L$ is called \emph{the Dolgachev--Nikulin dual} family to $U_K$.
\end{definition}

Consider a principally polarized family of anticanonical sections of a Fano threefold $X$ of index $i$ and degree $(-K_X)^3=i^3k$.
It is nothing but $U_{\langle 2n \rangle}$, $2n=ik$, where $\langle r \rangle$ is a rank $1$ lattice generated by
vector whose square is $r$. The lattice $U_{\langle 2n \rangle}$ is a sublattice of $H$.  Using this embedding to one of the $H$-summands
of $N$ we can see that Dolgachev--Nikulin dual lattice to $U_{\langle 2n \rangle}$ is the lattice
$$
M_n=H\oplus E_8(-1)\oplus E_8(-1)+\langle -2n \rangle.
$$

The surfaces with Picard lattices $M_n$ are \emph{Shioda--Inose}. They are resolutions of quotients of specific K3 surfaces $S$
by \emph{Nikulin involution}, the one keeping the transcendental lattice $T_S$; it changes two copies of $E_8(-1)$. Another description of Shioda--Inose surfaces is Kummer ones going back to products of elliptic curves with $n$-isogenic ones. $M_n$-polarized Shioda--Inose surfaces form an $1$-dimensional irreducible
family.

It turnes out that fibers of toric Landau--Ginzburg models from Table~\ref{table: Fano rank 1} can be compactified to  Shioda--Inose surfaces dual to anticanonical sections of Fano threefolds.
In this section we prove the following.

\begin{theorem}[\cite{DHKLP}]
\label{theorem: modularity}
Let $X$ be a Fano threefold of Picard rank $1$, index $i$ and let $(-K_X)^3=i^3k$. Then a general fiber of toric weak Landau--Ginzburg model from Table~\ref{table: Fano rank 1} is a Shioda--Inose surface with Picard lattice $M_{ik/2}$.
\end{theorem}

We say that toric Landau--Ginzburg model for $X$ satisfies the Dolgachev--Nikulin condition if the assertion of Theorem~\ref{theorem: modularity}
holds for it.
We call such toric Landau--Ginzburg model \emph{good}.

Thus compactifications of the Landau--Ginzburg models are, modulo coverings and the standard action of action of $PSL(2,\CC)$ on the base,
the unique families of corresponding Shioda--Inose surfaces. More precise, they are index-to-one coverings of the moduli spaces.

\begin{corollary}[cf.~\cite{DHNT17}]
\label{corollary:uniqueness}
A Calabi--Yau compactification of good weak Landau--Ginzburg model is unique up to flops.
\end{corollary}

Thus, if Homological Mirror Symmetry conjecture holds for Picard rank one Fano threefolds, then their Landau--Ginzburg models
(in the Homological Mirror Symmetry sense) up to flops are compactifications of toric ones from Table~\ref{table: Fano rank 1}.
Moreover, all other good toric Landau--Ginzburg models are birational (over $\Aff^1$) to them.

To prove Theorem~\ref{theorem: modularity} we study all $17$ one-by-one and compute Neron--Severi lattices of the compactified
toric Landau--Ginzburg models.

\begin{remark}
\label{remark: Golyshev modular}
In~\cite{Go07} Golyshev described Landau--Ginzburg models for Picard rank one Fano threefolds as universal families over $X_0(n)/\tau$, where $\tau$ is an Atkin--Lehner involution, with fibers that are Kummer surfaces associated with products of elliptic curves by an $n$-isogenic ones.
Golyshev's description of periods of these dual families as modular forms seems to be natural to expect from this point of view.
The variations of Hodge structures of our families of Shioda--Inose surfaces are the same (over $\QQ$) as the variations for the products of elliptic
curves and the same over $\ZZ$ as for the Kummer surfaces; this follows from the description of Shioda--Inose surfaces given above.
\end{remark}

\begin{remark}
Fibers of Landau--Ginzburg models are expected to be Dolgachev--Nikulin dual to anticanonical sections of Fano varieties of any Picard rank.
As the Picard rank of the Fano increase, the mirror K3 fibers will no longer be Shioda--Inose.  However they are still K3 surfaces of high Picard rank, so we can hope to find analogous modular-type properties (say, automorphic) in these cases as well.
Say, fibers of Landau--Ginzburg models are Kummer surfaces given by products of elliptic curves for the Picard rank $2$ case
and by abelian surfaces for the Picard rank $3$ case.
These lattices are computed over $\QQ$ in~\cite{CP18}; however the computations over $\ZZ$
need more deep methods.
\end{remark}

\subsection{Lattice facts}
\label{latticefacts}

If $L$ is a lattice and $ k $ a field, we will write $ L_k $ for $ L \otimes_\ZZ k $.  We will use $ \scr N, \scr M $ to denote two dual rank-three lattices.  Let $ f_{1-i} $ denote the Laurent polynomial defining the Landau--Ginzburg model from Table~\ref{table: Fano rank 1}
that correspond to $X_{1-i}$, let~$\mbox{$\Delta_{f_{1-i}}^* \subset \scr M_\RR$}$ be its Newton polytope, and let $ \nabla_{f_{1-i}} \subset \scr N_\RR $ be its polar.

%We will write $ \langle r \rangle $ for a one-dimensional lattice generated by an element of square $r$.
Via $ A_n, D_n, E_n $ we denote the %{\em negative-definite}
root lattices of the corresponding Dynkin diagrams. Via  $M$ we denote the rank 18 lattice $ H \oplus E_8(-1) \oplus E_8(-1) $, and via $ M_n $ the rank 19 lattice $ M \oplus \langle -2n \rangle $.

We will use $ (x:y:z:w) $ as homogeneous coordinates on $ \PP^3 $.  For distinct, non-empty subsets $ I, J, K \subset \{ 1, 2, 3, 4 \} $, we will write $H_I $ for the hyperplane defined by setting the sum of coordinates in $I$ equal to zero. Thus, for example, $H_{\{ 1\}} $ is the coordinate hyperplane $ x = 0 $, while $H_{\{2,4\}} $ is the hyperplane defined by $ y+w=0$.  We write $ L_{I,J} = H_I \cap H_J $, and $ p_{I,J,K} = H_I \cap H_J \cap H_K $.

In many cases, we will use Calabi--Yau compactifications that are different from those from Section~\ref{section: Calabi--Yau 3-folds}.
That is, we use compactifications given by
$$
(\CC^*)^3\hookrightarrow \PP[x:y:z:w],
$$
cf. Section~\ref{subsection:KKP-3}.
This gives precise descriptions of fibers of compactifications as quartics in $\PP^3$ with ordinary double points.
In those cases, we will identify some curves on the minimal resolutions of these singular quartics (which will be K3 surfaces) and compute the intersection matrix of the identified curves, then checking that this matrix has rank 19.  In the interest of not boring the reader to death, we will omit the details of these computations.  In other cases, we will use elliptic fibrations as described below.

Because we will use them later, we recall a few (perhaps not terribly well-known) facts about lattices.  Most are due to \cite{Nikulin}; a very readable reference is \cite{Be02}.  Let $ L$ be a lattice, and $ \langle \cdot, \cdot \rangle $ the bilinear pairing on $L$.  Denote by $ L^* $ the dual lattice $ \mathrm{Hom}(L,\ZZ)  $.  Since the pairing induces an isomorphism $ L_{\QQ} \simeq \mathrm{Hom}(L_\QQ, \QQ) $, we may think of $ L^* \subset L_\QQ $.  The pairing $ \langle \cdot, \cdot \rangle $ induces a quadratic form $ q_L $ on the discriminant group $ D(L) = L^*/L$ by $ q_L(\phi) = \langle \phi, \phi \rangle $.  {\em A priori}, $ q_L $ takes values in $ \QQ/\ZZ $, but if $L$ is an even lattice, it will take values in $ \QQ/(2\ZZ) $.

Fixing a basis $ e_1, \ldots e_r $ for $ L $, we may form the {\em Gram matrix} $ I_L $ whose $(i,j)$-th entry is $ \langle e_i, e_j \rangle $.  We call $ d(L) = \det(I_L) $ the {\em determinant} of $L$.

\begin{fact} \label{uniquenessfact} Let $L$ be an even, indefinite lattice of rank $r$ and signature $(s,r-s)$, and let $ d $ be the minimal number of generators of $ L^*/L $.  If $ r > d+2 $, then $ q_L $ and $s$ uniquely determine $L$.  \end{fact}

\begin{fact}  \label{indexfact} Let $ L \subset M $ be even lattices of the same rank.  Then $ [M:L]^2 = d(L)/d(M) $. \end{fact}

\begin{fact} \label{discriminantfact} Let $ L \subset M $ be even lattices of the same rank, and let $ G = M/L \subset L^*/L = D $.  Note since $ L \subset M \subset M^* \subset L^* $, we have $ G \subset M^*/L \subset D $ and $ (M^*/L)/G \simeq M^*/M $.  Now let $ G^\perp = \{ a \in D \: | \: q_L(a+H) = q_L(a) \} $. Since $M$ is even, $ q_L|_G = 0 $, and hence $ G \subset G^\perp $.  Moreover, given $ a \in D $, choose $ \tilde{a} \in L^* $ such that $ a = \tilde{a}+L$.  Then $ a \in G^\perp $ if and only if $ \langle \tilde{a}, M \rangle \subset \ZZ $, i.e. $ G^\perp = M^*/L $.  Thus we see that the quadratic form $ q_M $ is nothing but $ q_L |_{G^\perp} $ descended to $ G^\perp/G $.

Conversely, given a subgroup $ G \subset D $ such that $ q_L(G) = 0 $, there exists a lattice $M$ containing $ L $ such that $ M/L = G $.
\end{fact}

\begin{fact} \label{orthogonalfact} Let $ L $ be a sublattice of a unimodular lattice $ \Lambda $.  Then  $ D(L) \simeq D(L^\perp) $ and $ q_L = - q_{L^\perp} $. \end{fact}

For convenience, we also include the discriminant groups and forms of some of the lattices that play a role in the present study.
In Table~\ref{discriminanttable} we present the discriminant form by giving its values on generators of the discriminant group.  Note that this description is not unique.  For example, if the discriminant group is $ \ZZ/(8) $ and the form is listed as $ 1/8 $, this means that a generator $ g $ of the group has $ q(g) = 1/8 $.  Of course, $ 3g$ is also a generator, and it has $ q(3g) = 9/8 $.

\begin{table}[h]
\label{discriminanttable}
\centering
\begin{tabular}{|l|l|l|} \hline
Lattice $L$  & Group $D(L) $ & Form $ q_L$ \\ \hline
$H$ & $\{1\}$ & $0$ \\ \hline
$ \langle -2n \rangle $ & $ \ZZ/(2n) $ & $ -1/(2n) $ \\ \hline
$ A_1 $ & $\ZZ/(2) $ & $ -1/2 $ \\
$ A_2 $ & $ \ZZ/(3) $ & $ 4/3 $ \\
$ A_3 $ & $ \ZZ/(4) $ & $ 5/4 $ \\
$ A_4 $ & $ \ZZ/(5) $ & $ 4/5 $ \\
$ A_5 $ & $ \ZZ/(6) $ & $  $ \\
$ A_6 $ & $ \ZZ/(7) $ & $ 2/7 $ \\
$ A_7 $ & $ \ZZ/(8) $ & $ 1/8 $ \\
$ A_8 $ & $ \ZZ/(9) $ & $ 4/9 $ \\
$ A_9 $ & $ \ZZ/(10) $ & $ -9/10 $ \\
$ A_{10} $ & $ \ZZ/(11) $ & $ 4/11 $ \\
$ A_{11} $ & $ \ZZ/(12) $ & $ -11/12 $ \\ \hline
$ D_5 $ & $ \ZZ/(4) $ & $-5/4$ \\
$ D_8 $ & $ \ZZ/(2) \oplus \ZZ/(2) $ & $0, 1$ \\
$ D_{10} $ & $ \ZZ/(2) \oplus \ZZ/(2) $ & $ 1, 1 $ \\ \hline
$ E_6 $ & $ \ZZ/(3) $ & $ 2/3 $ \\
$ E_7 $ & $ \ZZ/(2) $ & $1/2$ \\
$ E_8 $ & $ \{1\} $ & $ 0 $ \\ \hline
 \end{tabular}
\caption{Some Discriminant Groups and Forms}
\end{table}

\subsection{Elliptic fibrations on K3 surfaces} \label{fibrationsec}
We briefly recall a few facts about elliptic fibrations with section on K3 surfaces.

\begin{definition} An {\em elliptic K3 surface with section} is a triple $ (X, \pi, \sigma) $, where $X$ is a K3 surface and $ \pi: X \to \PP^1 $ and $ \sigma: \PP^1 \to X $ are morphisms with the generic fiber of $ \pi $ an elliptic curve and $ \pi \circ \sigma = \mathrm{id}_{\PP^1} $. \end{definition}

Any elliptic curve over the complex numbers can be realized as a smooth cubic curve in $ \PP^2 $ in {\em Weierstrass normal form}
\begin{equation} \label{Weierstrass} y^2 z = 4x^3 -g_2 x z^2 - g_3 z^3. \end{equation}
Conversely, the equation~\eqref{Weierstrass} defines a smooth elliptic curve provided $ \Delta = g_2^3 -27g_3^2 \neq 0 $.

Similarly, an elliptic K3 surface with section can be embedded into the $ \PP^2 $ bundle $ \PP(\scr O_{\PP^1} \oplus \scr O_{\PP^1}(4) \oplus \scr O_{\PP^1}(6)) $ as a subvariety defined by equation~\eqref{Weierstrass}, where now $ g_2, g_3 $ are global sections of $ \scr O_{\PP^1}(8) $, $ \scr O_{\PP^1}(12) $ respectively (i.e. they are homogeneous polynomials of degrees 8 and 12).  The singular fibers of $ \pi $ are the roots of the degree 24 homogeneous polynomial $ \Delta = g_2^3 -27g_3^2 \in H^0(\scr O_{\PP^1}(24)) $.  Tate's algorithm can be used to determine the type of singular fiber over a root $ p$ of $ \Delta $ from the orders of vanishing of $ g_2$, $g_3$, and $\Delta $ at $p$.

\begin{proposition} \label{fibprop} \cite[Lemma 3.9]{CD07} A general fiber of $ \pi $ and the image of $ \sigma $ span a copy of $H$ in $ \mathrm{Pic}(X) $.  Further, the components of the singular fibers of $ \pi$ that do not intersect $ \sigma $ span a sublattice $ S$ of $ \mathrm{Pic}(X) $ orthogonal to this $H$,  and $ \mathrm{Pic}(X)/(H \oplus S) $ is isomorphic to the Mordell--Weil group $MW(X,\pi)$ of sections of $ \pi $.  \end{proposition}

\begin{proposition} \label{mordellprop} \cite[Corollary VII.3.1]{Mi89}   The torsion subgroup of $MW(X, \pi) $ embeds in $ D(S) $. \end{proposition}

When K3 surfaces are realized as hypersurfaces in toric varieties, one can construct elliptic fibrations combinatorially.  As before, let $ \Delta \subset \mathcal N_\QQ $ be a reflexive polytope, and suppose $ P \subset \mathcal N $ is a plane such that $ \Delta \cap P $ is a reflexive polygon $ \nabla $.  Let $ m \in \mathcal M=\mathcal N^\vee $ be a normal vector to $P$.  Then $P$ induces a torus-invariant map $ \PP(\Delta^*) \to \PP^1 $  with generic fiber $ \PP_{\nabla} $, given in homogeneous coordinates by
\begin{equation*}\pi_m:  (z_1, \ldots z_r) \mapsto \left[ \prod_{\langle v_i, m \rangle >0} z_i^{\langle v_i, m \rangle}, \prod_{\langle v_i, m \rangle <0} z_i^{-\langle v_i, m \rangle} \right]. \end{equation*}
Restricting $ \pi_m $ to an anticanonical K3 surface, we get an elliptic fibration.  If $ \nabla  $ has an edge without interior points, this fibration will have a section as well.  See~\cite{KS02} for more details.

\subsection{Picard lattices of fibers of the Landau--Ginzburg models}

\begin{enumerate}
%1
\item[$X_{1-1}$]
Recall that a Landau--Ginzburg model of Givental type for $X_{1-1}$ is

$$
\left\{%
\begin{array}{ll}
y_0y_1y_2y_3y_4^3=1\\
y_1+y_2+y_3+y_4=1
\end{array}%
\right.
$$
with superpotential
$$
w=y_0.
$$
Consider the change of variables
$$
y_1=\frac{x}{x+y+z+t}, \ \ y_2=\frac{y}{x+y+z+t},\ \ y_3=\frac{z}{x+y+z+t},\ \ y_4=\frac{t}{x+y+z+t},
$$
where $x,y,z,t$ are projective coordinates.
We get the Landau--Ginzburg model
$$
y_0xyzt^3=(x+y+z+t)^6, \ \ \ \ w=y_0.
$$
Thus in the local chart, say, $t\neq 0$ we get the toric Landau--Ginzburg model from Table~\ref{table: Fano rank 1}
$$
f_{1-1}=\frac{(x+y+z+1)^6}{xyz}.
$$
A general element of the pencil that correspond to $f_{1-1}$ is birational to the general element of the initial Landau--Ginzburg model.
Inverse the superpotential:~$\mbox{$u=1/w$}$. We get the pencil given by
$$
y_1y_2y_3y_4^3=u,\ \ \ \ y_1+y_2+y_3+y_4=1.
$$
This is the Landau--Ginzburg model for weighted projective space $\PP(1:1:1:3)$, see~\cite[(2)]{CG11}.
(In particular, by~\cite[Theorem 1.15]{CG11} its general element is birational to a K3 surface.)
However make another change of variables in the Givental's Landau--Ginzburg model putting
 $ x = y_1$, $ y = y_2 $, $ z = y_4 $.  We get the family given by
\begin{equation*} \widetilde{f}_{1-1} = x + y + z + \frac{w}{xyz^3} - 1 = 0 \end{equation*}
Let $ \widetilde{\Delta}_{f_{1-1}}$ be a Newton polytope of the polynomial $ \widetilde{f}_{1-1} $ and let $ \widetilde{\nabla}_{f_{1-1}} =\widetilde{\Delta}_{f_{1-1}}^\vee$.
Then fibers of the pencil $\{ \widetilde{f}_{1-1} = 0 \}$ can be compactified inside $T_{\widetilde \nabla_{f_{1-1}}}$,
cf. Section~\ref{section: Calabi--Yau 3-folds}.
The normal vector $ (1,2,3) $ induces an elliptic fibration with a section.  The Weierstrass form of the fibers of the elliptic fibration is
\begin{equation*} -\frac{t^4 u}{48}+\frac{1}{864} t^5 \left(864 t^2+1728 t \lambda -t+864 \lambda ^2\right)+u^3+v^2 = 0. \end{equation*}
Hence by Tate's algorithm there are singular fibers of type $ II^* $ at $ t= 0, \infty $ and $ I_2 $ at $ t = -\lambda $.  Therefore, the K3 surfaces in question are polarized by $ H \oplus E_8(-1) \oplus E_8(-1) \oplus A_1(-1) = M_1.$

There is also another fibration induced by the normal $ (1,0,1)$ which gives a polarization by
$$
H \oplus E_7(-1) \oplus D_{10}(-1).
$$

%2
\item[$X_{1-2}$]
Compactify this family to the family of quartics $ (x+y+z+w)^4 - \lambda x y z w =0 $ in  $\PP^3 $.
Intersecting the quartic with the pencil of planes containing one of lines lying on it gives a pencil of divisors with the line as base locus.  Subtracting the line gives a pencil of cubics.  Blowing up the base points of this pencil gives an elliptic fibration with section, which gives a polarization of the K3 surfaces by $ H \oplus E_6(-1) \oplus A_{11}(-1) $.  This fibration has a 3-torsion section, and it can have no other torsion sections by Proposition \ref{mordellprop}.  Thus by Fact \ref{indexfact}, the generic fiber $ X $ of $ f_{1-2} $ has
$ d(NS(X)) = 4 $.  As we shall see, fibers of the Landau--Ginzburg model $X_{1-17}=\PP^3$ have fibrations of this type as well, and comparing parameters of the two Weierstrass equations, we see that fibers of compactified toric Landau--Ginzburg models for $X_{1-1}$ and $X_{1-17}$
are the same.  Because generic fibers of compactifications for $f_{1-17}$ are $M_2$-polarized (as we will see soon), $ NS(X) \simeq M_2 $.

%3
\item[$X_{1-3}$]
Compactify the fibers of $ f_{1-3} $ as a family of anticanonical divisors in $ \PP^1 \times \PP^2 $ via $ (x,y,z) \mapsto ((x:1) \times (y:z:1))$. Explicitly, $ f_{1-3}^{-1}(\lambda) $ compactifies to the K3 surface
\begin{equation*} Y_\lambda = \{ ((x: x_0),(y:z:w)) \in \PP^1 \times \PP^2 \: | \:  (x+x_0)^2(y+z+w)^3 - \lambda x x_0 y z w = 0 \}. \end{equation*}
The projection $ \PP^1 \times \PP^2 \to \PP^1 $ induces an elliptic fibration on $ Y_\lambda $ for generic $ \lambda $.  The map $ (x: x_0) \mapsto ( (x: x_0) , (1:-1:0)) $ gives a section of this elliptic fibration.  Putting the fiber over $ (1:a) $ into Weierstrass form
\begin{equation*} \frac{a^3 \lambda^3(24(1+a)^2-a\lambda)}{48} X -\frac{a^4 \lambda^4 (36 (1 + a)^2 (6 (1 + a)^2 - a s) + a^2 s^2)}{864}+X^3+Y^2 = 0 \end{equation*}
and using Tate's algorithm, we see singular fibers of Kodaira type $IV^* $ at $ a =0, \infty $; $ I_6 $ at $ a = -1 $; and $ I_1 $ where $ 27(a+1)^2 - \lambda a = 0 $.  Hence the rank 19 lattice $ H \oplus E_6(-1) \oplus E_6(-1) \oplus A_5(-1) $ embeds in the Picard lattice of $ Y_\lambda $.

As we will see later, the fibers of $ f_{1-16} $ also have fibrations of this type and are $M_3$-polarized.  Matching the Weierstrass equations, we conclude that fibers for $ f_{1-3} $ are isomorphic to fibers for $ f_{1-16} $, and hence fibers for $ f_{1-3} $ must also be $ M_3$-polarized.

%4
\item[$X_{1-4}$]
Similar to the case above, we compactify the family as anticanonical K3 surfaces in  $ \PP^1 \times \PP^1 \times \PP^1 $.  Projection onto one of the $ \PP^1 $ factors gives the generic K3 fiber an elliptic fibration with section.  Putting this into Weierstrass form and running Tate's algorithm give us an embedding of the rank 19 lattice $\mbox{$ H \oplus A_7(-1) \oplus D_5(-1) \oplus D_5(-1) $}$ into the Picard lattice of the generic fiber.  Moreover, the Mordell--Weil group is isomorphic to $ \ZZ/(4) $.  Applying the lattice facts above, with $ L =  H \oplus A_7 \oplus D_5 \oplus D_5 $, $ M = NS(X) $, $ G = MW(X) \simeq \ZZ/(4) $, and $\mbox{$ D = L^*/L \simeq \ZZ/(8) \oplus \ZZ/(4) \oplus \ZZ/(4) $}$, we have that $ d(M) = 8 $.  Examining the possibilities for $ G \subset D $, we conclude $ M^*/M \simeq \ZZ/(8) $, and that $ q_M $ of a generator is $ 7/8$.

Now we claim that $ M \simeq M_4 $.  Let $ e $ be a generator of the $ \langle -8 \rangle $ direct summand of $ M_4 $.  Since $ H$ and $ E_8 $ are unimodular, the group $ M_{4}^*/M_{4} \simeq \ZZ/(8) $ is generated by $ \epsilon = \frac{1}{8} e $, and $ q_{M_4}(\epsilon) = -1/8 $.  Note that
the element $ 3 \epsilon $ also generates $ M_4^*/M_4 $, and $ q_{M_4}(3 \epsilon) = -9/8 \equiv 7/8 \: (mod\ 2\ZZ) $.  Thus we see that $M$ and $M_4$ have the same discriminant form, and their ranks are sufficiently large relative to the number of generators of the discriminant groups.  Hence by Fact~\ref{uniquenessfact} one gets that $M$ and $M_4$ must be isomorphic.

%5
\item[$X_{1-5}$]
Compactify fibers to singular quartics. There are singularities at $ p_{ \{i \}, \{j \}, \{4\}}$ for $1 \leq i \neq j \leq 3$ of type $ \sD_4 $ and at $  p_{\{i \} \{ j \}, \{k,4\}}$ where $\{i,j,k\} = \{1,2,3\}$ of type $\sA_1$.  Thus the exceptional curves generate a sublattice of rank 15.  The quartics also contain lines $ L_{ \{i\}, \{j, 4\}} $ and conics $ C_{\{i,j,4\}} $ for $ 1 \leq i \neq j \leq 3 $, subject to relations  from
\begin{eqnarray*}
\nonumber H_{\{1\}} & = & 2 L_{\{1\} , \{2,4 \}} + 2 L_{\{1 \}, \{3,4 \}}, \\
\nonumber H_{\{2\}} &=& 2 L_{\{2\},\{1,4\}} + 2 L_{\{2\},\{3,4\}}, \\
\nonumber H_{\{3\}} &=& 2 L_{\{3\},\{1,4\}} + 2 L_{\{3\},\{2,4\}}, \\
 H_{124} &=& L_{\{1\}\{2,4\}} + L_{\{2\},\{1,4\}} + C_{\{1,2,4\}}, \\
\nonumber H_{\{1,3,4\}} &=& L_{\{1\},\{3,4\}} + L_{\{3\},\{1,4\}} + C_{\{1,3,4\}}, \\
\nonumber H_{\{2,3,4\}} &=& L_{\{2\},\{3,4\}} + L_{\{3\},\{2,4\}} + C_{\{2,3,4\}},
\end{eqnarray*}
which leave a lattice of rank 19.

Explicitly computing the intersection matrix for the identified curves (one needs to blow up singular curves for this and
figure out how strict images of the curves intersect exceptional curves, see Proposition~\ref{proposition:du-Val-intersection}) shows that they generate a lattice with determinant 10, discriminant group $ \ZZ/(10) $ with a generator $ \alpha $ having $ q(\alpha) =11/10 $.  Choosing instead the generator $ \beta = 3 \alpha $, we have
$$ q(\beta) = 99/10 \equiv -1/10 \ (mod \ 2 \ZZ). $$  Hence this lattice is isomorphic to $ M_5 $.

We have just show that for $ X $ a general K3 in this pencil, $ NS(X) $ contains $ M_5 $.  To see that $NS(X) $ actually {\em equals} $M_5 $, we note that since $ M $ is unimodular and contained in $NS(X) $, it must be a direct summand.  Because $NS(X) $ is an even  lattice of signature (1,18), the orthogonal complement of $ M \subset NS(X) $ must be even, negative definite, and rank 1 and hence be equal to $ \langle -2n \rangle $ for some $n$.  From Fact \ref{indexfact}, $ 10/(\det NS(X)) = 5/n $ must be a square, and hence $n=5$.

Alternately, the intersection of one of the singular quartics with a plane containing $ L_{ \{1\}, \{2,4\}} $ consists of $ L_{ \{1\}, \{2,4\}} $ and a (generically) smooth cubic.  The pencil of these cubics, with base points blown up, give an elliptic fibration on the minimal resolution of the quartic.  This fibration has singular fibers of types $ I_2^* $, $I_1^* $, $ I_6 $, and 3 ones of type $I_1 $.  It also has a section of infinite order and a 2-torsion section.  Hence the Picard lattice of the generic member of this family is a rank 19 lattice containing $$
H \oplus D_6(-1) \oplus D_5(-1) \oplus A_5(-1)
$$ with quotient $ \ZZ \oplus \ZZ/(2) $.

%6
\item[$X_{1-6}$]
Again, we can compactify the fibers for $ f_{1-6} $ to singular quartics in the standard way.  There are $\sA_1 $ singularities at $ (1:-1:0:0)$, $(1:0:-1:0)$, and $ (0:1:-1:0)$; $ \sA_2$ singularities at $ (1:0:0:0) $ and $(0:0:1:-1)$; and $\sA_3$ singularities at $ (0:1:0:0) $ and $ (1:0:0:-1)$.  These quartics also contain twelve lines:
\begin{align*}  L_{ \{1\}, \{2,3\}},
 L_{ \{1\}, \{3,4\}},
 L_{ \{1\}, \{2,3,4\}},
 L_{ \{2\}, \{3\}},
 L_{ \{2\}, \{3,4\}},
 L_{ \{2\}, \{1,3,4\}}, \\
 L_{ \{3\}, \{4\}},
 L_{ \{3\}, \{1,4\}},
 L_{ \{3\}, \{1,2,4\}},
 L_{ \{4\}, \{1,3\}},
 L_{ \{4\}, \{2,3\}},
 L_{ \{4\}, \{1,2,3\}} \end{align*}
subject to relations coming from setting equal the hyperplane sections $ H_{\{1\}}$, $ H_{\{2\}}$, $ H_{\{3\}}$, $ H_{\{4\}}$, $ H_{\{1,3,4\}}$, $ H_{\{1,2,3,4\}}$, $ H_{\{3,4\}}$, and $ H_{\{2,3\}}$.  These relations show that only six of these twelve lines are linearly independent.  Hence the exceptional locus and strict transforms of lines generate a sublattice of the Picard lattice of the minimal resolutions of K3 surfaces of rank 13+6 =19.

By explicitly computing the intersection matrix for the 25 rational curves identified, we conclude that the lattice they generate has determinant $ \pm 12 $, discriminant group $ \ZZ/(12) $, and discriminant form $ 23/12 \equiv -1/12 \ ({mod} \ 2 \ZZ) $.  Hence this lattice is isomorphic to $ M_6 $.  Similar to the argument in the case for $X_{1-5}$, Fact \ref{indexfact} shows that the Picard lattice must be equal to $ M_6 $.

%7
\item[$X_{1-7}$]
Again, we compactify fibers to singular quartics.  The quartics are defined by
\begin{equation*} (x+y+z+w)(yz(x+y+z+w)+(y+z+w)(z+w)^2) - \lambda xyzw=0. \end{equation*}
The singularities are: type $\sA_1$ at $(0:1:0:-1)$, type $\sA_2 $ at $(1:0:0:0)$, $ (0:1:-1:0)$, and $(\lambda:0:-1:1)$, type $ \sA_3$ at $(0:0:1:-1)$, and type $ \sA_4$ at $(1:-1:0:0)$.  The quartics contain eight lines
\begin{align*}
L_{\{i\}, \{1,2,3,4\}} \ (1 \leq i \leq 4), \ L_{\{2\},\{3,4\}}, \  L_{\{3\}, \{ 2,4 \}},\\ L_{\{3\}, \{4\}}, L_{\{234\},*} = \{ y+z+w= x - \lambda w = 0 \}
\end{align*}
and two conics
\[ C_1 = \{ x=yz+(z+w)^2 = 0 \}, \ C_4 = \{ w =xy+(y+z)^2=0 \} \]
subject to relations coming from setting equal the hyperplane sections $ H_{\{1\}}$, $ H_{\{2\}}$, $ H_{\{3\}}$, $ H_{\{4\}}$, $ H_{\{2,3,4\}} $, and $ H_{\{1,2,3,4\}}$.  These relations show that these 10 rational curves on the quartic generate a sublattice of rank 5 in the Picard lattice.  Hence the exceptional locus and the strict transforms of these 10 curves generate a rank 19 sublattice of the Picard lattice of the minimal resolution.

Explicitly computing the intersection matrix for the curves identified shows they generate a lattice isomorphic to $M_7$.  Hence as in the cases above, the Picard lattice of the general fiber is $M_7 $.

%8
\item[$X_{1-8}$]
Compactifying to singular quartics gives singularities of type $ \sA_1 $ at
\[ (-1:0:0:1), (0:-1:0:1), (0:0:-1:1),\] \[ (1:-1:0:0), (1:0:-1:0),(0:1:-1:0), \]
and of type $\sA_2$ at
\[ (1:0:0:0), (0:1:0:0), (0:0:1:0).\]
 There are also 13 lines
\[ L_{\{i\}, \{1,2,3,4\}}, L_{\{j\}, \{4\}}, L_{\{j\},\{k,4\}} \: \mathrm{for} \: 1 \leq i \leq 4, \:  1 \leq j \neq k \leq 3, \]
subject to relations from setting equal the hyperplane sections by $ H_{\{i \}}$, $H_{\{j,4\}}$, and $H_{\{1,2,3,4\}}$ for $ 1 \leq i \leq 4 $, $ 1 \leq j \leq 3 $.  These relations show that the lattice generated by the 13 lines has rank 7.  Hence the strict transforms  of the lines and the exceptional locus generate a lattice of rank 19.

By explicitly computing the intersection matrix for the 25 rational curves identified, we conclude that the lattice they generate has determinant $ \pm 16 $, and discriminant group $ \ZZ/(16) $ with a generator $ \alpha $ such that $ q(\alpha) = 23/16 $.  Taking $ \beta = 5 \alpha $ as generator, we have $ q(\beta) = 575/16 \equiv -1/16 \ (\mathrm{mod} \ 2 \ZZ) $.  Hence this lattice is isomorphic to $ M_8 $.  In this case, Fact \ref{indexfact} shows that the Picard lattice of the generic K3 in the pencil is either $M_8$ or $M_2$.  We now use the results of \cite{Go07}, which implies that this pencil has the same variation of Hodge structure as the $M_8$ pencil, and hence a different variation from the $M_2$ variation.  Thus we conclude that this pencil must be $M_8$-polarized.

%9
\item[$X_{1-9}$]
Compactify fibers to quadrics in $ \PP^3 $ in the standard way.
These quartics have an elliptic fibration with a section coming from intersections with planes containing $ L_{ \{4\}, \{1,2,3\}} $ that gives a polarization of the Picard lattice of the minimal resolution by the rank 19 lattice $$ H \oplus A_8(-1) \oplus A_2(-1) \oplus A_1(-1) \oplus E_6(-1) .$$  By Proposition~\ref{mordellprop} there can be no sections of this fibration other than the zero section, and so the Picard lattice must be equal to
$$ H \oplus A_8(-1) \oplus A_2(-1) \oplus A_1(-1) \oplus E_6(-1)
=M_9.$$

%10
\item[$X_{1-10}$]
The quartic compactification contains lines
\[ L_{ \{1\}, \{3\}},
L_{ \{1\}, \{4 \} },
L_{ \{1\}, \{2,4\}},
L_{\{1\}, \{3,4\}},
L_{\{2\}, \{3\}},
L_{ \{2 \}, \{4\} },
L_{\{2\}, \{1,4\}},
L_{\{2\},\{3,4\}},
L_{\{3\}, \{1,4\}}, \]
\[ L_{\{3\}, \{2,4\}},
L_{ \{1,3 \}, \{4\} }
L_{ \{2,3 \}, \{4 \} },
L_{ \{1,4\}, *} = \{ x+w=(s-2)x+y =0 \},
\]\[L_{\{2,4\},*} = \{ y+w = (s-2)y+x = 0 \}
\]
and conics
\[ C_{\{3,4\}} = \{ z+w = xy +(\lambda-2)z^2 =0 \},\] \[C_{\{1,2,4\}} =\{ x+y+w= x y + (\lambda - 3) (x + y) z + z^2=0\} , \] \[ C = \{ z = (\lambda+1)w, (\lambda+1)w^2+xy=0 \},\] \[ C' = \{ z = (\lambda+1)w, 2 w (w + x + y) + \lambda w (x + y) + x y=0 \} \]
subject to relations coming from $ H_{\{i\}} $,
and singularities of type $ \sA_3 $ at $(1:0:0:0)$ and $ (0:1:0:0)$, type $ \sA_2 $ at $ (0:0:1:0) $, and type $ \sA_1 $ at $ (-1:0:0:1)$ and $ (0:-1:0:1)$.
The lines are subject to relations from setting equal $ H_{\{1\}} $, $H_{\{2\}}$, $ H_{\{3\}} $, $ H_{\{4\}} $, $ H_{\{1,3\}} $, $H_{\{2,3\}} $, $ H_{\{1,4\}} $, $ H_{\{2,4\}} $, and $ H_{\{3,4\}} $.

Computing the intersection matrix shows that the Picard lattice is $M_{10}$.

%11
\item[$X_{1-11}$]
By Proposition~\ref{proposition:CY hyperelliptic-CY}, the fibers we are interested in are birational to quartics
\begin{equation*} %\tilde{f}_{1-11} =
\{x^4 - (\lambda y -z)(x w -xy-w^2)z = 0\}. \end{equation*}
We may consider the elliptic fibration on the fibers for the family of quartics induced from intersections with planes containing $ L_{\{1\}, \{3\}} $.  Putting this fibration into Weierstrass form and applying Tate's algorithm gives a polarization by $ H \oplus E_7(-1) \oplus D_{10}(-1) $.  Comparing the Weierstrass form of this fibration to the Weierstrass form for the similar fibration for $f_{1-1}$, we conclude that the Picard lattice must be $ M_1 $.

%12
\item[$X_{1-12}$]
Compactify the fibers of the pencil for $f_{1-12}$ to quartics in $\PP^3$.
Intersecting the quartics with the pencil of planes containing $ L_{\{1\}, \{2,4\}} $, subtracting this line, and blowing up base points gives an elliptic fibration with section on the K3 surfaces.  The induced polarization is by the rank 19 lattice $ H \oplus E_6(-1) \oplus A_{11}(-1) $.  Comparing with the similar fibration for $X_{1-17}$, the generic lattice must be $ M_2 $.

%13
\item[$X_{1-13}$]
Compactify the fibers of the pencil for $f_{1-13}$ to quartics in $\PP^3$.
Intersecting the quartics with planes containing the line $ L_{\{1\}, \{4\}} $ gives an elliptic fibration that results in a polarization by $ H\oplus E_6(-1) \oplus E_6(-1) \oplus A_5(-1).
$  The Mordell--Weil group of this fibration is $ \ZZ/(3)$.  Hence, applying lattice facts, $ d(NS(X)) = \pm 6 $.   In fact, by matching parameters with the similar fibrations for $X_{1-3}$ and $X_{1-16}$, we conclude that $ NS(X) \simeq M_3 $.

%14
\item[$X_{1-14}$]
We can compactify the pencil in the toric variety $ T_{\nabla_{f_{1-14}}}$ and consider the elliptic fibration with section induced by $ (0,0,1) $.  This yields a fibration with fibers of type $ I_8 $ at $ \infty $ and $ I_1^* $ at $ t = \frac{1}{2} \left(\lambda \pm \sqrt{\lambda ^2+16}\right) $.  Hence the fibers carry a polarization by $ H \oplus A_7(-1) \oplus D_5(-1) \oplus D_5(-1) $.  Moreover, the Mordell--Weil group is isomorphic to $ \ZZ/(4) $. So, as for $X_{1-4}$, these K3 surfaces are $ M_4$-polarized.

%15
%\item
\item[$X_{1-15}$]
Compactify the fibers for $ f_{1-15} $ in $ T_{\nabla_{f_{1-15}}} $. The vector $ m = (1,1,0) $ induces an elliptic fibration on the generic compactified fiber $Y_\lambda$.  The Weierstrass form of this fibration is
\begin{equation*}
-\frac{1}{48} t^2 P(s,t) u+\frac{1}{864} t^3 \left(s^2 (-t)+4 t^2+12 t+8\right) \left( P(s,t) + 24(1+t)^2 \right)+u^3+v^2= 0,
\end{equation*}
where $ P(s,t) = s^4 t^2-8 s^2 t^3-24 s^2 t^2-16 s^2 t+16 t^4+24 t^3-8 t^2-24 t-8 $.  This fibration has a section of infinite order given by
\begin{equation*}
t \mapsto \left( -\frac{1}{12} t \left(s^2 t+8 t^2+12 t+4\right),-\frac{1}{2} s t^2 (t+1)^2 \right) = (u,v)
\end{equation*}
and a 2-torsion section given by
\begin{equation*}
t \mapsto \left( \frac{1}{12} \left(-s^2 t+4 t^2+12 t+8\right), 0 \right) = (u,v). \end{equation*}

Hence by Proposition \ref{fibprop}, the lattice $ NS(X) $ is a rank 19 lattice containing $ H \oplus D_6(-1) \oplus D_5(-1) \oplus A_5(-1) $ with the quotient $ \ZZ \oplus \ZZ/(2) $.  Matching this elliptic fibration with the one for $X_{1-5}$, we conclude that fibers for $ f_{1-15} $ are isomorphic to fibers of $f_{1-5}$, and hence these K3 surfaces are $M_5$-polarized.

%16
\item[$X_{1-16}$]
The vector $ m = (1,2,1)$ defines an elliptic fibration with section on the generic fiber $Y_\lambda$ of the Landau--Ginzburg model.  The Weierstrass form of this fibration is
\begin{multline*}-\frac{1}{48} s t^3 u \left(s^3 t+48 t+48\right)+\frac{1}{864} t^5 \left(s^6 (-t)-\right.\\
\left.-72 s^3 t-72 s^3+864 t^2+1728 t+864\right)+u^3+v^2 = 0, \end{multline*}
and there are singular fibers of types $III^* $ at $ t=0$, $ II^* $ at $ t= \infty $, and $ I_3 $ at $ t=-1$.  Hence the K3 fiber is polarized by the rank 19 lattice $ N \oplus A_2(-1) $.  By Proposition~\ref{mordellprop}, there can be no torsion sections (the discriminant groups of the two singular fibers have coprime orders), so $ NS(Y_\lambda) = K \oplus A_2(-1) $, where $ K=H \oplus E_8(-1) \oplus E_7(-1)$.

Now note that $ D(K \oplus A_2(-1)) = \ZZ/(2) \oplus \ZZ/(3) \simeq \ZZ/(6) $. If we write the isomorphism $ \ZZ/(6) \to \ZZ/(2) \oplus \ZZ/(3) $ as $ 1 \mapsto (1,1) $, then we can write the form $ q_{N \oplus A_2}: \ZZ/(6) \to \QQ/(2 \ZZ) $ by specifying
$$ q_{N \oplus A_2(-1)}(1) = 1/2 + 4/3 = 11/6 \equiv -1/6 \ (mod\ 2\ZZ). $$  Thus $ q_{N \oplus A_2} \simeq q_{\langle -6 \rangle} \simeq  q_{M_{3}} $.  Hence by Fact \ref{uniquenessfact}, $ NS(X) \simeq M_3 $.

Because it is useful for cases $X_{1-3}$ and $X_{1-13}$, we note that $ m = (1,0,0) $ gives a fibration with lattice $ H \oplus E_6(-1) \oplus E_6(-1) \oplus A_5(-1) $ plus additional sections.

%17
\item[$X_{1-17}$]
Anticanonical K3's in $ \PP^3 $ have generic Picard lattice $ \langle 4 \rangle $ generated by the hyperplane section.  We claim that the mirror family has rank 19 Picard lattice $ M_2 $.  We can see this explicitly from the toric  fibration on $T_{\nabla_{f_{1-17}}}$ defined by the normal vector $m= (1,-1,-2) $.  Restricting this fibration to the generic fiber of the Landau--Ginzburg model gives the fiber the structure of an elliptic surface with Weierstrass equation
\begin{equation*} -\frac{1}{48} \left(s^4+144\right) t^4 u+\frac{1}{864} t^5 \left(s^6 (-t)+648 s^2 t+864 t^2+864\right)+u^3+v^2 =0. \end{equation*}
Applying Tate's algorithm we see singular fibers of Kodaira type $ II^* $ at $ t= 0, \infty $, and hence the K3 surfaces are $ M $-polarized.  Moreover,
\begin{equation*} (u,v) = \left( -\frac{4 s^4+120 s^2+108}{12 s^2},  \frac{3 \left(4 s^4+30 s^2+18\right)}{2 s^3}\right) \end{equation*}
gives a section of infinite order in $ MW(\pi_m) $, enhancing the polarization to rank 19.  Since these are $M$-polarized rank 19 K3 surfaces, they must be $M_n$ polarized for some $n$, and as in the case for $X_{1-8}$, we now appeal to \cite{Go07} to conclude that the Picard lattice must be $ M_2 $.

\end{enumerate}

\begin{remark}[cf. Remark~\ref{remark: Golyshev modular}] In \cite{Go07} it is shown that the Landau--Ginzburg models for the cases under consideration have the same variation of Hodge structures (up to pullbacks) as modular variations associated to products of elliptic curves with isogeny.  Explicitly, for $X$ one of the Fano threefolds under consideration, let $ (N,d) = \left(\frac{\deg(X)}{2 \cdot \mathrm{ind}(X)^2}, \mathrm{ind}(X)\right)$.  Let $ X_0(N)+N $ denote the modular curve $ \overline{(\Gamma_0(N)+N) \backslash \mathbb{H}} $, and let $ t_N $ be a hauptmodul for $ X_0(N)+N $ such that $t_N = 0 $ at the image of the cusp $ i \infty $.  The Picard--Fuchs equation for the Landau--Ginzburg model of $X$ is now the pullback of the symmetric square of the uniformizing differential equation for $ X_0(N)+N $ by $ \lambda = t_N^d $.

We can check that the pullback part of Golyshev's theorem follows in a straightforward way from the geometry of the fibers of the Landau--Ginzburg model.

\begin{itemize}

\item {\bf Cases $X_{1-1}$ and $X_{1-11}$:} Both have polarizations by $ H \oplus E_7(-1) \oplus D_{10}(-1) $.  Clearly, since the moduli space of $ H \oplus E_7(-1)  \oplus D_{10}(-1) $ polarized  K3 surfaces is 1-dimensional, we see {\em a posteriori} that the Landau--Ginzburg models $ f_{1-1} $ and $ {f}_{1-11} $ have isomorphic K3-compactified fibers.

\item {\bf Cases $X_{1-2}$ $X_{1-12}$, and $X_{1-17}$:}    Similarly, since the moduli space of K3 surfaces polarized by $H \oplus E_6(-1) \oplus A_{11}(-1)$ is 1-dimensional, we see {\em a posteriori} that the Landau--Ginzburg models $ f_{1-2} $, $ f_{1-12} $, and $ f_{1-17} $ have isomorphic fibers.    Writing the Weierstrass forms for the elliptic fibrations that give this polarization in each case, we can match the fibrations fiberwise to check that indeed fibers for $X_{1-12}$ are given from fibers for $X_{1-2}$ by pullback $ \lambda \mapsto \lambda^2 $ , and similarly the compactification for $X_{1-17}$ is a pullback $ \lambda \mapsto \lambda^4 $ of the compactification for $X_{1-17}$.
\item {\bf $X_{1-3}$, $X_{1-13}$, and $X_{1-16}$:} Similar to the previous cases, using the polarizations by $ H \oplus E_6(-1) \oplus E_6(-1) \oplus A_5(-1) $.
\item {\bf Cases $X_{1-4}$ and $X_{1-14}$:} Similar to the previous cases, using the polarizations by $ H \oplus A_7(-1) \oplus D_5(-1) \oplus D_5(-1) $.
\item {\bf Cases $X_{1-5}$ and $X_{1-15}$:}  In this case, the pullback was used to derive the polynomial $ f_{1-15} $.
\end{itemize}

\end{remark}

\part{Katzarkov--Kontsevich--Pantev conjectures}
\label{part:KKP}
This part is based on papers~\cite{KKP17},~\cite{LP18}, and~\cite{CP18}.
We study here Katzarkov--Kontsevich--Pantev conjectures about Hodge numbers of Landau--Ginzburg models
and prove them in the cases of dimension 2 and 3.
\label{section: KKP}

\section{Formulation}

Let us recall some numerical conjectures from~\cite{KKP17} which are supposed
to follow from the conjectural Homological Mirror Symmetry between Fano manifolds and Landau--Ginzburg models.

\begin{definition} \label{def-1}\emph{A Landau--Ginzburg model} is a pair $(Y,w)$, where

\begin{enumerate}
\item $Y$ is a smooth complex quasi-projective variety with trivial canonical bundle $K_Y$;

\item $w\colon Y\to \bbA ^1$ is a morphism with a compact critical locus $crit(w)\subset Y$.
\end{enumerate}

\end{definition}

\begin{remark}
Note that there are no conditions on singularities of fibers.
\end{remark}

Following~\cite{KKP17} we assume that there exists a {\it tame} compactification of the Landau--Ginzburg model as defined below (cf. Definition~\ref{definition: log CY}).

\begin{definition} \label{def-3}\emph{A tame compactified Landau--Ginzburg model} is the data $((Z,f),D_Z)$, where

\begin{enumerate}
\item $Z$ is a smooth projective variety and $f\colon Z\to \bbP ^1$ is a flat morphism.

\item $D_Z=(\cup _i D^h_i)\cup (\cup _jD_j^v)$ is a reduced normal crossings divisor such that

\begin{itemize}
\item[(i)] $D^v=\cup _jD^v_j$ is a scheme-theoretical pole divisor of $f$, i.e. $f^{-1}(\infty)=D^v$. In particular $ord _{D^v_j}(f)=-1$ for all $j$;

\item[(ii)] each component $D_i^h$ of $D^h=\cup _iD^h_i$ is smooth and horizontal for $f$, i.e. $f\vert _{D^h_i}$ is a flat morphism;

\item[(iii)] The critical locus $crit(f)\subset Z$ does not intersect $D^h$.

\end{itemize}

\item $D_Z$ is an anticanonical divisor on $Z$.

\noindent One says that $((Z,f),D_Z)$ is \emph{a compactification of the Landau--Ginzburg model}
$(Y,w)$ if in addition the following holds:

\item $Y=Z\setminus D_Z$, $f\vert _Y=w$.
\end{enumerate}

\end{definition}

\begin{remark} In~\cite{KKP17} the authors require in above definitions an additional choice of
compatible holomorphic volume forms on $Z$ and $Y$. Since these forms will play no role
in this paper we omitted them.
\end{remark}

Assume that we are given a Landau--Ginzburg model $(Y,w)$ with a tame compactification~$((Z,f),D_Z)$ as above.
We denote by $n=\dim Y=\dim Z$ the (complex) dimension of~$Y$ and $Z$. Choose a point $b\in  \bbA ^1$ which is near $\infty$ and such that the fiber $Y_b=w^{-1}(b)\subset Y$ is smooth.
In~\cite{KKP17} the authors define geometrically three sets of what they call ``Hodge numbers''  $i^{p,q}(Y,w)$, $h^{p,q}(Y,w)$, $f^{p,q}(Y,w)$. Let us recall the definitions.

\subsection{The numbers $f^{p,q}(Y,w)$}\label{subs-def-fpq}
Recall the definition of the logarithmic de Rham complex~$\Omega^* _Z(log\,  D_Z)$.
Namely, $\Omega ^s _Z(log\, D_Z)=\wedge ^s \Omega^1 _Z(log\,D_Z )$ and $\Omega^1 _Z(log\,D_Z )$ is a locally free~$\mbox{$\cO _Z$-module}$ generated locally by
$$\frac{dz_1}{z_1},\ldots,\frac{dz_k}{z_k},dz_{k+1},\ldots,dz_n$$
if $z_1\cdot \ldots\cdot z_k=0$ is a local equation of the divisor $D_Z$. Hence in particular $\Omega ^0 _Z(log\, D_Z)=\cO _Z$.

The numbers $f^{p,q}(Y,w)$ are defined using the  subcomplex $\Omega^* _Z(log\, D_Z ,f)\subset \Omega^* _Z(log\, D_Z)$ of~$f$-{\it adapted forms}, which we recall next.

\begin{definition}[{\cite[Definition 2.11]{KKP17}}] For each $a\geq 0$ define \emph{a sheaf $\Omega ^a _Z(log\, D_Z ,f)$ of $f$-adapted logarithmic forms} as a subsheaf of $\Omega ^a _Z(log\, D_Z)$ consisting of forms which stay logarithmic after multiplication by $df$. Thus
$$\Omega ^a _Z(log\, D_Z ,f)=\{\alpha \in \Omega ^a _Z(log\, D_Z)\ \vert \ df\wedge \alpha \in \Omega ^{a+1} _Z(log\, D_Z )\},$$
where one considers $f$ as a meromorphic function on $Z$ and $df$ is viewed as a meromorphic~$\mbox{1-form}$.
\end{definition}

\begin{definition}[{\cite[Definition 3.1]{KKP17}}] \emph{The Landau--Ginzburg Hodge numbers} $f^{p,q}(Y,w)$ are defined as follows:
$$f^{p,q}(Y,w)=\dim H^p(Z,\Omega ^q _Z(log\, D_Z ,f)).$$
\end{definition}

\subsection{The numbers $h^{p,q}(Y,w)$}
Let $N\colon V\to V$ be a nilpotent operator on a finite dimensional vector space $V$ such that $N^{m+1}=0$. Recall that this data defines a canonical (monodromy) {\it weight filtration centered at $m$}, $W=W_* (N,m)$ of $V$
$$0\subset W_0(N,m)\subset W_1(N,m)\subset \ldots\subset W_{2m-1}(N,m)\subset W_{2m}(N,m)=V$$
with the properties
\begin{enumerate}
\item $N(W_i)\subset W_{i-2}$,
\item the map $N^l\colon gr ^{W,m}_{m+l}V\to gr ^{W,m}_{m-l}V$ is an isomorphism for all $l\geq 0$.
\end{enumerate}

Let $S^1\simeq C\subset \bbP ^1$ be a smooth loop passing through the point $b$ that goes once around $\infty$ in the counter clockwise direction in such a way that there are no singular points of $w$ on or inside~$C$. It gives
the monodromy transformation
\begin{equation*} \label{monodromy-1}T\colon H^* (Y_b)\to H^*(Y_b)
\end{equation*}
and also the corresponding monodromy transformation on the relative cohomology
\begin{equation}\label{monodromy-2} T\colon H^* (Y,Y_b)\to H^* (Y,Y_b)
\end{equation}
in such a way that the sequence
\begin{equation*}\label{long-exact-seq}
\ldots\to H^m(Y,Y_b)\to H^m(Y)\to H^m(Y_b)\to H^{m+1}(Y,Y_b)\to \ldots
\end{equation*}
 is $T$-equivariant, where $T$ acts
trivially on $H^* (Y)$. (See Subsection~\ref{subsection:monodromy} for the construction and the discussion of the
monodromy transformation $T\colon H^* (Y,Y_b)\to H^* (Y,Y_b)$.) Since we assume that the infinite fiber $f^{-1}(\infty)\subset Z$ is a reduced divisor with normal crossings, by Griffiths--Landman--Grothendieck Theorem  (see~\cite{Ka70}) the operator $T\colon H^m (Y_b)\to H^m (Y_b)$ is unipotent and $\left(T-\id\right)^{m+1}=0$. It follows that the transformation \eqref{monodromy-2} is also unipotent. Denote by $N$ the logarithm of the transformation \eqref{monodromy-2}, which is therefore a nilpotent operator on $H^* (Y,Y_b)$.
One has $N^{m+1}=0$.

\begin{definition}[{\cite[Definition 7]{LP18}}] \label{def-Fano-type} We say that the Landau--Ginzburg model $(Y,w)$ is {\it of Fano type} if the operator~$N$ on the relative cohomology $H^{n+a} (Y,Y_b)$ has the following properties:
\begin{enumerate}
\item $N^{n-|a|}\neq 0$,

\item $N^{n-|a|+1}=0$.
\end{enumerate}
\end{definition}

The above definition is motivated by the expectation that
the Landau--Ginzburg model of Fano type usually appears as a mirror of a projective Fano
manifold %$X$
(see Subsection~\ref{subsection: Hodge numbers conjectures}).

\begin{definition}[{see~\cite[Definition 3.2]{KKP17} and~\cite[Definition 8]{LP18}}]
\label{definition:hpq} Assume that $(Y,w)$ is a Landau--Ginzburg model of Fano type. Consider the relative cohomology $H^*(Y,Y_b)$ with the nilpotent operator $N$ and the induced canonical filtration $W$. \emph{The Landau--Ginzburg Hodge numbers} $h^{p,q}(Y,w)$ are defined as follows:
$$h^{p,n-q}(Y,w)=\dim gr _{2(n-p)}^{W,n-a}H^{n+p-q}(Y,Y_b)\ \ \text{if $a=p-q\geq 0$},$$
$$h^{p,n-q}(Y,w)=\dim gr _{2(n-q)}^{W,n+a}H^{n+p-q}(Y,Y_b)\ \ \text{if $a=p-q< 0$}.$$

\end{definition}

\begin{remark}
\label{remark:correction}
Our Definition~\ref{definition:hpq} differs from~\cite[Definition~3.2]{KKP17}
\begin{equation}\label{wrong-def}h^{p,q}(Y,w)=\dim gr _{p}^{W,p+q}H^{p+q}(Y,Y_b)
\end{equation}
by the indices of the grading. The equation \eqref{wrong-def} seems not to be what the authors had in mind. For example according to the equation \eqref{wrong-def} the index $p$ is allowed to vary from $0$ to $2n$ and $q$ is allowed to be negative (details see in Subsection~\ref{subsection: Hodge numbers conjectures}).
\end{remark}

\subsection{The numbers $i^{p,q}(Y,w)$}
\label{subsection: i-numbers}
Recall that for each $\lambda \in \bbA ^1$ one has the corresponding sheaf~$\phi _{w-\lambda}\bbC _Y$ of vanishing cycles for the fiber $Y_\lambda$. The sheaf $\phi _{w-\lambda}\bbC _Y$ is supported on the fiber $Y_\lambda$ and is equal to zero if $\lambda $ is not a critical value of $w$. From the works of Schmid, Steenbrink, and Saito
it is classically known that the constructible complex $\phi _{w-\lambda}\bbC _Y$ carries a structure of a mixed Hodge module and so its hypercohomology inherits a mixed Hodge structure. For a mixed Hodge module $S$ we will denote by $i^{p,q}S$ the $(p,q)$-Hodge numbers of the $p+q$ weight graded piece $gr ^W_{p+q}S$.

\begin{definition}[{\cite[Definition 3.4]{KKP17}}]

\begin{enumerate}
\item
Assume that the horizontal divisor $D^h\subset Z$ is empty, i.e. assume that the map $w\colon Y\to \bbA ^1$ is proper. Then \emph{the Landau--Ginzburg Hodge numbers}~$i^{p,q}(Y,w)$ are defined as follows:
$$
i^{p,q}(Y,w)=\sum _{\lambda \in \bbA ^1}\sum _ki^{p,q+k}\bbH ^{p+q-1}(Y_\lambda ,
\phi _{w-\lambda}\bbC _Y).
$$

\item
In the general case denote by $j\colon Y\hookrightarrow Z$ the open embedding and define
similarly
$$
i^{p,q}(Y,w)=\sum _{\lambda \in \bbA ^1}\sum _ki^{p,q+k}\bbH ^{p+q-1}(Y_\lambda ,
\phi _{w-\lambda}{\bf R}j_{*}\bbC _Y).
$$
\end{enumerate}
\end{definition}

\subsection{Conjectures}
\label{subsection: Hodge numbers conjectures}
It is proved in~\cite{KKP17} that for every $m$ the above numbers satisfy the equalities
\begin{equation} \label{sum-of-hodge}
\dim H^m(Y,Y_b;\bbC)=\sum _{p+q=m}i^{p,q}(Y,w)=\sum _{p+q=m}f^{p,q}(Y,w).
\end{equation}
The authors state several conjectures which together refine the
equalities \eqref{sum-of-hodge}. The next is a modification of~\cite[Conjecture $3.6$]{KKP17}, see Remark~\ref{remark:correction}.

\begin{conjecture}\label{conj-1} Assume that $(Y,w)$ is a Landau--Ginzburg model of Fano type. Then for every $p,q$ there are  equalities
\begin{equation*}\label{equat-of-indiv-hodge-numbers}
h^{p,q}(Y,w)=f^{p,q}(Y,w)=i^{p,q}(Y,w).
\end{equation*}
\end{conjecture}

The Landau--Ginzburg model $(Y,w)$ of Fano type (together with a tame compactification) typically arises as a mirror of a
projective Fano manifold $X$, $\dim X=\dim Y$.

The following is~\cite[Conjecture $3.7$]{KKP17}, see Remark~\ref{remark:correction}.

\begin{conjecture} \label{conj-2} In the above mirror situation for each $p,q$ we have the equality
$$f^{p,q}(Y,w)=h^{p,n-q}(X),$$
where $h^{p,q}(X)$'s are the usual Hodge numbers for $X$.
\end{conjecture}

We refer the interested reader to~\cite{KKP17} for a detailed description of the motivation for Conjectures \ref{conj-1} and \ref{conj-2}. Basically the motivation comes from Homological Mirror Symmetry,  Hochschild homology identifications, and identification of the monodromy operator with the Serre functor. Namely, assume that the Landau--Ginzburg model $(Y,w)$ as above (together with a tame compactification) is of Fano type and is a mirror of a projective Fano manifold $X$, $\dim X=\dim Y$. Then by Homological Mirror Symmetry conjecture one expects an equivalence of categories
\begin{equation}\label{eq-of-cat}D^b(coh\ X)\simeq FS((Y,w),\omega_Y),
\end{equation}
where %$D^b(coh\ X)$ is the bounded derived category of coherent sheaves on $X$ and
$FS((Y,w),\omega_Y)$ is the Fukaya--Seidel category of the Landau--Ginzburg model $(Y,w)$ with an appropriate  symplectic form $\omega_Y$. This equivalence induces for each $a$ an isomorphism of the Hochschild homology spaces
\begin{equation*}\label{eq-of-hoch-hom}
HH_a(D^b(coh\ X))\simeq HH_a(FS((Y,w),\omega_Y)).
\end{equation*}
It is known that
\begin{equation}\label{eq-of-hoch-derham}HH_a(D^b(coh\ X))\simeq \bigoplus _{p-q=a}H^p(X,\Omega _X^q)
\end{equation}
and it is expected that
\begin{equation}\label{eq-of-hoch-fuk}HH_a(FS((Y,w),\omega_Y))\simeq H^{n+a}(Y,Y_b).
\end{equation}
The equivalence \eqref{eq-of-cat} and isomorphisms \eqref{eq-of-hoch-derham}, \eqref{eq-of-hoch-fuk} suggest an isomorphism
\begin{equation*}\label{summary-isom}
H^{n+a}(Y,Y_b)\simeq\bigoplus _{p-q=a}H^p(X,\Omega _X^q).
\end{equation*}
Moreover, the equivalence \eqref{eq-of-cat} identifies the Serre functors $S_X$ and $S_Y$ on the two categories. The
functor $S_X$ acts on the cohomology $H^* (X)$ and the logarithm of this operator is equal (up to a sign) to
the cup-product with $c_1(K_X)$. Since $X$ is Fano, the operator $c_1(K_X)\cup (\cdot )$ is a Lefschetz operator on the
space
$$\bigoplus _{p-q=a}H^p(X,\Omega _X^q)$$
for each $a$.
On the other hand, the Serre functor $S_Y$ induces an operator on the space~$H^{n+a}(Y,Y_b)$ which is the inverse
of the monodromy transformation $T$. This suggests that the weight filtration for the nilpotent operator $c_1(K_X)\cup (\cdot )$ on
the space $\bigoplus _{p-q=a}H^p(X,\Omega _X^q)$ should coincide with the similar filtration for the logarithm $N$ of the operator $T$ on
$H^{n+a}(Y,Y_b)$.
First notice that the operator
$c_1(K_X)\cup (\cdot )$ on
the space $\bigoplus _{p-q=a}H^p(X,\Omega _X^q)$ satisfies
$(c_1(K_X)\cup (\cdot ))^{n-|a|}\neq 0$ by the Hard Lefschetz theorem and~$\left(c_1(K_X)\cup (\cdot )\right)^{n-|a|+1}= 0$. This explains our Definition \ref{def-Fano-type}.
Moreover, the induced filtration $W$ on $\bigoplus _{p-q=a}H^p(X,\Omega _X^q)$ has the properties:
$$h^{p,q}(X)=gr^{W,n-a}_{2(n-p)}\left[\bigoplus _{p-q=a}H^p(X,\Omega _X^q)\right]\ \ \text{if $a\geq 0$}$$
and
$$h^{p,q}(X)=gr^{W,n+a}_{2(n-q)}\left[\bigoplus _{p-q=a}H^p(X,\Omega _X^q)\right]\ \ \text{if $a< 0$}.$$
Thus one expects the equality of Hodge numbers
$$h^{p,n-q}(Y,w)=h^{p,q}(X),$$
which is a combination of the above conjectures.

\section{Del Pezzo surfaces}
\label{section: KKP for surfaces}
Mirror symmetry conjecture we are interested in this section is Homological Mirror Symmetry conjecture.
It (more precise, its half) was proven for del Pezzo surfaces in~\cite{AKO06}.
A Landau--Ginzburg model for del Pezzo surface of degree $d$ is constructed there as a pencil of elliptic curves
whose fiber over infinity is a wheel of $12-d$ curves, while the rest singular fibers are $d$ fibers having
a single ordinary double point (node). Such pencil is a Landau--Ginzburg model for the del Pezzo surface
with a general symplectic form on the model. However a Fukaya--Seidel category is invariant under deformations
of pencils, so to study mirror symmetry it is enough to consider the case of a general form.
Moreover, the results of the section do not depend on singular fibers away from infinity.
Finally note that Landau--Ginzburg models studied here correspond to all del Pezzo surfaces, not only
of degree greater than $2$ as in Part~\ref{part: del Pezzo surfaces}.

Following~\cite{LP18}, we correct a bit and prove Conjectures~\ref{conj-1} and~\ref{conj-2} for del Pezzo surfaces.

Consider tame compactified Landau--Ginzburg model $(Z,f)$ of dimension $2$. More precisely, consider a rational elliptic surface $f\colon Z\to \bbP ^1$  with $f^{-1}(\infty)$ being a reduced divisor which is a wheel of $d$ rational curves, $1\leq d\leq 9$ (it is a nodal rational curve if $d=1$). In this case the horizontal divisor $D^h$ is empty, so $D=D^v$. In the paper \cite{AKO06} it is proved that the corresponding Landau--Ginzburg model $(Y,w)$ appears as a (homological) mirror of a del Pezzo surface $S_d$ of degree $d$. The authors also establish Homological Mirror Symmetry for the case $d=0$: in this case $f^{-1}(\infty)$ is a smooth elliptic curve and $(Y,w)$ is mirror to the blowup $S_0$ of $\bbP ^2$ in $9$ points of intersection of two cubic curves. Note that such~$S_0$ is not Fano, hence one expects that the corresponding Landau--Ginzburg model $(Y,w)$ is not of Fano type. We confirm this prediction. The next theorem summarizes the main results of this section.

\begin{theorem}[{\cite[Theorem 11]{LP18}}]
\label{theorem:main-Hodge-2} Let $f\colon Z\to \bbP ^1$ be an elliptic surface with the reduced infinite fiber
$D=f^{-1}(\infty)$ which is a wheel of $d$ rational curves for $1\leq d\leq 9$ or is a smooth elliptic curve for $d=0$. We assume that $f$ has a section. As before put $(Y,w)=(Z\setminus D,f\vert _{Z\setminus D})$.

\begin{enumerate}
\item[(i)] If $1\leq d\leq 9$, then the Landau--Ginzburg model $(Y,w)$ is of Fano type and there are equalities of Hodge numbers $$f^{p,q}(Y,w)=h^{p,q}(Y,w).$$

\item[(ii)] Let $1\leq d\leq 9$ and let $X$ be a del Pezzo surface
which is a mirror
in the sense of~\cite{AKO06} to the Landau--Ginzburg model $(Y,w)$. There are equalities of Hodge numbers $$f^{p,q}(Y,w)=h^{p,2-q}(X).$$

\item[(iii)] If $d=0$, then $(Y,w)$ is not of Fano type.
\end{enumerate}
\end{theorem}

The proof of Theorem \ref{theorem:main-Hodge-2} is contained in~Proposition~\ref{proposition:h-numbers}, Proposition~\ref{proposition:f-numbers}, and Remark~\ref{remark: Hodge for del Pezzo}.

Thus Conjecture \ref{conj-1} about the numbers $f^{p,q}(Y,w)$, $h^{p,q}(Y,w)$ and Conjecture~\ref{conj-2} hold in case~$(Y,w)$ is of Fano type ($1\leq d\leq 9$). We will also show that in the context of Theorem~\ref{theorem:main-Hodge-2} the numbers
$i^{p,q}(Y,w)$ are {\it not} equal to the numbers $f^{p,q}(Y,w)$ (or to the numbers
$h^{p,q}(Y,w)$, or $h^{p,2-q}(X)$), therefore providing a
counterexample to Conjecture \ref{conj-1}, see Remark~\ref{remark:i-numbers}.
We do not know how to define the ``correct'' numbers $i^{p,q}(Y,w)$, which would make Conjecture~\ref{conj-1} true.

\subsection{Monodromy action on relative cohomology}
\label{subsection:monodromy}

Let $V$ be a smooth complex algebraic variety
of dimension $n$ with a proper morphism~$w\colon V\to \bbC$. Let $b \in \bbC $ be a regular value of $w$.
In this section we construct the monodromy action on the relative homology $H_* (V,V_b)$, which by duality will induce the desired action on~$H^* (V,V_b)$.

Let $C\simeq S^1\subset \bbP ^1$ be a smooth loop passing through the point $b$ that goes once around the $\infty$ in the counter clockwise direction in such a way that there are no singular values of $w$ on or inside $C$. Denote by $M$ the preimage $w^{-1}(C)\subset Y$. Then $M$ is a compact oriented smooth manifold which contains the fiber $V_b$.
The (real) dimensions of $M$ and $V_b$ are $2n-1$ and~$2n-2$ respectively.
By Ehresmann's Lemma the map $w\colon M\to C$ is a locally trivial fibration of smooth manifolds with the fibers diffeomorphic to $V_b$.
Hence there exists a diffeomorphism~$T\colon V_b\to V_b$ such that $M$ is diffeomorphic to the quotient
$$M=V_b\times [0,1]/\{ (a,0)=(T(a),1)\ \text{for all $a\in V_b$}\}.$$
For the pair $(M,V_b)$ we have the corresponding long exact homology sequence
\begin{equation}\label{les}\ldots\to H_i(V_b)\stackrel{\alpha _i}{\to} H_i(M)\stackrel{\beta _i}{\to} H_i(M,V_b)\stackrel{\partial _i}{\to} H_{i-1}(V_b)\to \ldots
\end{equation}
The diffeomorphism $T\colon V_b\to V_b$ induces an automorphism~$T\colon H_i(V_b)\to H_i(V_b)$ for each~$i$.

\begin{lemma} \label{lemma-formula} For each $i\geq 0$, there exists a homomorphism $L_i\colon H_i(V_b)\to H_{i+1}(M,V_b)$ such that for all $x\in H_i(V_b)$ we have
$$\partial _{i+1}L_i(x)=T(x)-x.$$
\end{lemma}

\begin{proof} Let $z$ be an $i$-dimensional cycle in $V_b$. Consider the $(i+1)$-dimensional relative cycle~$z\times [0,1]$ in $(V_b\times [0,1],V_b\times \{0\}\cup V_b\times \{1\})$ with boundary $z\times \{1\}-z\times \{0\}$.
Its image~$L_i(z)$ in $M$ is a relative $(i+1)$-cycle with boundary $T(z)-z$ in $V_b$.
This construction yields the required homomorphism $L_i\colon H_i(V_b)\to H_{i+1}(M,V_b)$. Given $x\in H_i(V_b)$ the
%equality $$\partial _{i+1}L_i(x)=T(x)-x$$
assertion of the lemma
is clear from the construction.
\end{proof}

\begin{proposition}[{\cite[Proposition 13]{LP18}}]\label{injective} The map $L_i\colon H_i(V_b)\to H_{i+1}(M,V_b)$ is injective for each $i\geq 0$ .
\end{proposition}

\begin{definition} For each $i$ define the endomorphism $T\colon H_{i}(M,V_b)\to H_{i}(M,V_b)$ as~$\mbox{$T=\id +L_{i-1}\partial _{i}$}$ and the endomorphism $T\colon H_i(M)\to H_i(M)$ as $T=\id $.
(In particular~$T=\id $ on $H_0(M,V_b)$.)
\end{definition}

The inclusion of the pairs $(M,V_b)\subset (V,V_b)$ induces a morphism of the
homology sequences
$$\begin{array}{ccccccc}

\ldots \to & H_i(M) & \to & H_i(M,V_b)& \stackrel{\partial _i}{\to } & H_{i-1}(V_b) & \to \ldots\\
         & \downarrow & & \downarrow \gamma _i &  & \parallel & \\
\ldots\to & H_i(V) & \to & H_i(V,V_b) & \stackrel{\partial _i}{\to } & H_{i-1}(V_b)  & \to  \ldots
\end{array}
$$

\begin{definition} Let us define for each $i\geq 0$ the endomorphism $T\colon H_i(V,V_b)\to H_i(V,V_b)$ as the composition
$$T(y)=y+\gamma _iL_{i-1}\partial _i(y)$$
for $y\in H_i(V,V_b)$. In particular, $T=\id $ on $H_0(V,V_b)$. We also define $T\colon H_i(V)\to H_i(V)$ to be the identity.

By duality this defines the operators $T$ on the cohomology $H^i(V_b)$, $H^i(V,V_b)$, $H^i(V)$.
\end{definition}

\begin{corollary} The sequence
\begin{equation*}\label{seq2}
\ldots\to H_i(V) \to  H_i(V,V_b) \to   H_{i-1}(V_b)   \to  \ldots
\end{equation*}
is compatible with the endomorphisms $T$. Hence also the dual cohomology sequence
\begin{equation*}\label{seq3}
\ldots\to H^{i-1}(V_b) \to  H^i(V,V_b)  \to   H^i(V)   \to  \ldots
\end{equation*}
is compatible with $T$.
\end{corollary}

\begin{proof} This follows directly from the definition of the operators $T$ together with
the formula in Lemma \ref{lemma-formula}.
\end{proof}

\begin{proposition}[{\cite[Proposition 18]{LP18}}] \label{prop-sum}
\begin{itemize}
\item[(i)] Assume that the morphism
$$\gamma _i\colon H_i(M,V_b)\to H_i(V,V_b)$$
is injective. Then the image of the morphism $H_i(V)\to H_i(V,V_b)$ is the space $H_i(V,V_b)^T$
of $T$-invariants.

\item[(ii)] If $H^{2n-i-1}(V)=0$, then the map $H_i(M,V_b)\to H_i(V,V_b)$ is injective.
Hence by (i) the image of the morphism $H_i(V)\to H_i(V,V_b)$ is the space $H_i(V,V_b)^T$
of \mbox{$T$-invariants}.
\end{itemize}
\end{proposition}

\subsection{Topology of rational elliptic surfaces}
\label{subsection:topology}
Now we use the notation of the beginning of the section for the special case which we will consider in the rest of the section.
Fix a number $0\leq d\leq 9$ and let $f\colon Z\to \PP^1$ be a rational elliptic surface such that $D=D^v=f^{-1}(\infty)$ is a wheel $I_d$
of $d$ smooth rational curves for $d\geq 2$, a rational curve with one node $I_1$ for $d=1$, and a smooth elliptic curve $I_0$ for $d=0$.
Assume in addition that there exists a section $\PP^1\to E\subset Z$.
Recall that $Y=Z\setminus D$.

Since $Z$ is rational, $\chi (\cO _Z)=1$. One has $-K_Z=D$, see, for instance,~\cite[\S10.2]{ISh89}. Hence~$c_1^2(Z)=0$, so by Noether's formula the topological Euler characteristic of $Z$ is equal to $12$. This means that
$$
h^i(Z)=\left\{
                     \begin{array}{ll}
                       1, & i=0,4; \\
                       10, & i=2; \\
                       0, & \hbox{otherwise.}
                     \end{array}
                   \right.
$$
By the
adjunction formula $\left(K_Z+E\right)\cdot E=2g(E)-2=-2$, so
$E^2=-1$.

\begin{lemma}
\label{lemma:C_D}
\begin{itemize}
\item[(i)]
If $d=0$, then
$$
h^i(D)=\left\{
                     \begin{array}{ll}
                       1, & i=0,2; \\
                       2, & i=1; \\
                       0, & \hbox{otherwise.}
                     \end{array}
                   \right.
$$
\item[(ii)]
If $d>0$, then
$$
h^i(D)=\left\{
                     \begin{array}{ll}
                       1, & i=0,1; \\
                       d, & i=2; \\
                       0, & \hbox{otherwise.}
                     \end{array}
                   \right.
$$
\end{itemize}
\end{lemma}

\begin{proof}

The part (i) is clear. Prove the part (ii).
Let $p_1,\ldots p_d$ be the intersection points of the components of $D$.
Let $\pi\colon \widetilde{D}\to D$ be the normalization. Then $\widetilde{D}$ is a disjoint union of $d$ copies of $\PP^1$.
Consider an exact sequence of sheaves on $D$
\begin{equation}\label{resol-seq}
0\to \CC_{D}\to \pi_*\pi^*\CC_{D}\to \oplus_{i=1}^d \CC_{p_i}\to 0,
\end{equation}
where $\CC_{p_i}$ is a skyscraper sheaf supported at $p_i$.
Notice that
$$
\dim H^i(D,\pi_*\pi^*\CC_{D})=\dim H^i(\widetilde{D})=\left\{
                                                    \begin{array}{ll}
                                                      d, & i=0,2; \\
                                                      0, & i=1.
                                                    \end{array}
                                                  \right.
$$
Notice also that $H^0(D,\CC_{D})=\CC$ and the map $H^0(D,\CC_{D})\to H^0(D,\pi_*\pi^*\CC_{D})$ is injective.
The lemma now follows from the long exact sequence of cohomology applied to the short exact sequence~\eqref{resol-seq}.
\end{proof}

\begin{lemma}[{\cite[Lemma 20]{LP18}}]
\label{lemma:ZtoD}
The restriction map $s\colon H^2(Z)\to H^2(D)$ is surjective.
\end{lemma}

Next we compute the cohomology $H_c^i(Y)$ of $Y$ with compact support.

\begin{lemma}[{\cite[Lemma 21]{LP18}}]
\label{lemma:C_Y}
The following equalities hold.
$$h^i_c(Y)=
h^i(Z,j_!\CC_Y)=\left\{
                     \begin{array}{ll}
                       0, & i=0,1,3; \\
                       11-d, & i=2; \\
                       1, & i=4.
                     \end{array}
                   \right.
$$
\end{lemma}

\begin{proof}[Idea of the proof]
This follows from the long exact sequence of cohomology $H^*(Z,-)$ for the short exact sequence
\begin{equation*}
0\to j_!\CC_{Y}\to \CC_{Z}\to \CC_{D}\to 0.
\qedhere
\end{equation*}
\end{proof}

\begin{corollary}\label{coho-of-y}
By Poincare duality for $Y$ one has
\begin{equation*}h^i(Y)=\left\{ \begin{array} {rl}  1, & \text{if $i=0$;} \\
                       11-d, & \text{if $i=2$;}\\
                        0, & \text{if $i=1,3,4$.}
                        \end{array}
                        \right.
\end{equation*}
\end{corollary}

\subsection{Landau--Ginzburg Hodge numbers for rational elliptic surfaces}
\label{subsection:LG-Hodge-numbers-for-surfaces}

\subsubsection{The numbers $h^{p,q}(Y,w)$}
We keep the notation of Subsection~\ref{subsection:topology}.

Consider the long exact sequence of homology
\begin{equation*} \label{seq4} \ldots\to  H_2(Y)  \to  H_2(Y,Y_b)  \to   H_{1}(Y_b)   \to  \ldots
\end{equation*}
Recall that there is a compatible action of the monodromy $T$ on each term of this sequence as explained in Subsection~\ref{subsection:monodromy}.

\begin{corollary}\label{cor-of-general} The image of the map $H_2(Y)  \to  H_2(Y,Y_b)$ coincides with the space $H_2(Y,Y_b)^T$
of $T$-invariants.
\end{corollary}

\begin{proof} In the notation of Proposition~\ref{prop-sum} we have $n=2$, $i=2$, and by Corollary \ref{coho-of-y} we have $H^{2n-i-1}(Y)=H^1(Y)=0$. Hence the assertion follows from
 Proposition \ref{prop-sum}(ii).
\end{proof}

\begin{proposition}
\label{proposition:h-numbers}
\begin{enumerate}
\item[(i)]
We have \begin{equation}\label{fact-coh}H^k(Y,Y_b)=
\left\{\begin{array}{rl}
                       \CC^{12-d}, & k=2; \\
                       0, & \hbox{otherwise.}
\end{array}\right.
\end{equation}

\item[(ii)] For $d>0$ the Landau--Ginzburg model $(Y,w)$ is of Fano type and
\begin{equation}\label{eq-hpq} h^{p,q}(Y,w)=\left\{\begin{array}{rl}
                       1, & (p,q)=(0,2),(2,0); \\
                       10-d, & (p,q)=(1,1); \\
                       0, & \hbox{otherwise.}
\end{array}\right.
\end{equation}

\item[(iii)] For $d=0$ the Landau--Ginzburg model $(Y,w)$ is not of Fano type. More precisely, the $T$-action on $H^2(Y,Y_b)$ has 2 Jordan blocks of size 2 and 8 blocks of size 1.
(So no blocks of size 3).
\end{enumerate}
\end{proposition}

This proposition proves Theorem~\ref{theorem:main-Hodge-2}(iii) and computes the right hand side of the equality of Theorem~\ref{theorem:main-Hodge-2}(i).

The proof of the proposition will occupy the rest of this subsection.

\begin{lemma} \label{lemma-surj} The restriction map $H^2(Y)\to H^2(Y_b)$ is surjective. Hence the map $\mbox{$H_2(Y_b)\to H_2(Y)$}$
is injective.
\end{lemma}

\begin{proof} Since $Y_b$ is a smooth projective curve, $H^2(Y_b)$ has dimension one and is spanned by the first Chern class $c_1(L)$ of any ample line bundle $L$ on $Y_b$. It suffices to take any ample line bundle $M$ on $Y$, so that its restriction
$L=M\vert _{Y_b}$ is also ample and $c_1(M)\in H^2(Y)$ restricts to $c_1(L)\in H^2(Y_b)$.
\end{proof}

The equation \eqref{fact-coh} now follows from the long exact sequence of cohomology
$$\ldots\to H^i(Y,Y_b)\to H^i(Y)\to H^i(Y_b)\to \ldots $$ using Corollary~\ref{coho-of-y}, the fact that $Y_b$ is an elliptic curve, and
Lemma \ref{lemma-surj}. This proves part (i) of the proposition.

To prove parts (ii) and (iii) it remains to understand the action of the monodromy $T$ on~$H_2(Y,Y_b)$.

Consider the part of the long exact sequence of homology
\begin{equation*}\label{piece-of-long-homology-1}
H_3(Y,Y_b)\to H_2(Y_b)\to H_2(Y)\to H_2(Y,Y_b)\to H_1(Y_b)\to H_1(Y).
\end{equation*}
We know that the map $H_2(Y_b)\to H_2(Y)$ is injective and that $H_1(Y)=H^1(Y)^\vee=0$. Hence the sequence
\begin{equation}\label{piece-of-long-homology-2}
0\to H_2(Y_b)\to H_2(Y)\to H_2(Y,Y_b)\to H_1(Y_b)\to 0
\end{equation}
is also exact. We have $H_2(Y_b)=\CC$, $H_1(Y_b)=\CC ^2$, $H_2(Y)=\CC^{11-d}$, hence the sequence
\eqref{piece-of-long-homology-2} is isomorphic to
\begin{equation*}\label{piece-of-long-homology-3}
0\to \bbC \to \bbC ^{11-d}\to \bbC ^{12-d}\to \bbC ^{2} \to 0.
\end{equation*}

These sequences are $T$-equivariant, where $T$ acts trivially on $H_2(Y_b)$ and $H_2(Y)$. By Landman's theorem
$T$ acts unipotently on $H_1(Y_b)$.

For $d=0$ the fiber $f^{-1}(\infty)$ is smooth, hence the action of $T$ on $H_1(Y_b)$ is trivial.
Therefore the exact sequence~\eqref{piece-of-long-homology-2} and Corollary \ref{cor-of-general} imply that the $T$-action on $H_2(Y,Y_b)$ is unipotent with two Jordan blocks of size $2$ and eight blocks of size $1$. This means that the Landau--Ginzburg model $(Y,w)$ is not of Fano type, which proves (iii).

For $d>0$ the fiber $f^{-1}(\infty)$ is singular, so the $T$-action on $H_1(Y_b)$ is nontrivial (see~\cite[Table $1$]{Ko63}).
Therefore the exact sequence~\eqref{piece-of-long-homology-2} and Corollary \ref{cor-of-general} imply that the $T$-action on~$H_2(Y,Y_b)$ is unipotent with one Jordan block of size $3$ and $9-d$ blocks of size $1$. Therefore~$(Y,w)$ is of Fano type and equations \eqref{eq-hpq} hold. This completes the proof of  Proposition~\ref{proposition:h-numbers}.

\subsubsection{The numbers $f^{p,q}(Y,w)$}
Recall that we have the open embedding $j\colon Y\hookrightarrow Z$.

\begin{lemma}
\label{lemma:log_cohomology} We have
$$\Omega^0_Z(log\, D)\simeq\cO_Z \quad \text{and} \quad \Omega^2_Z(log\, D)\simeq\cO_Z.$$
Hence
$$\Omega^0_Z(log\, D)(-D)\simeq\Omega^2_Z(log\, D)(-D)\simeq\omega _Z.$$
\end{lemma}

\begin{proof} This follows from the definition of the logarithmic complex in Subsection \ref{subs-def-fpq}
and the fact that $D$ is the anticanonical divisor.
\end{proof}

\begin{proposition}[{\cite[Proposition 27]{LP18}}]
\label{prop-fpq}
The following equalities hold.
\begin{equation}\label{first-two}
h^i(Z,\Omega^0_{Z}(log\, D)(-D))=h^i(Z,\Omega^2_{Z}(log\, D)(-D))
=\left\{                                                                   \begin{array}{ll}
                                                                              0, & \hbox{i=0,1;} \\
                                                                                        1, & \hbox{i=2,}
                                                                                      \end{array}                                                                                  \right.
\end{equation}
\begin{equation}\label{last-one}
h^i(Z,\Omega^1_Z(log\, D)(-D))=\left\{
                     \begin{array}{ll}
                       0, & \hbox{i=0,2;} \\
                       10-d, & \hbox{i=1.}
                     \end{array}
                   \right.
\end{equation}
\end{proposition}

\begin{proof}[Idea of the proof]
The equalities \eqref{first-two} follows from Serre duality and Lemma \ref{lemma:log_cohomology}.
The equality \eqref{last-one} follows from the analysis of the complex
$$
\Omega^0_{Z}(log\,D)(-D)\to\Omega^1_{Z}(log\,D)(-D)\to \Omega^2_{Z}(log\,D)(-D)\to 0,
$$
which is a resolvent of the sheaf $j_!\CC_{Y}$, see, for instance,~\cite[p. 268]{DI87}.
This complex gives the spectral sequence
$$E_1^{pq}=H^p(Z,\Omega^q_{Z}(log\,D)(-D))$$
which converges to $H^{p+q}(Z,j_!\CC_{Y})$.
\end{proof}

\begin{proposition}[{\cite[Proposition 28]{LP18}}]
\label{proposition: del Pezzo f}
There are the isomorphisms
\begin{itemize}
  \item[(i)] $\Omega^0_Z(log\, D,f)\simeq\cO_Z(-D)\simeq\omega _Z;$
  \item[(ii)]
        $\Omega^2_Z(log\, D,f)\simeq\Omega ^2_Z(log\, D)\simeq\cO_Z.$
  \item[(iii)] There exists a short exact sequence of sheaves on $Z$
       $$0\to \Omega^1_Z(log\, D)(-D)\to \Omega^1_Z(log\, D,f)\to \cO _D\to 0.$$
\end{itemize}
\end{proposition}

\begin{proposition}
\label{proposition:f-numbers}
One has
$$f^{p,q}(Y,w)=\left\{\begin{array}{rl}
                       1, & (p,q)=(0,2),(2,0); \\
                       10-d, & (p,q)=(1,1); \\
                       0, & \hbox{otherwise.}
\end{array}\right.
$$
\end{proposition}

\begin{proof}
Proposition \ref{prop-fpq} and Lemma \ref{proposition: del Pezzo f} give
$$f^{p,0}(Y,w)=h^p\left(Z,\Omega^0_Z(log\, D,f)\right)=h^p\left(Z,\omega_Z\right)=
\left\{\begin{array}{ll}
0, & \hbox{p=0,1;} \\
1, & \hbox{p=2,}
\end{array}
\right.
$$
$$f^{p,1}(Y,w)=h^p\left(Z,\Omega^1_Z(log\, D,f)\right)=
h^p\left(Z,\Omega^1_Z(log\, D)(-D)\right)=
\left\{ \begin{array}{ll}
                       0, & \hbox{p=0,2;} \\
                       10-d, & \hbox{p=1,}
                     \end{array}
                   \right.
                   $$
and
$$f^{p,2}(Y,w)=h^p\left(Z,\Omega^2_Z(log\, D,f)\right)=h^p\left(Z,\cO_Z\right)=\left\{ \begin{array}{ll}
1, & \hbox{p=0;} \\
0, & \hbox{p=1,2.}
\end{array}
\right.
$$
\end{proof}

\subsection{End of proof of Theorem \ref{theorem:main-Hodge-2}  and discussion}
Studying elliptic surfaces in Section~\ref{subsection:LG-Hodge-numbers-for-surfaces} is motivated by Mirror Symmetry constructions
for del Pezzo surfaces from~\cite{AKO06}. The authors prove there ``a half'' of Homological Mirror Symmetry conjecture for del Pezzo surfaces.
More precise, they prove that for a general del Pezzo surface $S_d$ of degree $d$, $1\leq d\leq 9$, obtained
by blow up of $\PP^2$ in $9-d$ general points there exist a complexified symplectic form $\omega _Y$ on $(Y,w)$, where~$(Y,w)$ has $12-d$
nodal singular fibers, and that $Y$ can be compactified to $Z$ for which $D$ is a wheel of $d$ curves, such that
\begin{equation}\label{HMS-equivalence}
D^b(coh\ S_d)\cong FS((Y,w),\omega_Y).
\end{equation}
We call $(Y,w)$ {\it a Landau--Ginzburg model for $S_d$}.
We allow the case $d=0$ as well; in this case~$(Y,w)$ is a Landau--Ginzburg model for $\PP^2$ blown up in $9$ intersection points
of two elliptic curves, see~\cite{AKO06}. The equivalence~\eqref{HMS-equivalence} holds in this case as well.

\begin{remark}
\label{remark: Hodge for del Pezzo}
The description of del Pezzo surface $X$ of degree $d$ as a blow up of $\PP^2$ gives the following equalities:
\begin{equation*}
h^{p,q}(X)=\left\{\begin{array}{rl}
                       1, & (p,q)=(0,2),(2,0); \\
                       10-d, & (p,q)=(1,1); \\
                       0, & \hbox{otherwise.}
\end{array}\right.
\end{equation*}
\end{remark}

This remark, together with Proposition~\ref{proposition:h-numbers}, provides a proof of
part (ii) of Theorem~\ref{theorem:main-Hodge-2} and thus completes the proof of this theorem.
In other words, Conjecture~\ref{conj-2} and ``a half'' of Conjecture~\ref{conj-1} hold for (mirrors of) del Pezzo surfaces.

\begin{remark}
\label{remark:i-numbers}
The second part of Conjecture~\ref{conj-1} does not hold already for Landau--Ginzburg model $(Y,w)$ for $\PP^2$. Indeed, one has
$h^{0,0}(Y,w)=h^{1,1}(Y,w)=h^{2,2}(Y,w)=1$. However the Landau--Ginzburg model $(Y,w)$ has exactly three singular fibers, and the singular set of
these fibers is a single node.
Hence
the numbers $i^{p,q}(Y,w)$ are integers divisible by $3$.
\end{remark}

\begin{remark}
Del Pezzo surfaces are blow ups of $\PP^2$ with one exception, that is, a quadric surface. However toric Landau--Ginzburg model for
quadric by Part~\ref{part: del Pezzo surfaces} is an elliptic pencil with reduced fiber rover infinity which is a wheel of $8$ curves.
Thus the assertion of Theorem~\ref{theorem:main-Hodge-2} holds for quadric as well.
\end{remark}

\section{Fano threefolds}
\label{subsection:KKP-3}

In this section we, following~\cite{CP18}, study Conjecture~\ref{conj-2} in the three-dimensional case.
The important ingredient of the proof is the following result of A.\,Harder that treats this conjecture
in terms of geometry of Landau--Ginzburg models.
That is, consider a tame compactified Landau--Ginzburg model $(Y,w)$, where $w\colon Y\to \CC$ and $\dim Y=3$.
Denote its compactification by $(Z,f)$. Let the divisor over infinity $f^{-1}(\infty)$ combinatorially be a triangulation of a sphere.
Assume that $h^{i,0}(Z)=0$ for $i>0$. Let a general fiber $f^{-1}(\lambda)$ be a K3 surface.

\begin{theorem}[{\cite[Theorem 10]{Ha17}}]
The Hodge diamond for $f^{p,q}(Y,w)$ numbers is
$$
\begin{matrix}
&&&0&&& \\
&&0&&0&&\\
&0&& k_Y &&0& \\
1\qquad&& ph-2 + h^{1,2}(Z) && ph-2 + h^{2,1}(Z) &&\qquad1 \\
&0&& k_Y &&0& \\
&&0&&0&&\\
&&&0&&&
\end{matrix}
$$
where
$$
ph=\dim\Bigg(\mathrm{coker}\Big(H^2\big(Z,\mathbb{R}\big)\to H^2\big(V,\mathbb{R}\big)\Big)\Bigg)
$$
is a corank of restriction of second cohomology of the ambient space to a general fiber $V$, and $k_Y$
is given by
$$
k_Y = \sum_{s\in \Sigma}(\rho_s-1),
$$
where $\Sigma$ is a set of critical values of $w$ and $\rho_s$ is the number of irreducible components of $w^{-1}(s)$.
\end{theorem}

In particular the assumptions of this theorem hold for toric Landau--Ginzburg models by Theorem~\ref{theorem: Minkowski CY} and Remark~\ref{remark:Hodge purity}. Moreover in this case  $h^{2,1}(Z)=0$.

Note that birational smooth Calabi--Yau varieties are isomorphic in codimension 1,
so the numbers $k_Y$ and $ph$ do not depend on particular Calabi--Yau compactification $Y$ of toric Landau--Ginzburg model for $X$.
Moreover, by Remark~\ref{remark:Minkowski_equivalence}, they do not depend on certain Minkowski toric model.

We need the following statements on the intersection theory for du Val surfaces for the proof.

\begin{proposition}[{\cite[Proposition A.1.2]{CP18}}]
\label{proposition:du-Val-intersection}
Suppose that $O$ is a Du Val singular point of the surface $S$, both curves $C$ and $Z$ are smooth at $O$,
and $C$ intersects $Z$ transversally at the point $O$.
Then for the locla intersection indices $\Big(C\cdot Z\Big)_O$ the following assertions hold.
\begin{itemize}
\item[(i)] The point $O$ is a singular point of $S$ of type $\sA_n$ or $\sD_n$.

\item[(ii)] If $O$ is a singular point of type $\sA_n$ and proper transforms of the curves $C$ and $Z$ on
the minimal resolution $\widetilde{S}$ of $O$ intersect $k$-th and $r$-th exceptional curves in the chain of exceptional curves of the minimal resolution of~$O$,
then
$$
\Big(C\cdot Z\Big)_O=\left\{\aligned%
&\frac{r(n+1-k)}{n+1}\ \text{for}\ r\leqslant k,\\
&\frac{k(n+1-r)}{n+1}\ \text{for}\ r>k.\\
\endaligned
\right.
$$

\item[(iii)] If $O$ is of type $\sD_n$, then $\Big(C\cdot Z\Big)_O=\frac{1}{2}$.
\end{itemize}
\end{proposition}

\begin{proposition}[{\cite[Proposition A.1.3]{CP18}}]
\label{proposition:du-Val-self-intersection}
Suppose that $O$ is a Du Val singular point of the surface $S$,
and the curve $C$ is smooth at the point $O$.
Then the following holds.
\begin{itemize}
\item[(i)]
The point $O$ is a singular point of the surface $S$ of type $\sA_n$, $\sD_n$, $\sE_6$ or $\sE_7$.

\item[(ii)] If $O$ is a singular point of type $\sA_n$, and a proper transform $\widetilde{C}$ of $C$ intersects $k$-th exceptional curve in the chain of exceptional curves of the minimal resolution of~$O$, then
$$
C^2=\widetilde{C}^2+\frac{k(n+1-k)}{n+1}.
$$

\item[(iii)] If $O$ is a singular point of type $\sD_n$, then $C^2=\widetilde{C}^2+1$ or $C^2=\widetilde{C}^2+\frac{n}{4}$.

\item[(iv)] If $O$ is a singular point of type $\sE_6$, then $C^2=\widetilde{C}^2+\frac{4}{3}$.

\item[(v)] If $O$ is a singular point of type $\sE_7$, then $C^2=\widetilde{C}^2+\frac{3}{2}$.
\end{itemize}
\end{proposition}

\begin{theorem}[{\cite[Main Theorem]{CP18}}]
\label{theotem:KKP conjecture-3}
Conjecture~\ref{conj-2} holds for smooth Fano threefolds.
\end{theorem}

\begin{proof}[Idea of the proof]
Consider a smooth Fano threefold $X$. By Corollary~\ref{corollary: toric LG for threefolds} it has a toric Landau--Ginzburg model.
If $-K_X$ is very ample, then choose a model $f(x,y,z)$ such that after a multiplication by $xyz$ and the compactification
given by a natural embedding $\Aff[x,y,z]\hookrightarrow \PP[x:y:z:t]$ we get a family of quartics $\mathcal S$ defined by
$$
f_4(x,y,z,t)=\lambda xyzt, \ \ \ \ \lambda\in \CC\cup\{\infty\}.
$$
One can check that this is always possible.
If $-K_X$ is not very ample, then compactify a toric Landau--Ginzburg model for $X$ to a family of quartics $\mathcal S$
using Proposition~\ref{proposition:CY hyperelliptic-CY}.

Now resolve these families blowing up base loci and keeping track the number of exceptional divisors lying in fibers.
For this study singularities of fibers along the base locus. Say, ``floating'' singularity (whose coordinates changing
when elements of the family vary) or isolated du Val singularity for each fiber does not give a component
to a fiber of the resolution.
In a general case for each fiber of the family $\mathcal S_\lambda$ one can define \emph{defect} $\mathbf{D}_P^\lambda$ of a singular point $P$
as a number of exceptional divisors of the resolution lying in the fiber over the point, and \emph{defect} $\mathbf{C}^\lambda$
of a base curve $C$ of the pencil in the fiber $\mathcal S_\lambda$.
In particular, defect of isolated du Val singularity is 0.

Defects of curves can be computed in terms of multiplicities of the curves in fibers.
To compute defects of points one need more deep analysis, that is counting of base curves lying over the point.
More details see in~\cite[\S 1]{CP18}.

Denote the number of irreducible components of a variety $V$ by $[V]$.
For a resolution $f\colon Y\to \PP^1$ of the pencil $\mathcal S_\lambda$ it holds
$$
{[{f}^{-1}(\lambda)]=
%\boxed{\text{the number of irreducible components of the surface $
[\mathcal S_{\lambda}]+\sum_{i=1}^{r}\mathbf{C}_j^\lambda+\sum_{P\in\Sigma}\mathbf{D}_P^\lambda},
$$
where $\{C_1,\ldots,C_r\}$ is a set of base curves and $\Sigma$ is a set of points over which exceptional divisors lie.
We denote the total space of the resolution by $Y$ since by Remark~\ref{remark: codimension 1} it is isomorphic in codimension $1$
to the log Calabi--Yau compactification from Corollary~\ref{corollary: toric LG for threefolds}.
Taking sum of the defects over all fibers find~$k_Y$ and compare it with the number $h^{1,2}(X)$, which can be found, say, in~\cite{IP99}.

Let $M$ be the $r\times r$ matrix with entries $M_{ij}\in\mathbb{Q}$ that are given by
$$
M_{ij}=C_i\cdot C_j,
$$
where $C_i\cdot C_j$ is the intersection of the curves $C_i$ and $C_j$
on the surface $S_\lambda$.
One can easily show that for general $\lambda$
\begin{equation*}
\label{equation:main-2}
\dim\Bigg(\mathrm{coker}\Big(H^2\big(Z,\mathbb{R}\big)\to H^2\big(V,\mathbb{R}\big)\Big)\Bigg)-2=
22-\mathrm{rk}\,\mathrm{Pic}\Big(\widetilde{\mathcal S}_{\lambda}\slash \mathcal S_{\lambda}\Big)-\mathrm{rk}(M),
\end{equation*}
where $\widetilde{\mathcal S}_\lambda$ is a minimal resolution.
Since for a general
$\lambda$ the surface $\widetilde S_\lambda$ has du Val singularities, it is enough to find types of these singularities
to find a relative Picard number.

The theorem can be proved by direct computations for each Fano threefold in the way outlined above.
\end{proof}

A.\,Harder's results and Conjecture~\ref{conj-2} motivate the following.
Consider a smooth Fano variety $X$ of dimension $N$ and let $Y$ be its $N$-dimensional Landau--Ginzburg model.
Define, as before $k_{Y}$ as a difference between the number of irreducible components of reducible fibers
of $Y$ and the number of reducible fibers.

\begin{conjecture}[{\cite[Problem 27]{Prz13}, \cite[Conjecture 1.1]{PSh15a}, cf.~\cite{GKR12}}]
\label{conjecture: Hodge-components}
For a smooth Fano variety $X$ of dimension $N\geqslant 3$ one has $h^{1,N-1}(X)=k_{Y}$.
\end{conjecture}

Thus Theorem~\ref{theotem:KKP conjecture-3} implies Conjecture~\ref{conjecture: Hodge-components}) for threefolds. A proof of
Conjecture~\ref{conjecture: Hodge-components} for complete intersections is given by Theorem~\ref{theorem: CI components}.

Finally, by Homological Mirror Symmetry one expects that the number of reducible fibers
of threefold Landau--Ginzburg model is not greater than the Picard rank of the corresponding Fano variety.
In particular, the proof of Theorem~\ref{theotem:KKP conjecture-3} implies that for the Picard rank one case
one has at most one reducible fiber. It turns out that one can get an important information
from the monodromy at the reducible fiber.
That is, comparing results of Iskovskikh~\cite{Isk77}, Golyshev~\cite{Go07}, and compactified toric Landau--Ginzburg models
constructed above, one can get the following.

\begin{theorem}[{\cite[Theorem 3.3]{KP09}}]
\label{theo:Gol-Isk}
  Let $X$ be a smooth Picard rank one Fano threefold whose compactified Landau--Ginzburg model has a fiber with non-isolated singularities.
  Then the monodromy (in the second cohomology) at this fiber is unipotent if and only if $X$ is rational.
\label{theorem:Gol-Isk}
\end{theorem}

Another approach to (non-)rationality of Fano varieties via their Landau--Ginzburg models see in~\cite{IKP14}.

\part{Complete intersections in (weighted) projective spaces and Grassmannians}
\label{part:complete intersections in Grass}

In this part we study (toric) Landau--Ginzburg models of smooth complete intersections in weighted projective spaces and Grassmannians.

We mainly focus on complete intersections in Grassmannians.
Weighted complete intersections are studied in the preprint~\cite{PSh}.
We just briefly present the main results here.

First describe Givental's construction from~\cite{Gi97b} for Landau--Ginzburg models of Fano complete intersections
in smooth toric varieties. We also describe their period integrals. We apply this construction to complete intersections,
and its generalization to ``good'' toric varieties to del Pezzo surfaces (see Part~\ref{part: del Pezzo surfaces}) and, following~\cite{BCFKS97}, to complete intersections in Grassmannians (see Section~\ref{section: CI in Grass}).

\section{Givental's construction}
\label{section: complete Givental}
Let $X$ be a factorial $N$-dimensional toric Fano variety of Picard rank $\rho$ corresponding to
a fan $\Sigma_X$ in a lattice $\mathcal{N}\simeq\ZZ^N$. Let $D_1,\ldots, D_{N+\rho}$ be its %corresponding boundary
prime invariant divisors.
Let~\mbox{$Y_1,\ldots,Y_l$} be ample divisors in $X$ cutting out a smooth Fano complete intersection
$$Y=Y_1\cap\ldots\cap Y_l.$$
Put $Y_0=-K_X-Y_1-\ldots-Y_l$.
Choose a
basis
$$
\{H_1,\ldots,H_\rho\}\subset H^2(X)$$
so that for any $i\in [1,\rho]$ and any curve $\beta$ in the K\"ahler cone
$K$ of $X$ one has $H_i\cdot\beta\ge 0$.
Introduce variables $q_1,\ldots,q_{\rho}$ as in Section~\ref{section:toric}.
Define $\kappa_i$ by $-K_Y=\sum \kappa_i H_i$.

The following theorem is a particular case of Quantum Lefschetz hyperplane theorem, see~\cite[Theorem 0.1]{Gi97b}.
\begin{theorem}
Suppose that $\dim (Y)\geqslant 3$. Then the constant term of regularized $I$-series for $Y$ is given by
\begin{equation}\label{particular-Quantum-Lefschetz}
\widetilde{I}^Y_{0}(q_1,\ldots,q_{\rho})=
\exp\big(\mu(q)\big)\cdot
\sum_{\beta\in K\cap H_2(X)}q^\beta\frac{\prod_{i=0}^l|\beta\cdot Y_i|!%^{\frac{(\beta\cdot Y_i)}{|\beta\cdot Y_i|}}
}{\prod_{j=1}^{N+\rho}|\beta\cdot D_j|!^{\frac{\beta\cdot D_j}{|\beta\cdot D_j|}}}
\end{equation}
where $\mu(q)$
is a correction term linear in $q_i$ (in particular it is trivial in the higher index case).
For $\dim(Y)=2$ the same formula holds after replacing $H^2 (Y)$ in the definition of~\mbox{$\widetilde{I}^Y_{0}$} by the restriction of $H^2(X)$ to $Y$.
\end{theorem}

\begin{remark}
Note that the summands of
the series~\eqref{particular-Quantum-Lefschetz} have non-negative degrees in $q_i$.
\end{remark}

Now we describe Givental's construction of a dual to $Y$ Landau--Ginzburg model and compute its periods.
Introduce $N+\rho$ formal variables $u_1,\ldots,u_{N+\rho}$
corresponding to divisors~\mbox{$D_1,\ldots,D_{N+\rho}$.}

Recall that the short exact sequence~\eqref{sequence} identifies $\pic(X)^{\vee}$ with the lattice of relations on primitive vectors on the
rays of $\Sigma_X$ considered as Laurent monomials in variables $u_i$.
On the other hand, as the basis in $\pic(X)$ is chosen, we can identify $\pic(X)^{\vee}$ and $\pic(X)=H^2(X)$.
Hence we can choose a basis in the lattice of relations on primitive vectors on the
rays of $\Sigma_X$ corresponding to $\{H_i\}$ and, thus, to $\{q_i\}$.
We denote these relations by $R_i$, and interpret them as monomials in
the variables~\mbox{$u_1,\ldots,u_{N+\rho}$}.
We denote images of $D_i\in \mathcal D$ в $\pic (X)$ by $D_i$ as well.

Choose a nef-partition, i.\,e.
a partition of the set $[1,N+\rho]$ into sets $E_0,\ldots,E_l$ such that for any $i\in [1,l]$
the divisor $\sum_{j\in E_i} D_j$ is linearly equivalent to $Y_i$ (which also
implies that the divisor~\mbox{$\sum_{j\in E_0} D_j$} is linearly equivalent to $Y_0$).

The following definition is well-known (see discussion after Corollary~0.4 in~\cite{Gi97b},
and also~\cite[\S7.2]{HV00}).

\begin{definition}
\label{definition: Givental LG}
\emph{Givental's Landau--Ginzburg model} for $Y$
is a variety~\mbox{$LG_0(Y)$} in a torus
$$\Spec \C_q[u_1^{\pm 1}, \ldots,u_{N+\rho}^{\pm 1}]$$
given by equations
\begin{equation}\label{eq:Rq}
R_i=q_i, \ i\in [1,\rho],
\end{equation}
and
\begin{equation*}\label{eq:Eu}
\left(\sum_{s\in E_j} u_s\right)=1, \ j\in [1,l],
\end{equation*}
with a superpotential $w=\sum_{s\in E_0} u_s$. %,
Given a divisor~\mbox{$D\sim \sum r_i H_i\in \Pic (Y)_\CC,$}
define a \emph{Landau--Ginzburg model of Givental type~\mbox{$LG(Y,D)$} corresponding to $(Y,D)$},
putting $q_i=\exp(r_i)$.
Put $LG(Y)=LG(Y,0)$.
\end{definition}

\begin{remark}
\label{remark:shift}
The superpotential of Givental's Landau--Ginzburg model can be defined as~\mbox{$w'=u_1+\ldots+u_{N+\rho}$}. However we don't make a distinction between
two superpotentials $w$ and $w'$ as $w'=w+l$, since both these functions define the same family over $\C_q$.
\end{remark}

Given variables $x_1,\ldots, x_r$, define a \emph{standard logarithmic form in these variables} as the form
\begin{equation*}
\label{eq: standard form}
\Omega(x_1,\ldots,x_r)=\frac{1}{(2\pi i)^r}\frac{dx_1}{x_1}\wedge\ldots\wedge\frac{dx_r}{x_r}.
\end{equation*}

The following definition is well-known (see discussion after Corollary~0.4 in~\cite{Gi97b},
and also~\cite{Gi97a}).

\begin{definition}
%\label{def:integral}
Fix $N+\rho$ real positive numbers $\varepsilon_1,\ldots,\varepsilon_{N+\rho}$ and define an $(N+\rho)$-cycle
$$
\delta = \{|u_i=\varepsilon_i|\}\subset \C[u_1^{\pm 1},\ldots, u_{N+\rho}^{\pm 1}]
.$$
%and $\rho$ complex numbers~\mbox{$q_1,\ldots,q_\rho$}.
\emph{Givental's integral} for $Y$ or $LG_0(Y)$ is an integral
\begin{equation}\label{eq:Givental-integral}
I_Y^0=
\int\limits_{\delta}\frac{
\Omega (u_1,\ldots,u_{N+\rho})}
{\prod_{i=1}^\rho (1-\frac{q_i}{R_i})\cdot \prod_{j=0}^k\left(1-\sum_{s\in E_j} u_s\right)}\in \C[[q_1,\ldots,q_\rho]].
\end{equation}
Given a divisor $D=\sum r_i H_i$ one can \emph{specialize Givental's integral to the anticanonical direction and divisor $D$}
putting~\mbox{$q_i=e^{r_i}t^{\kappa_i}$} in the integral~\eqref{eq:Givental-integral}, cf. Definition~\ref{definition:I-series}.
We denote the result of specialization by $I_{(Y,D)}$.
Put $I_{(Y,0)}=I_{Y}$.
\end{definition}

\begin{remark}
The integral~\eqref{eq:Givental-integral} does not depend on numbers $\varepsilon_i$ provided they are small enough.
\end{remark}

\begin{remark}
\label{remark:Givental-integral sign}
The integral~\eqref{eq:Givental-integral} is defined up to a sign sinse we do not specify an order of variables.
\end{remark}

The following assertion is well-known to experts (see~\cite[Theorem~0.1]{Gi97b},
and also discussion after Corollary~0.4 in~\cite{Gi97b}).

\begin{theorem}
\label{theorem: Givental I-series toric}
One (up to a sign, see Remark~\ref{remark:Givental-integral sign}) has
$$
\widetilde{I}^Y_0=I_Y^0.
$$
\end{theorem}

The recipe for Givental's Landau--Ginzburg model and integral can be written down in another, more simple, way.
That is, we make suitable monomial change of variables~\mbox{$u_1,\ldots,u_{N+\rho}$} an get rid
of some of them using equations~\eqref{eq:Rq}. More precisely,
since~$\mathcal{N}$ is a free group, using the exact sequence~\eqref{sequence}
one obtains an isomorphism
$$\mathcal{D}\simeq \pic{(X)}^\vee\oplus \mathcal{N}.$$
Thus one can find a monomial change of variables
$u_1,\ldots, u_{N+\rho}$ to some new variables~\mbox{$x_1,\ldots,x_N, y_1,\ldots,y_\rho$}, so that
$$u_i=\widetilde{X}_i(x_1,\ldots, x_N, y_1,\ldots,y_\rho,q_1,\ldots,q_\rho)$$
such that for any $i\in [1,\rho]$
one has
$$\frac{R_i(u_1,\ldots,u_{N+\rho})}{q_i}=\frac{1}{y_i}.$$
Put
$$X_i=\widetilde{X}_i(x_1,\ldots, x_N, 1,\ldots,1,q_1,\ldots,q_\rho).$$
Then $LG_0(Y)$ is given in the torus $\Spec \C_q[x_1^{\pm 1},\ldots,x_N^{\pm 1}]$
by equations
$$
\sum_{s\in E_j} \alpha_s X_s=1, \ j\in [1,l],
$$
with superpotential $w=\sum_{s\in E_0}\alpha_s X_s$, where
$\alpha_i=\prod q_j^{r_{i,j}}$ for some integers~$r_{i,j}$.

Let us mention that given
a Laurent monomial $U_i$ in variables $u_j$, $j\in [1,N+\rho]$,
that does not depend on a variable $u_i$ one has
\begin{equation*}
\Omega(u_1,\ldots, u_i^{\pm 1}\cdot U_i,\ldots, u_{N+\rho})=\pm\Omega(u_1,\ldots, u_i,\ldots, u_{N+\rho}).
\end{equation*}

This means that
\begin{equation}
\label{eq:integral xy}
I_Y^0=
\int_{\delta'}%\limits_{\substack{|x_j|=\varepsilon_j\\ |y_i|=\mu_i}}
\frac{
\pm \Omega(y_1,\ldots, y_\rho)\wedge\Omega(x_1,\ldots,x_N)
}{\prod_{i=1}^\rho (1-y_i)\prod_{j=0}^k\left(1-\sum_{s\in E_j} \alpha_s\widetilde{X}_s\right)}
\end{equation}
for some $(N+\rho)$-cycle $\delta'$.

Consider an integral
$$
\int_{\sigma} \frac{dU}{U}\wedge \Omega_0
$$
for some form $\Omega_0$ and a cycle $\sigma=\sigma'\cap \{|U|=\varepsilon\}$
%which is a boundary of a tubular neighborhood of
for some cycle~\mbox{$\sigma'\subset \{U=0\}$}.
%Then Residue Theorem states
It is well known that
(see, for instance,~\cite[Theorem 1.1]{ATY85}) that
$$
\frac{1}{2\pi i}\int_{\sigma} \frac{dU}{U}\wedge \Omega_0=\int_{\sigma'}\left.\Omega_0\right|_{U=0}
$$
if both integrals are well defined (in particular the form $\Omega_0$ does not have a pole along~\mbox{$\{U=0\}$}).

We denote $$\left.\Omega_0\right|_{U=0}=\Res_U\left(\frac{dU}{U}\wedge \Omega_0\right).$$

Taking residues of the form on the right hand side of the formula~\eqref{eq:integral xy} with respect to $y_i$ one gets
$$
I_Y^0=
\int_{\delta''}%\limits_{|x_i|=\varepsilon_i}
\frac{
\pm \Omega(x_1,\ldots,x_N)
}{\prod_{j=0}^l\left(1-\sum_{s\in E_j} \alpha_s X_s\right)}
$$
for some $N$-cycle $\delta''$.

Moreover, one can introduce a new parameter $t$ and scale $u_i\to t u_i$ for $i\in E_0$.
Fix a divisor class $D=\sum r_i H_i$.
One can check that after a change of coordinates $q_i=e^{r_i}t^{\kappa_i}$
the initial integral restricts to the integral
\begin{equation}
\label{eq:restricted integral}
\int_{\delta_1}%\limits_{|x_i|=\varepsilon_i}
\frac{
\pm \Omega (x_1,\ldots,x_N)
}{\prod_{j=1}^k\left(1-\left(\sum_{s\in E_j} \gamma_s X_s\right)\right)\cdot \left(1-t\sum_{i\in E_0} \gamma_i X_i\right)}=I_{(Y,D)}
\end{equation}
for some monomials $\gamma_i$ and $N$-cycle $\delta_1$ homologous to a cycle
$$\delta_1^0=\{|x_i|=\varepsilon_i\mid i\in [1,N]\}.$$
In particular, for $D=0$ we have $\gamma_i=1$.
The same specialization defines the Landau--Ginzburg model $LG(Y)$
\begin{equation}
\label{equation: Givental singular}
\sum_{s\in E_j} X_s=1, \quad j\in [1,k],
\end{equation}
with superpotential $w=\sum_{s\in E_0} X_s$.

Consider a non-toric variety $X$ that has a small (that is, terminal Gorenstein) toric degeneration $T$.
Let $Y$ be a Fano complete intersection in $X$. Consider a nef-partition for the set of rays of the fan of $T$ corresponding to (degenerations of)
hypersurfaces cutting out $Y$. Let $LG(Y)$ be a result of applying the procedure discussed above
for Givental's integral defined for $T$ and the nef-partition in the same way as in the case of complete
intersections in toric varieties.
Batyrev in~\cite{Ba97} suggested $LG(Y)$ as a Landau--Ginzburg model for $Y$.
Moreover, at least in some cases such as for complete intersections in Grassmannians (see Subsection~\ref{section: Grass periods}) Givental's integral and Landau--Ginzburg model can be simplified further by making birational
changes of variables and taking residues. This gives weak Landau--Ginzburg models for complete intersections in
projective spaces (see Section~\ref{section: weak CI}) and, more general, Grassmannians (see Section~\ref{section: CI in Grass}).

We also generalize the model~\eqref{equation: Givental singular} for smooth complete intersections in weighted projective spaces,
see Subsection~\ref{section: weak CI}. Such complete intersection can be described as a complete intersection in smooth toric variety
after resolution of singularities that are far away from the complete intersection. However this description is equivalent
to applying the construction~\eqref{equation: Givental singular} directly, cf.~\cite{Prz07b}.

\section{Weighted complete intersections}
\label{section:complete intersections}

In this section we apply constructions from Section~\ref{section: complete Givental} for complete intersections in weighted projective spaces. More details see in~\cite{PSh}.

\subsection{Nef-partitions}
\label{subsection: nef-partitions}
A crucial ingredient of Givental's construction from Section~\ref{section: complete Givental} for complete intersections in toric varieties
and its generalization is an existence of nef-partitions for the complete intersections.
Obviously such nef-partitions exist for complete intersections in projective spaces.
However in general the existence of such nef-partitions is not clear.
From classification point of view the most interesting Fano varieties are ones with Picard group $\ZZ$.
If a complete intersection admit a nef-partition, then the ambient toric variety is a weighted projective space
(or its quotient if the complete intersection is singular).
In general the existence of nef-partition for weighted complete intersection is expected but not proven.

\begin{conjecture}
\label{conjecture: nef for weighted}
Smooth well formed weighted complete intersection has a good nef-partition and a toric Landau--Ginzburg model
(definitions see below).
\end{conjecture}

\begin{remark}
The existence of good (see Definition~\ref{definition: good nef}) nef-partition implies the existence of weak Landau--Ginzburg model
(see Section~\ref{section: weak CI}) satisfying the toric condition (see Section~\ref{section: toric LG CI}).
In a lot of cases by analogy with Theorem~\ref{theorem: CY for CI} one can check the Calabi--Yau condition
which will show that the Landau--Ginzburg model is toric. The main problem to show this is that in general
the Newton polytope of the weak Landau--Ginzburg model is not reflexive.
\end{remark}

We denote the greatest common divisor of the numbers $a_1,\ldots,a_r\in \NN$ by $(a_1,\ldots,a_r)$.

Remind some facts about weighted projective spaces. More details see in~\cite{Do82}.
Consider a weighted projective space $\P=\PP(w_0,\ldots,w_N)$.

\begin{definition}[{see~\cite[Definition 5.11]{IF00}}]
\label{definition:well-formed-P}
The weighted projective space $\P$ is said to be \emph{well formed} if the greatest common divisor of any $N$ of the weights~$w_i$ is~$1$.
\end{definition}

Any weighted projective space is isomorphic to a well formed one, see~\cite[1.3.1]{Do82}.

\begin{lemma}[{see~\cite[5.15]{IF00}}]
\label{lemma:singularities-of-P}
The singular locus of $\PP$ is a union of strata
$$
\Lambda_J=\left\{(x_0:\ldots:x_n) \mid x_j=0 \text{\ for all\ } j\notin J\right\}
$$
for all subsets $J\subset \{0,\ldots,n\}$ such that the greatest common divisor of the weights~$a_j$
for~\mbox{$j\in J$} is greater than~$1$.
\end{lemma}

\begin{definition}[{see~\cite[Definition 6.9]{IF00}}]
\label{definition:well-formed-WCI}
A subvariety $X\subset \PP$ of codimension $c$ is said to be \emph{well formed}
if~$\PP$ is well formed and
$$
\mathrm{codim}_X \left( X\cap\mathrm{Sing}\,\P \right)\ge 2.
$$
\end{definition}

\begin{definition}
Zeroes of weighted homogenous polynomial
$$
f\in
\CC[x_0,\ldots,x_N],
$$
where $\wt (x_i)=w_i$, of weighted degree $d$ are called a \emph{degree $d$ hypersurface} in
$\PP$.
\end{definition}

The rank of divisor class group of weighted projective space is  1,
so some multiple of any effective Weyl divisor is zeros of some weighted homogenous polynomial.
This enables us to define a degree of any Weyl divisor.
It is easy to see that a Weyl divisor of degree $d$ is Cartier if and only if all weights $w_i$ divide $d$.

Singularities of general complete intersection $X=X_1\cap
\ldots \cap X_k$ of Cartier divisors $X_1,\ldots,X_k$ are the intersection of $X$ with singularities of $\PP$. Thus $X$
is smooth if and only if the greatest codimension of strata of singularities of $\PP$ is less than $k$.
This means that
$(w_{i_1},\ldots,w_{i_{k+1}})=1$ for any collection of weights
$w_{i_1},\ldots,w_{i_{k+1}}$ (cf.~\cite{Di86}).

Let $\deg X_i=d_i$. A canonical sheaf of $X$ is
$$
\cO(d_1+\ldots+d_k-w_0-\ldots-w_N)|_X.
$$
Thus $X$ is Fano if and only if
 $\sum d_i<\sum w_j$.

\begin{definition}
\label{definition: good nef}
Let $X$ be a smooth complete intersection of divisors of degrees $d_1,\ldots,d_k$
in well formed weighted projective space $\PP(w_0,\ldots,w_N)$.
Recall that a splitting of $[0,N]$ into
 $k$ non-intersecting subsects $E_0, \ldots,E_k\subset
[0,n]$ such that $d_i=\sum_{j\in E_i} w_j$ for all $i>0$,
is called a nef-partition. A nef-partition is called \emph{good} if there is an index
$$j\in E_0=[0,N]\setminus \left(E_1\cup\ldots\cup E_k\right)$$ such that
$w_j=1$. A good nef-partition is called \emph{very good} if
$w_j=1$ for all $j\in E_0$.
\end{definition}

\begin{proposition}[{\cite[Theorem 9 and Remark 14]{Prz11}}]
%\label{proposition}
\label{proposition: nef for cartier}
Let $X$ be a smooth complete intersection of Cartier divisors in well formed weighted projective space.
Let $X$ be Fano. Then it admits a very good nef-partition.
\end{proposition}

\begin{remark}
\label{remark: nef all 1}
Denote the Fano index of a variety $X$ by $d_0=\sum w_i-\sum d_j$.
The proof of Proposition~\ref{proposition: nef for cartier} shows that at least $d_0+1$ weights are equal to 1.
This bound is strict: the example is a hypersurface of degree 6 in $\PP(1,1,2,3)$.
\end{remark}

Conjecture~\ref{conjecture: nef for weighted} holds not only for complete intersections of Cartier divisors.

\begin{theorem}[{\cite[Theorem 1.3]{PSh17b}}]
\label{theorem: nef for codim 2}
Smooth well-formed Fano complete intersection of codimension $2$ admits a very good nef-partition.
\end{theorem}

\begin{proof}[Idea of the proof]
One need to study the so called~\emph{weighted projective graphs}, that is graphs whose vertices are marked by weights of the weighted projective space, and edges connect those and only those vertices whose markings have non-trivial common divisor.
\end{proof}

If $X$ is a smooth well formed Calabi--Yau weighted complete intersection
of codimension~$1$ or~$2$, we can argue in the same way as in the proof of Proposition~\ref{proposition: nef for cartier} and Theorem~\ref{theorem: nef for codim 2}
to show that there exists a nef partition for $X$, for which we necessarily have $E_0=\varnothing$ in
the notation of Definition~\ref{definition: good nef}.
Constructing the dual nef partition we obtain a Calabi--Yau variety $Y$ that is mirror dual to $X$, see~\cite{BB96}. In the same paper
it is proved that the Hodge-theoretic mirror symmetry
holds for $X$ and $Y$. That is, for a given variety $V$ one can define \emph{string Hodge numbers} $h_{st}^{p,q}(V)$
as Hodge numbers of a crepant resolution of $V$ if such resolution exists.
Then, for $n=\dim X=\dim Y$, one has $h_{st}^{p,q}(X)=h_{st}^{n-p,q}(Y)$ provided that the ambient toric variety (weighted projective
space in our case) is Gorenstein.

Finally, we would like to point out a possible approach to a proof
of Conjecture~\ref{conjecture: nef for weighted} along the lines of Theorem~\ref{theorem: nef for codim 2}.
If $X$ is a smooth well formed Fano weighted complete intersection
of codimension $3$ or higher in a weighted projective space $\P=\P(w_0,\ldots,w_N)$,
it is possible that some three weights $w_{i_1}$, $w_{i_2}$, and $w_{i_3}$
are not coprime. Thus a weighted projective graph constructed in the proof of
Theorem~\ref{theorem: nef for codim 2} does not provide an adequate
description of singularities of the weighted projective space~$\P$.
An obvious way to (try to) cope with this
is to replace a graph by a simplicial complex that would remember the greatest common divisors
of arbitrary subsets of weights. However, this leads to combinatorial difficulties
that we cannot overcome at the moment. Except for the most straightforward ones,
like the effects on \emph{weak vertices} (which would be not that easy to control) and
possibly larger number of exceptions,
there is also a less obvious one (which is in fact easy to deal with). Namely, we
need a finer information about weights and degrees than that provided by
\cite[Lemma~2.15]{PrzyalkowskiShramov-Weighted}.

\begin{example}
Let $X$ be a weighted complete intersection of hypersurfaces of degrees
$2$, $3$, $5$, and $30$ in $\P(1^k,6,10,15)$, where $1^k$ stands for $1$ repeated $k$ times.
Then $X$ is a well formed
Fano weighted complete intersection
provided that $k$ is large and~$X$ is general. Note that the conclusion of
\cite[Lemma~2.15]{PrzyalkowskiShramov-Weighted} holds for $X$. However, it is easy to see
that~$X$ is not smooth. Moreover, there is no nef partition for~$X$.
\end{example}

In any case, it is easy to see that the actual information
one can deduce from the fact that a weighted complete intersection is
smooth is much stronger than that provided by \cite[Lemma~2.15]{PrzyalkowskiShramov-Weighted}.
We also expect that combinatorial difficulties that one has to face on
the way to the proof of Conjecture~\ref{conjecture: nef for weighted} proposed in the proof of Theorem~\ref{theorem: nef for codim 2} are possible
to overcome.

\subsection{Weak Landau--Ginzburg models}
\label{section: weak CI}

Consider a general complete intersection $Y\subset \PP[w_0,\ldots,w_N]$ of hypersurfaces of degrees $d_1,\ldots,d_k$.

Put
$$d_0=\sum w_i-\sum d_j.$$
Let $d_0\ge 1$, that is let $Y$ be Fano.
Assume the existence of nef-partition $E_0,\ldots,E_k$ for $Y$.
Let $a_{i,1},\ldots,a_{i,r_i}$ be variables that correspond to indices from $E_i$.
Givental's Landau--Ginzburg model for $Y$ and a trivial divisor is given in the torus
$$(\CC^*)^{N}\simeq \TTT[a_{i,j}], \ \ \ i\in [1,k], \ j\in [1,i_r],$$ by equations
\begin{equation}
\label{projspace}
a_{i,1}+\ldots+a_{i,r_i}=1, \ \ \ \ i\in [1,k], %\quad i\in [1,k],
\end{equation}
and the superpotential
$
w=\sum a_{0,j}.
$

The variety given by the equations~\eqref{projspace},
after the change of variables
$$x_{i,j}=\frac{a_{i,j}}{\sum_{s} a_{i,s}},\ a_{i,r_i}=1,\ \ \ i\in [1,k],$$
is birational to the torus
$$(\CC^*)^m\simeq \TTT[x_{i,j}], \quad i\in[0,k], j\in [1,r_i-1].$$

The superpotential $w$ in the new variables is given by the Laurent polynomial
\begin{equation}
\label{formula: CI LG}
f_{Y}=\frac{\prod_{i=1}^k(x_{i,1}+\ldots+x_{i,r_i-1}+1)^{d_i}}{\prod_{i=0}^k \prod_{j=1}^{d_i-1} x_{i,j}}+x_{0,1}+\ldots+x_{0,r_0-1}.
\end{equation}

Formula~\eqref{particular-Quantum-Lefschetz} enables one to easily find the constant term of regularized  $I$-series for $Y$
and to compare it with constant terms series for $f_Y$.
The formula for this series can be found easily combinatorially. One can check that the period condition holds for $f_Y$,
that is it is a weak Landau--Ginzburg model for $Y$.
However one can prove that the series coincide with Givental's integral.

\begin{proposition}[{see~\cite[Proposition 10.4]{PSh17}}]
\label{proposition:periods-for-proj-space}
The following holds:
$$I_Y=\int\limits_{
\substack{|x_{i,j}|=\varepsilon_{i,j} %\\ |y_s|=\varepsilon_s
}} \frac{\Omega(x_{0,1}, \ldots, x_{k,d_k-1})}{1-tf_Y}.$$
\end{proposition}

\begin{proof}[Idea of the proof]
Use changes of variables and the Residue Theorem.
\end{proof}

Thus smooth complete intersections having a good nef-partition have a weak Landau--Ginzburg model as well.

%\begin{corollary}
%Let $X$ be a complete intersection of Cartier divisors (for instance, a smooth hypersurface) in a weighted projective space
%or a smooth complete intersection of codimension $2$. By Proposition~\ref{proposition: nef for cartier} and Theorem~\ref{theorem: nef for codim 2}, the variety $X$ has a very good nef-partition. Thus $X$ has a weak Landau--Ginzburg model.
%\end{corollary}

\begin{remark}
It seems to be natural to consider Givental's Landau--Ginzburg models for quasi-smooth Fano complete intersections.
However even quasi-smooth Cartier hypersurface does not always admit such a model.
The example is a hypersurface of degree 30 in $\PP(1, 6, 10, 15)$.
Moreover, even if such a hypersurface has a Givental type Landau--Ginzburg, it's now always presentable by a weak Landau--Ginzburg model
as above. An example is the hypersurface of degree 30 in $\PP(1,1,1,1,1,6,10,15)$.
\end{remark}

\subsection{Calabi--Yau compactifications}
\label{section: Calabi--Yau CI}
The method of constructing log Calabi--Yau compactifications applied in Theorem~\ref{theorem: Minkowski CY} can be generalized to higher dimensions. That is, this can be done if coefficients of weak Landau--Ginzburg models of Givental type guarantee that
the base locus of the pencil of hypersurfaces in a toric variety we compactify in is a union of components corresponding to linear sections.
These components can be singular in the case of complete intersections, however the singularities ``come from the ambient space''
and can be resolved under a crepant resolution of the toric variety we compactify in.
This proves that the Calabi--Yau principle holds for weighted complete intersections.
However this works only if the Newton polytope of the weak Landau--Ginzburg model is reflexive.
This always holds for usual complete intersections but rarely holds for weighted ones.

Consider the matrix
$$
M_{d_1,\ldots,d_k;\iY}
=\left(%
\begin{array}{rrrr|r|rrrr|rrr}
  \iY & 0 & \ldots & 0 & \ldots & 0 & 0 & \ldots & 0 & -1 & \ldots & -1 \\
  0 & \iY & \ldots & 0 & \ldots & 0 & 0 & \ldots & 0 & -1 & \ldots & -1 \\
  \ldots & \ldots & \ldots & \ldots & \ldots & \ldots & \ldots & \ldots & \ldots & \ldots & \ldots & \ldots\\
  0 & 0 & \ldots & \iY & \ldots & 0 & 0 & \ldots & 0 & -1 & \ldots & -1 \\
  -\iY & -\iY & \ldots & -\iY & \ldots & 0 & 0 & \ldots & 0 & -1 & \ldots & -1 \\
  \hline
  \ldots & \ldots & \ldots & \ldots & \ldots & \ldots & \ldots & \ldots & \ldots & \ldots & \ldots & \ldots\\
  \hline
  0 & 0 & \ldots & 0 & \ldots & \iY & 0 & \ldots & 0 & -1 & \ldots & -1 \\
  0 & 0 & \ldots & 0 & \ldots & 0 & \iY & \ldots & 0 & -1 & \ldots & -1 \\
  \ldots & \ldots & \ldots & \ldots & \ldots & \ldots & \ldots & \ldots & \ldots & \ldots & \ldots & \ldots\\
  0 & 0 & \ldots & 0 & \ldots & 0 & 0 & \ldots & \iY & -1 & \ldots & -1 \\
  0 & 0 & \ldots & 0 & \ldots & -\iY & -\iY & \ldots & -\iY& -1 & \ldots & -1 \\
\hline
  0 & 0 & \ldots & 0 & \ldots & 0 & 0 & \ldots & 0 & \iY-1 & \ldots & -1 \\
  \ldots & \ldots & \ldots & \ldots & \ldots & \ldots & \ldots & \ldots & \ldots & \ldots & \ldots & \ldots \\
  0 & 0 & \ldots & 0 & \ldots & 0 & 0 & \ldots & 0 & -1 & \ldots & \iY-1 \\
\end{array}%
\right),
$$
depending on positive integer numbers $d_1,\ldots,d_k,\iY$,
which is formed from $k$ blocks of sizes $(d_i-1)\times d_i$ and one last block of size $\iX\times\iX$.
Define $k_{d_1,\ldots,d_k;\iY}$ as the number that is less by one than the number of integral points in the convex hull of rays of rows of the matrix.

\begin{theorem}[{\cite[Theorem 1]{Prz18}}]
\label{theorem: CY for CI}
Let $X\subset \PP^N$ be a Fano complete intersection of hypersurfaces of degrees $d_1,\ldots,d_k$.
Let $\iX=N+1-\sum d_i$. Let $f_X$ be a toric Landau--Ginzburg model of Givental's type for $X$.
Then $f_X$ admits a log Calabi--Yau compactification $f_X\colon Z\to \PP^1$ such that $f_X^{-1}(\infty)$
is a reduced divisor, which is a union of smooth rational varieties.
It consists of $k_{d_1,\ldots,d_k;\iX}$ components
and combinatorially it is given by a triangulation of a sphere.
\end{theorem}

\begin{proof}[Idea of the proof]
Similar to the proof of Theorem~\ref{theorem: Minkowski CY}.
\end{proof}

\begin{problem}
Find a formula for $k_{d_1,\ldots,d_k;\iX}$ in terms of $d_1,\ldots,d_k,\iX$.
\end{problem}

\begin{question}
By Corollary~\ref{proposition: fibers over infinity} and Theorem~\ref{theorem: CY for CI}, fibers over infinity
of log Calabi--Yau compactifications of toric Landau--Ginzburg models for Fano threefolds and complete intersections
are reduced and combinatorially are given by triangulations of spheres.
Does this hold in a general case?
\end{question}

\begin{remark}[cf. Remark~\ref{remark: generic 3-Laurent}]
Let $T$ be a smooth toric variety with $F(T)=\Delta$.
Let $f$ be a \emph{general} Laurent polynomial with $N(f)=\Delta$. The Laurent polynomial $f$ is a toric Landau--Ginzburg model for a pair $(T,D)$, where $D$ is a general divisor on $\widetilde T$.
Indeed, the period condition for it is satisfied by~\cite{Gi97b}.
Following the compactification procedure from Theorem~\ref{theorem: CY for CI}, one can see that the base locus $B$ is a union of smooth transversally intersecting subvarieties of codimension 2 (not necessary rational). This means that in the same way as above $f$
satisfies the Calabi--Yau condition. Finally the toric condition holds for $f$ tautologically. Thus $f$ is a toric Landau--Ginzburg model for $(T,D)$.
\end{remark}

In~\cite{PSh15a} Calabi--Yau compactifications for Fano complete intersections in usual projective spaces are constructed in a way
different from the one in Theorem~\ref{theorem: CY for CI}.
The method used in loc. cit. enables one to follow the number of reducible fibers components of the compactification.

\begin{theorem}[{\cite[Theorem 1.2]{PSh15a}}]
\label{theorem: CI components}
Conjecture~\ref{conjecture: Hodge-components} holds for Fano complete intersections.
\end{theorem}

\subsection{Toric Landau--Ginzburg models}
\label{section: toric LG CI}
The toric variety given by a polytope dual to a Newton polytope of toric Landau--Ginzburg model
enables one to show that Landau--Ginzburg models for weighted complete intersections satisfy the toric
condition. Here, unlike in Theorem~\ref{theorem: CY for CI}, we do not need integrality
of the polytope. Recall that Facts~\ref{fact: toric 1} and~\ref{fact: toric 3} enable one
to define the toric variety whose fan polytope is the Newton polytope for the polynomial~\eqref{formula: CI LG} by equations.
Recall also that these equations are homogenous relations on integral points of the Newton polytope.
The shape of the polynomial shows that the polytope is given by ``triangles'', so the relations are of Veronese type.
In other words the toric degeneration corresponding to the polynomial~\eqref{formula: CI LG}
is the image by Veronese map of the complete intersection
$$
\left\{
  \begin{array}{l}
    z_{1,1}\cdot \ldots\cdot z_{1,r_1}=z_{0,1}^{d_1} \\
    \ldots \\
    z_{k,1}\cdot \ldots\cdot z_{k,r_k}=z_{0,1}^{d_k}
  \end{array}
\right.
$$
in $\PP[z_{i,j}]$, $i\in [0,k]$, $j\in [1,r_i]$, where weights of $z_{i,j}$ correspond to elements of $E_i$ and the weight
of $z_{0,1}$ is $1$.

Thus, the following theorem holds.

\begin{theorem}[{\cite[Theorem 2.2]{ILP13}}]\label{theorem: CI are toric LG}
    There exist a flat degeneration of $X$ to the toric variety $T_{N(f_X)}$.
\end{theorem}

\begin{example}[The del Pezzo surface of degree 2]\label{ex:ci}
We now consider the example of del Pezzo surface of degree 2 and a description of its degeneration via generators and relations,
cf. Remark~\ref{remark: del Pezzo 1 and 2}. This is a hypersurface of degree 4 in $\PP(1,1,1,2)$. Its weak Landau--Ginzburg model is
$$
f_X=\frac{(x+y+1)^4}{xy}.
$$
The corresponding Newton polytope $\Delta_{f_{S_2}}$ has vertices equal to the columns of the matrix
$$
\left(
  \begin{array}{rrr}
3&-1&-1\\
-1&3&-1\\
  \end{array}
\right).
$$
The dual polytope $\nabla_{f_{S_2}}=\Delta_{f_{S_2}}^\vee$ thus has vertices equal to the columns of the matrix
$$
\left(
  \begin{array}{rrr}
	  1 & 0 & -1/2\\
	  0&1&-1/2\\
  \end{array}
\right).
$$
This is not a lattice polytope (so that the polygon $\Delta_{f_{S_2}}$ is not reflexive). However, its double dilation $\nabla^2_{f_{S_2}}=2\cdot \nabla_{f_{S_2}}$ is in fact integral.
The integral points of $\nabla$ are $u=(-1,-1)$ and $v_{ab}=(a,b)$ for $a,b\geq 0$, $a+b\leq 2$. These correspond to generators for the homogeneous coordinate ring of the toric degeneration $T$ in this (the doubleanticanonical) embedding.

Affine homogeneous relations among these lattice points correspond to binomial relations in the ideal of $T$. In this case, these relations are generated by five 2-Veronese type relations
\begin{align*}
	&v_{20}+v_{02}=2v_{11}, &v_{20}+v_{01}=v_{10}+v_{11},\\
	&v_{20}+v_{00}=2v_{10}, &v_{02}+v_{10}=v_{01}+v_{11},\\
	&v_{02}+v_{00}=2v_{01}
\end{align*}
together with the relation
$$
u+v_{11}=2v_{00}.
$$

On the other hand, consider the 2-Veronese embedding of $\{x_0x_1x_2=y_0^4\}\subset \PP(1,1,2,1)$.
In coordinates $z_{02}=x_0^2$, $z_{20}=x_1^2$, $w=x_2$, $z_{00}=y_0^2$,
$z_{11}=x_0x_1$, $z_{01}=x_0y_0$, $z_{10}=x_1y_0$, this hypersurface is given by the equation
$$
wz_{11}=z_{00}^2
$$
together with five 2-Veronese-type equations
\begin{align*}
	&z_{20}z_{02}=z_{11}^2, &z_{20}z_{01}=z_{10}z_{11},\\
	&z_{20}z_{00}=z_{10}^2, &z_{02}z_{10}=z_{01}z_{11},\\
	&z_{02}z_{00}=z_{01}^2.
\end{align*}
These correspond to the affine homogeneous relations above, so we can in fact realize our $T$ as the hypersurface $\{x_0x_1x_2=y_0^4\}\subset \PP(1,1,2,1)$.
Thus, by degenerating the equation defining $S_2$, we get a degeneration of the del Pezzo surface of degree 2 to $T$.
\end{example}

Proposition~\ref{proposition:periods-for-proj-space}, Theorem~\ref{theorem: CY for CI}, and Theorem~\ref{theorem: CI are toric LG} imply the following.

\begin{corollary}
\label{corollary: CI are toric LG}
Smooth Fano complete intersections have toric Landau--Ginzburg models.
\end{corollary}

\subsection{Boundness of complete intersections}
In the previous subsections we discussed toric Landau--Ginzburg models for weighted complete intersections.
One can easily bound the number of usual complete intersections of given dimension.
It turns out that the number of weighted complete intersections is also bounded.

That is, the following statement is a combination of~\cite[Theorem~1.1]{PrzyalkowskiShramov-Weighted}, \cite[Theorem~1.3]{ChenChenChen}, and~\cite[Corollary~5.3(i)]{PST17}.

\begin{theorem}[{see~\cite[Theorem 2.4]{PSh18}}]
\label{theorem:Fano-invariants}
Let $X$ be a smooth well formed Fano complete intersection in the weighted projective space~$\mbox{$\P=\PP(w_0,\ldots,w_N)$}$
which is not a section of a linear cone (in other words, all degrees defining the complete intersection
differ from weights of $\PP$).
Let $k$ be a codimension of $X$ in $\PP$, let $n=N-k=\dim (X)$, and let $l$ be a numbers of weights among $w_i$ that are equal to $1$.
Then

\begin{itemize}
\item[(i)] $w_N\le N$;

\item[(ii)] $k\le n$;

\item[(iii)] $l\ge k$.
\end{itemize}
\end{theorem}

In particular, this theorem implies the following.

\begin{proposition}[{\cite[\S 5]{PSh17b}}]
\label{proposition: CI dim 4 5}
Smooth Fano weighted complete intersections of dimension at most $5$ have very good nef-partitions.
In particular, they have weak Landau--Ginzburg models satisfying the toric condition.
\end{proposition}

Thus, the discussion above implies the following.

\begin{theorem}
%\label{theorem}
\label{theorem: nef toric}
Let $X$ be a smooth complete intersection in a well formed weighted projective space such that either $X$
is a complete intersection of Cartier divisors, or it is of codimension $2$,
or its dimension is not greater than $5$.
Than $X$ has a weak Landau--Ginzburg model satisfying the toric condition.
\end{theorem}

\begin{proof}
By Proposition~\ref{proposition: nef for cartier}, Theorem~\ref{theorem: nef for codim 2},
or Proposition~\ref{proposition: CI dim 4 5},
the variety $X$ has a very good nef-partition.
Thus, applying the change of variables from Section~\ref{section: weak CI},
we get a Laurent polynomial of type~\eqref{formula: CI LG}.
A standard combinatorial count (or straightforward generalization of Proposition~\ref{proposition:periods-for-proj-space}
for weighted projective spaces) shows that the polynomials satisfy the period condition.
Moreover, by Theorem~\ref{theorem: CI are toric LG} they satisfy the toric condition as well.
\end{proof}

\begin{question}
A lot of varieties have several different nef-partitions.
In~\cite{Li16} (see also~\cite{Pri}) it is shown that under some mild conditions Givental's Landau--Ginzburg models
for complete intersections in Gorenstein toric varieties corresponding to different nef-partitions are birational.
Is it true for smooth weighted complete intersections?
\end{question}

\section{Complete intersections in Grassmannians}
\label{section: CI in Grass}
It turns out that Givental's constructions can be applied not only to complete intersections in smooth toric varieties,
but also to complete intersections in varieties admitting ``good'' toric degenerations.
In this section we, following~\cite{PSh15b}, use such degenerations for Grassmannians $\G(n, k+n)$, $k,n\ge 2$,
and construct weak Landau--Ginzburg models for complete intersections therein.
We will use constructions for Landau--Ginzburg models analogous to Givental's ones, which are presented in~\cite{BCFKS97} and~\cite{BCFKS98}
(see also~\cite[B25]{EHX97}). We show that they can be presented by weak Landau--Ginzburg models following~\cite{PSh15b}.
Other methods of presenting them as weak Landau--Ginzburg models see in~\cite{PSh17}, see also~\cite{PSh14} and~\cite{Pri16}.

\subsection{Construction}
\label{section: Grass construction}

We define a quiver $\QQQ$
as a set of vertices
$$
\Ver(\QQQ)=\{(i,j)\mid i\in [1,k], j\in [1,n]\}\cup\{(0,1), (k,n+1)\}
$$
and a set of arrows
$\Ar(\QQQ)$ described as follows.
All arrows are either \emph{vertical} or \emph{horizontal}.
For any $i\in [1,k-1]$ and any $j\in [1,n]$ there is one
vertical arrow~\mbox{$\vv_{i,j}=\arrow{(i,j)}{(i+1,j)}$} that goes from the vertex $(i,j)$
down to the vertex $(i+1,j)$. For any $i\in [1,k]$ and any $j\in [1,n-1]$
there is one horizontal
arrow~\mbox{$\hh_{i,j}=\arrow{(i,j)}{(i,j+1)}$} that goes from the vertex $(i,j)$ to the right
to the vertex $(i,j+1)$.
We
also add an extra vertical arrow~\mbox{$\vv_{0,1}=\arrow{(0,1)}{(1,1)}$}
and an extra horizontal
arrow~\mbox{$\hh_{k,n}=\arrow{(k,n)}{(k,n+1)}$} to~\mbox{$\Ar(\QQQ)$},
see Figure~\ref{figure:quiverG36}.

\begin{figure}[htbp]
\begin{center}
\includegraphics[width=7cm]{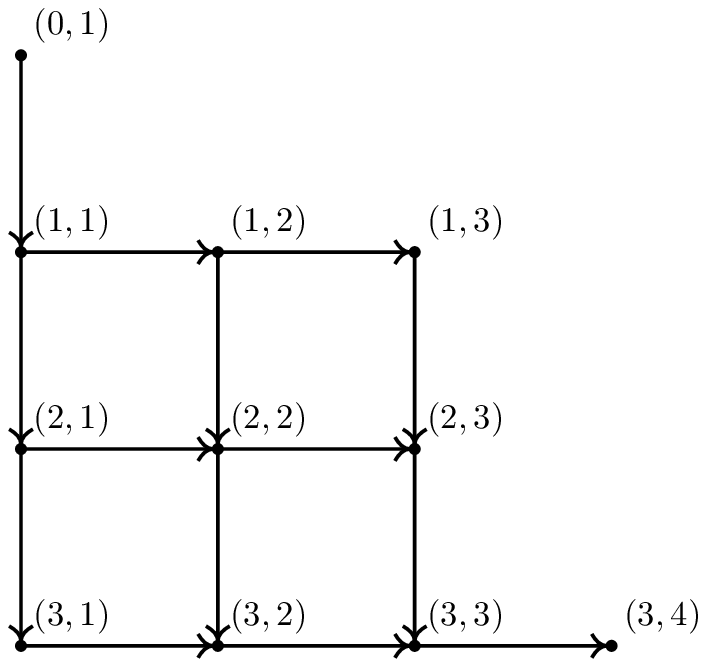}
\end{center}
\caption{Quiver $\QQQ$ for Grassmannian $\G(3,6)$}
\label{figure:quiverG36}
\end{figure}

For any arrow
$$\alpha=\arrow{(i,j)}{(i',j')}\in\Ar(\QQQ)$$
we define its \emph{tail} $t(\alpha)$
and its \emph{head} $h(\alpha)$ as the vertices $(i,j)$ and $(i',j')$, respectively.

For $r,s\in [0,k]$, $r<s$, we define a \emph{horizontal block}
$\HB(r,s)$ as a set of all vertical arrows $\vv_{i,j}$
with $i\in [r,s-1]$.
For example, the horizontal block $\HB(0,1)$ consists of a single
arrow $\vv_{0,1}$, while the horizontal block $\HB(1,3)$ consists
of all arrows $\vv_{1,j}$ and $\vv_{2,j}$, $j\in [1,n]$.
Similarly, for $r,s\in [1, n+1]$, $r<s$, we define a \emph{vertical block}
$\VB(r,s)$ as a set of all horizontal arrows $\hh_{i,j}$
with $j\in [r,s-1]$. Finally, for $r\in [0,k]$, $s\in [1,n+1]$
we define a \emph{mixed block}
$\MB(r,s)=\HB(r,k)\cup\VB(1,s)$.
For example, the mixed block $\MB(0,n)$ consists
of all arrows of $\Ar(\QQQ)$ except the arrow $\hh_{k,n}$.
When we speak about a block, we mean either a horizontal,
or a vertical, or a mixed block.
We say that the \emph{size}
of a horizontal block $\HB(r,s)$ and of a vertical block $\VB(r,s)$
equals $s-r$, and the size of a mixed block $\MB(r,s)$ equals $s+k-r$.

Let $B_1,\ldots,B_l$ be blocks. We say that they are \emph{consecutive}
if the arrow $\vv_{0,1}$ is contained in $B_1$, and
for any $p\in [1,l]$ the union $B_1\cup\ldots\cup B_p$ is a block.
This happens only in one of the following
two situations: either there is an index $p_0\in [1,l]$ and sequences
of integers $0<r_1<\ldots<r_{p_0}=k$ and $0<r_1'<\ldots<r_{l-p_0}'\le n+1$
such that
\begin{multline*}B_1=\HB(0,r_1), B_2=\HB(r_1,r_2), \ldots,
B_{p_0}=\HB(r_{p_0-1}, r_{p_0}),\\
B_{p_0+1}=\VB(0, r_1'), \ldots, B_l=\VB(r_{l-p_0-1}', r_{l-p_0}'),
\end{multline*}
or there is an index $p_0\in [1,l]$ and sequences
of integers $0<r_1<\ldots<r_{p_0-1}<k$ and
$0<r_1'<\ldots<r_{l-p_0-1}'\le n+1$
such that
\begin{multline*}
B_1=\HB(0,r_1), B_2=\HB(r_1,r_2), \ldots,
B_{p_0-1}=\HB(r_{p_0-2}, r_{p_0-1}),
B_{p_0}=\MB(r_{p_0}, r_1'),\\
B_{p_0+1}=\VB(r_1',r_2'), \ldots, B_l=\VB(r_{l-p_0-2}', r_{l-p_0-1}').
\end{multline*}
The first case occurs when there are no mixed blocks among $B_1,\ldots,B_l$,
and the second case occurs when one of blocks is mixed.

Let $S=\{x_1,\ldots,x_N\}$ be a finite set. %We denote the torus
%$$\Spec \TT[x_1,\ldots, x_N]\simeq (\C^*)^N$$
%by $\TT(S)$. Note that $x_1,\ldots, x_N$
%may be interpreted as coordinates on $\TT(S)$.
We introduce a set of variables
$\widetilde{V}=\{\widetilde{a}_{i,j}\mid i\in [1,k], j\in [1,n]\}$.
It is convenient to think that the variable $\widetilde{a}_{i,j}$
is associated to a vertex $(i,j)$ of the quiver $\QQQ$.
Laurent polynomials in the variables $\widetilde{a}_{i,j}$ are
regular functions on the torus $\TTT(\widetilde{V})$.
We also put $\widetilde{a}_{0,1}=\widetilde{a}_{k,n+1}=1$.

For any subset $A\subset\Ar(\QQQ)$ we define
a regular function
$$\widetilde{F}_A=\sum\limits_{\alpha\in A}\frac{\widetilde{a}_{h(\alpha)}}
{\widetilde{a}_{t(\alpha)}}
$$
on the torus $\TTT(\widetilde{V})$.

Let $Y$ be a complete intersection of hypersurfaces of degrees
$d_1,\ldots, d_l$ in $\G(k,n+k)$, $\sum d_i<n+k$. Consider consecutive blocks
$B_1, \ldots, B_l$ of size $d_1,\ldots, d_l$, respectively, and put
$$B_0=\Ar(\QQQ)\setminus\big(B_1\cup\ldots\cup B_l\big).$$

Let $\widetilde{L}\subset\TTT(\widetilde{V})$ be the subvariety
defined by equations
$$\widetilde{F}_{B_1}=\ldots=\widetilde{F}_{B_l}=1.$$
In~\cite{BCFKS97} and~\cite{BCFKS98}
it was suggested that a Landau--Ginzburg model for $Y$
is given by the variety $\widetilde{L}$ with
superpotential given by the function
$\widetilde{F}_{B_0}$.
We call it  model \emph{of type BCFKS}.

The main result of this subsection is the following.

\begin{theorem}[{\cite[Theorem 2.2]{PSh15b}}]\label{theorem: CI in Grass Laurent}
The subvariety $\widetilde{L}$ is birational to a to\-rus $\mbox{$\Y\simeq(\C^*)^{nk-l}$}$,
and the birational equivalence $\widetilde{\tau}\colon\Y\dasharrow\widetilde{L}$
can be chosen so that
$\widetilde \tau^*\left(\widetilde F_{B_0}\right)$ is a
%the pull-back of $\widetilde{F}_{B_0}$ is a
regular
function on~$\Y$. In particular this function is given by a Laurent polynomial.
\end{theorem}

\begin{remark}\label{remark:order}
The Laurent polynomial provided by Theorem~\ref{theorem: CI in Grass Laurent}
may significantly change if one takes the degrees $d_1, \ldots, d_l$
in a different order (cf. Examples~\ref{example:1121}
and~\ref{example:1112}).
\end{remark}

To prove Theorem~\ref{theorem: CI in Grass Laurent}
we will use slightly more convenient coordinates than $\widetilde{a}_{i,j}$.
Make a monomial change of variables $\psi\colon \TTT(V)\to\TTT(V)$ defined by
\begin{equation}
\label{eq: psi}
a_{i,j}=\widetilde{a}_{i,j}\cdot \widetilde{a}_{k,n},\quad a=\widetilde{a}_{k,n}.
\end{equation}
Put
$$V=\{a_{i,j}\mid i\in [1,k], j\in [1,n], (i,j)\neq (k,n)\}\cup \{a\}.$$
Put $a_{k,n}=1$ and $a_{0,1}=a_{k,n+1}=a$ for convenience.
As above, for any subset $A\subset\Ar(\QQQ)$ we define
a regular function
$$F_A=\sum\limits_{\alpha\in A}\frac{a_{h(\alpha)}}
{a_{t(\alpha)}}
$$
on the torus $\TTT(V)$. Let
$L\subset\TTT(V)$ be the subvariety
defined by equations
$$F_{B_1}=\ldots=F_{B_l}=1.$$
We are going to check that
the subvariety $L$ is birational to a torus $\Y\simeq(\C^*)^{nk-l}$,
and the birational equivalence $\tau\colon\Y\dasharrow L$
can be chosen so that
the pull-back of $F_{B_0}$ is a regular
function on $\Y$.
Obviously, the latter assertion is equivalent to
Theorem~\ref{theorem: CI in Grass Laurent}.

The following assertion is well known and easy to check.

\begin{lemma}
\label{lemma:torus}
Let $\X$ be a variety with a free action of a torus $\TTT$. Put $\Y=\X/\TTT$,
and let $\varphi\colon \X\to \Y$ be the natural projection.
Suppose that $\varphi$ has a section $\sigma\colon \Y\to \X$. Then one has an isomorphism
\begin{equation*}%\label{eq:iso}
\xi\colon \X\stackrel{\sim}\to \TTT\times \Y.
\end{equation*}
Moreover, suppose that a function $F\in \Gamma(\X,\O_{\X})$
is semi-invariant with respect to the $\TTT$-action,
i.\,e. there is a character $\chi$ of $\TTT$
such that for any $x\in \X$ and $t\in \TTT$ one has $F(tx)=\chi(t)F(x)$.
Then there is a function $\bar{F}\in\Gamma(\Y,\O_{\Y})$ such that
$F=\xi^*\big(\chi\cdot\bar{F}\big)$.
\end{lemma}

Recall that $B_1,\ldots,B_l$ are consecutive blocks.
In particular, the arrow $\vv_{0,1}$ is contained in~$B_1$.

We are going to define the weights $\wt_1, \ldots, \wt_l$ of the vertices
of $\QQQ$ so that the following properties are satisfied.
Consider an arrow $\alpha%=\arrow{(i,j)}{(i',j')}
\in\Ar(\QQQ)$.
Then
$$
%\wt_p(i',j')-\wt_p(i,j)=
\wt_p\left(h(\alpha)\right)-wt_p \left(t(\alpha)\right)=
\left\{
\begin{array}{l}
-1 \text{\ if\ } \alpha\in B_p,\\
0 \text{\ if\ } \alpha\notin B_p \text{\ and\ } \alpha\neq\hh_{k,n}.
\end{array}
\right.
$$
Also, for any $p\in [1,l]$ we require the following properties:
\begin{itemize}
\item one has $\wt_p(i,j)\ge 0$ for all $(i,j)$;

\item one has $\wt_p(k,n)=0$, so that
$$\wt_p(k,n+1)-\wt_p(k,n)=\wt_p(k,n+1)\ge 0;$$

\item one has $\wt_p(0,1)=\wt_p(k,n+1)$.
\end{itemize}

Actually, there is only one way to assign weights
so that the above requirements are met.
Choose an index~\mbox{$p\in [1,l]$}.
If $B_p=\HB(r,s)$ is a horizontal
block, we put
$$
\wt_p(i,j)=
\left\{
\begin{array}{l}
s-i, \text{\ if\ } i\in [r,s], j\in [1,n],\\
0, \text{\ if\ } i\in [s+1,k], j\in [1,n],\\
s-r, \text{\ if\ } i\in [1,r-1], j\in [1,n], \text{\ or\ } (i,j)=(0,1).
\end{array}
\right.
$$
In particular, this gives
$\wt_p(0,1)=s-r$.
If $B_p=\MB(r,s)$ is a mixed block, we put
$$
\wt_p(i,j)=
\left\{
\begin{array}{l}
(k-i)+(s-j), \text{\ if\ } i\in [r,k], j\in [1,s],\\
k-i, \text{\ if\ } i\in [r,k], j\in [s+1,n],\\
(k-r)+(s-j), \text{\ if\ }
i\in [1,r-1], j\in [1,s], \text{\ or\ } (i,j)=(0,1),\\
k-r, \text{\ if\ } i\in [1,r-1], j\in [s+1,n].
\end{array}
\right.
$$
If $B_p=\VB(r,s)$ is a vertical block,
we put
$$
\wt_p(i,j)=
\left\{
\begin{array}{l}
s-j, \text{\ if\ } i\in [1,k], j\in [r,s],\\
s-r, \text{\ if\ } i\in [1,k], j\in [1,r-1], \text{\ or\ } (i,j)=(0,1),\\
0, \text{\ if\ } i\in [1,k], j\in [s+1,n].
\end{array}
\right.
$$
Finally, we always put $\wt_p(k,n+1)=\wt_p(0,1)$.

To any block $B$ we associate a \emph{weight vertex} of the quiver $\QQQ$
as follows. If $B=\HB(r,s)$ is a horizontal block,
then its weight vertex is $(s-1,1)$. If $B$ is a mixed block $\MB(r,s)$ or
a vertical %weight\footnote{VERTICAL}
block $\VB(r,s)$, then its weight vertex is $(k,s-1)$.
If $B$ is a block and $(i,j)$ is its weight vertex, we
define the \emph{weight variable} of $B$ to be $a_{i,j}$
provided that $(i,j)\neq (0,1)$, and to be $a$ otherwise.

An example of weights assignment corresponding
to Grassmannian $\G(3,6)$ and mixed block $B=\MB(2,2)$
is given on Figure~\ref{figure:weightsG36}.
The solid arrows are ones that are
contained in $B$, while the dashed
arrows are those of $\Ar(\QQQ)\setminus B$.
The weight vertex $(3,1)$ of $B$ is marked by a white circle.

\begin{figure}[htbp]
\begin{center}
\includegraphics[width=7cm]{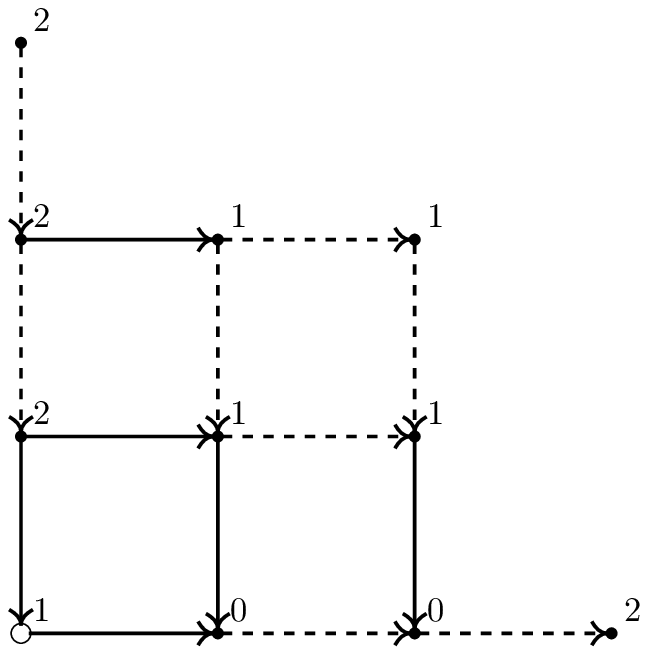}
\end{center}
\caption{Weights for Grassmannian $\G(3,6)$ and mixed block $\MB(2,2)$}
\label{figure:weightsG36}
\end{figure}

\begin{example}\label{example:weight-coordinates}
Consider the quiver $\QQQ$ corresponding to the Grassmannian $\G(3,6)$
(see Figure~\ref{figure:quiverG36}).
Suppose that $l=4$, $B_1=\HB(0,1)$, $B_2=\HB(1,2)$, $B_3=\MB(2,2)$,
and~\mbox{$B_4=\VB(2,3)$}.
Then the weight vertices of the blocks are $(0,1)$,
$(1,1)$, $(3,1)$, and $(3,2)$, respectively, and the weight variables
are $a$, $a_{1,1}$, $a_{3,1}$, and $a_{3,2}$.
\end{example}

Consider a torus
$$
\X=\TTT(V)\simeq(\C^*)^{nk}
$$
and a torus $\TTT\simeq (\C^*)^l$ with coordinates $w_1,\ldots,w_l$.
Define an action of $\TTT$ on $\X$ by
$$
(w_1,\ldots, w_l)\cdot a_{i,j}=w_1^{\wt_1(i,j)}\cdot\ldots\cdot w_l^{\wt_l(i,j)}\cdot a_{i,j}
$$
for all $i\in [1,k]$, $j\in [1,n]$, $(i,j)\neq (k,n)$, and
$$
(w_1,\ldots, w_l)\cdot a=
w_1^{\wt_1(0,1)}\cdot\ldots\cdot w_l^{\wt_l(0,1)}\cdot a.
$$

Using nothing but the basic properties of weights, we obtain the following lemmas.

\begin{lemma}\label{lemma:weights-Fi}
Fix $p\in [1,l]$. Then $F_{B_p}$
is a semi-invariant function on $\X$ with respect to the action of $\TTT$
with weight $w_p^{-1}$.
\end{lemma}

Recall that
$$B_0=\Ar(\QQQ)\setminus \big(B_1\cup\ldots\cup B_l\big).$$
Put $A=B_0\setminus\{\hh_{k,n}\}$.
Note that %$F_{\{\hh_{k,n}\}}=a$, and
$F_{B_0}=F_{A}+a$.

\begin{lemma}\label{lemma:weights-FA}
The function $F_{A}$ is invariant with respect to the action of $\TTT$.
On the other hand, the function $a$ is semi-invariant with weight
$$\mu(w)=w_1^{d_1}\cdot\ldots\cdot w_l^{d_l}.$$
\end{lemma}

Consider the quotient $\Y=\X/\TTT$,
and let $\varphi\colon \X\to \Y$ be the
natural projection.
Let $x_1,\ldots, x_l$ be weight variables
of the blocks $B_1,\ldots, B_l$, respectively, and
$\Sigma\subset \X$ be the subvariety defined by
equations
$$\{x_i=1\mid i\in [1,l]\}\subset \X.$$
Note that $\TTT$ acts on a coordinate $x_i$ multiplying
it by
$w_i\cdot N_i$, where $N_i$ is a monomial in~\mbox{$w_{i+1},\ldots, w_l$}.
In other words, define the matrix $M$
by
$$
(w_1,\ldots,w_l)\cdot x_i=\prod w_j^{M_{i,j}} x_i.
$$
Then $M$ is integral upper-triangular %unipotent
matrix
with %non-negative integral entries.
units on the diagonal.
Thus $\Sigma$ has a unique common point
with any fiber of $\varphi$. Therefore, there exists a section
$\sigma\colon \Y\to \X$ of the projection $\varphi$ whose image is $\Sigma$.
Also, we see that the action of $\TTT$ on $\X$ is free.
By Lemma~\ref{lemma:torus} we conclude that $\X\simeq \TTT\times \Y$.
In particular, one has $\Y\simeq (\C^*)^{nk-l}$.

Let $V'$ be the set of all variables of
$V$ except for $x_1,\ldots, x_l$.
We regard the variables of $V$ as coordinates on $\X$,
and the variables of $V'$ as coordinates on $\Y\simeq\TTT(V')$.
In these coordinates the morphism $\sigma$ is given
in a particularly simple way. Namely,
for any point $y\in \Y$ the point $\sigma(y)$
has all weight coordinates equal to $1$,
and the other coordinates equal to the corresponding coordinates of $y$.

\begin{example}\label{example:pi-sigma}
In the notation of Example~\ref{example:weight-coordinates}
one has
$$
\X=\TTT\big(\{a, a_{1,1}, a_{1,2}, a_{1,3}, a_{2,1}, a_{2,2}, a_{2,3}, a_{3,1}, a_{3,2}\}\big)
$$
and
$$
\Y=\TTT\big(\{a_{1,2}, a_{1,3}, a_{2,1}, a_{2,2}, a_{2,3}\}\big).
$$
The action of the torus $\TTT\simeq (\C^*)^4$ is defined
by the matrix
$$
M=\left(
    \begin{array}{cccc}
      1 & 1 & 2 & 1 \\
      0 & 1 & 2 & 1 \\
      0 & 0 & 1 & 1 \\
      0 & 0 & 0 & 1 \\
    \end{array}
  \right)
$$
as
\begin{multline*}
(w_1,w_2,w_3,w_4)\colon \big(a, a_{1,1}, a_{1,2}, a_{1,3}, a_{2,1}, a_{2,2}, a_{2,3}, a_{3,1}, a_{3,2}\big)
\mapsto\\ \mapsto
\big(w_1w_2w_3^2w_4\cdot a, w_2w_3^2w_4\cdot a_{1,1},
w_2w_3w_4\cdot a_{1,2},
w_2w_3\cdot a_{1,3},\\
w_3^2w_4\cdot a_{2,1},
w_3w_4\cdot a_{2,2}, w_3\cdot a_{2,3}, w_3w_4\cdot a_{3,1},
w_4\cdot a_{3,2}\big).
\end{multline*}
(Note that the weights corresponding to the block $B_3$ can be seen on Figure~\ref{figure:weightsG36}.)
The matrix
$$
M^{-1}=\left(
    \begin{array}{rrrr}
      1 & -1 & 0 & 0 \\
      0 & 1 & -2 & 1 \\
      0 & 0 & 1 & -1 \\
      0 & 0 & 0 & 1 \\
    \end{array}
  \right)
$$
gives $w_1^{-1}=\frac{a}{a_{1,1}}$, $w_2^{-1}=\frac{a_{1,1}a_{3,2}}{a_{3,1}^2}$, $w_3^{-1}=\frac{a_{3,1}}{a_{3,2}}$, and $w_4^{-1}=a_{3,2}$, so
the projection $\varphi\colon \X\to \Y$ is given by
\begin{multline*}
\varphi\colon \left(a, a_{1,1}, a_{1,2}, a_{1,3}, a_{2,1}, a_{2,2}, a_{2,3}, a_{3,1}, a_{3,2}\right)
\mapsto\\ \mapsto
\left( \frac{a_{3,1}}{a_{1,1}a_{3,2}}\cdot a_{1,2},
\frac{a_{3,1}}{a_{1,1}}\cdot a_{1,3},
\frac{a_{3,2}}{a_{3,1}^2}\cdot a_{2,1},
\frac{1}{a_{3,1}}\cdot a_{2,2}, \frac{a_{3,2}}{a_{3,1}}\cdot a_{2,3} \right),
\end{multline*}
and the section $\sigma\colon \Y\to \X$ is given by
$$
\sigma\colon \big(a_{1,2}, a_{1,3}, a_{2,1}, a_{2,2}, a_{2,3}\big)\mapsto
\big(1, 1, a_{1,2}, a_{1,3}, a_{2,1}, a_{2,2}, a_{2,3}, 1, 1\big).
$$
\end{example}

Applying Lemma~\ref{lemma:weights-Fi} together with
Lemma~\ref{lemma:torus}, we see that there exist regular functions
$\bar{F}_p$, $p\in [1,l]$, on $\Y$ such that
under the identification $\X\simeq \TTT\times \Y$ one has
$$F_p=w_p^{-1}\cdot \varphi^*\bar{F}_p.$$
Similarly, applying Lemma~\ref{lemma:weights-FA} together with
Lemma~\ref{lemma:torus}, we see that there exist regular functions
$\bar{F}_{A}$ and $\bar{a}$ on $\Y$ such that
$F_{A}=\varphi^*\bar{F}_{A}$ and $a=\mu(w)\varphi^*\bar{a}$.

Consider a rational map
$$
y\mapsto \big(\bar{F_1}(y),\cdots,\bar{F}_l(y)\big)
$$
from $\Y$ to $\TTT$.
Define a rational map $\tau\colon \Y\dasharrow \X$
as
$$
y\mapsto \big(\bar{F_1}(y),\cdots,\bar{F}_l(y)\big)\cdot
\sigma(y).
$$
It is easy to see that
the closure of the image of $\Y$ under the map $\tau$ is the
subvariety~\mbox{$L\subset \X$}.
In particular, $\tau$ gives a birational equivalence between
$\Y$ and $L$.

Now it remains to notice that
$$\tau^*F_{A}=\tau^*\varphi^*\bar{F}_A=\bar{F}_A.$$
On the other hand,
one has
$$\tau^*a=
\mu\big(\bar{F_1}(y),\cdots,\bar{F}_l(y)\big)\sigma^*\varphi^*\bar{a}=
\mu\big(\bar{F_1}(y),\cdots,\bar{F}_l(y)\big)\bar{a}.
$$
This means that the map
$\widetilde \tau=\tau\varphi\psi$, where $\psi$ is given by formulas~\eqref{eq: psi} provides a birational map
required for Theorem~\ref{theorem: CI in Grass Laurent}.

\begin{remark}
\label{remark: algorithm}
The above proof of Theorem~\ref{theorem: CI in Grass Laurent} provides a very explicit
way to write down the Laurent polynomial $\tau^*F_{B_0}$. Namely,
consider a complete intersection~\mbox{$Y\subset \G(n,n+k)$}
of hypersurfaces of degrees $d_i$, $i\in [1,l]$.
The following cases may occur.
\begin{itemize}
  \item One has $d_1+\ldots+d_l\le k.$
Put $u_i=d_1+\ldots+d_i$ for $i\in [1,l]$.
Then the BCFKS Landau--Ginzburg model for $Y$ is birational to $(\CC^*)^{nk-l}$ with superpotential
$$
\sum_{i=u_{l}+1}^{k}\sum_{j=1}^n \frac{a_{i,j}}{a_{i-1,j}}+\sum_{i=1}^{k}\sum_{j=2}^n \frac{a_{i,j}}{a_{i,j-1}}
+a\left(\frac{a_{1,1}}{a}+\sum_{i=2}^{d_1}\sum_{j=1}^n \frac{a_{i,j}}{a_{i-1,j}}\right)^{d_1}
\prod_{p=2}^l
\left(\sum_{i=u_{p-1}}^{u_p}\sum_{j=1}^n \frac{a_{i,j}}{a_{i-1,j}}\right)^{d_p},
$$
where we put $a_{1,u_1-1}=1$ if $u_1>1$ and $a=1$ otherwise,
$a_{1,u_i-1}=1$ for $i\in[2,l]$, and $a_{k,n}=1$.

  \item One has $d_1+\ldots+d_l> k.$ Let $m\in [0,l-1]$
be the maximal index such that
$d_1+\ldots+d_m\le k.$
Put $u_i=d_1+\ldots+d_i$ for $i\in [1,m]$ and
$u_i=d_1+\ldots+d_i-k$ for $i\in [m+1,l]$.

If $m=0$, then the BCFKS Landau--Ginzburg model for $Y$ is birational to $(\CC^*)^{nk-l}$ with superpotential
\begin{multline*}
\sum_{i=1}^{k}\sum_{j=u_l+1}^n \frac{a_{i,j}}{a_{i,j-1}}+\\
+a\left(\frac{a_{1,1}}{a}+
\sum_{i=2}^{k}\sum_{j=1}^n \frac{a_{i,j}}{a_{i-1,j}}+
\sum_{i=1}^{k}\sum_{j=2}^{u_1} \frac{a_{i,j}}{a_{i,j-1}}\right)^{d_1}
\cdot \prod_{p=2}^{l}\left(\sum_{i=1}^k\sum_{j=u_{p-1}}^{u_{p}}
\frac{a_{i,j}}{a_{i,j-1}}\right)^{d_p}.
\end{multline*}
If $m>1$, then the BCFKS Landau--Ginzburg model for $Y$ is birational to $(\CC^*)^{nk-l}$ with superpotential
\begin{multline*}
\sum_{i=1}^{k}\sum_{j=u_l+1}^n \frac{a_{i,j}}{a_{i,j-1}}
+a\left(\frac{a_{1,1}}{a}+\sum_{i=2}^{d_1}\sum_{j=1}^n \frac{a_{i,j}}{a_{i-1,j}}\right)^{d_1}\cdot\\
\cdot \prod_{p=2}^m
\left(\sum_{i=u_{p-1}}^{u_p}\sum_{j=1}^n
\frac{a_{i,j}}{a_{i-1,j}}\right)^{d_p}\cdot
\left(\sum_{i=u_m}^{k}\sum_{j=1}^n \frac{a_{i,j}}{a_{i-1,j}}+\sum_{i=1}^k\sum_{j=2}^{u_{m+1}} \frac{a_{i,j}}{a_{i,j-1}}\right)^{d_{m+1}}\cdot \\ \cdot
\prod_{p=m+2}^{l}\left(\sum_{i=1}^{k}\sum_{j=u_{p-1}}^{u_p} \frac{a_{i,j}}{a_{i,j-1}}\right)^{d_p}
.
\end{multline*}

In both cases we put $a_{1,u_1-1}=1$ if $u_1>1$ and $a=1$ otherwise,
$a_{1,u_p-1}=1$ for $p\in[2,m]$, $a_{k,u_p-1}$ for $p\in [m+1,l]$, and $a_{k,n}=1$.

\end{itemize}

\end{remark}

\begin{example}[\cite{PSh14}]
Consider a smooth Fano fourfold $Y$
of index~$2$ that is a section of the Grass\-man\-ni\-an $\G(2,6)$
by $4$ hyperplanes.
A very weak LG model of $Y$ is given by
$$
f_Y=\frac{(a_{4}+a_{3})\cdot (a_{4}+a_{3}+a_{2})}
{a_{3}\cdot a_{2}\cdot a_{1}}+
\frac{a_{4}+a_{3}}
{a_{3}\cdot a_{2}}+
\frac{1}{a_{3}}+
\frac{1}{a_{4}}+
a_{4}+a_{3}+a_{2}+a_{1}.
$$
Put $\TTT=\TTT[a_1,a_2,a_3,a_4]$.
Consider a relative compactification of a family $f_Y\colon \TTT\to \Aff^1$ given by
an embedding of $\TTT$ into the projective space $\PP^4$ with homogeneous coordinates $a_0,\ldots,a_4$.
It is a family of compact singular Calabi--Yau threefolds.
The total space of this family admits a crepant resolution of singularities~\mbox{$LG(Y)$}.
Moreover one can check that $LG(Y)$ is a family of Calabi--Yau threefolds such that
its generic fiber is smooth, and $LG(Y)$ has exactly $12$ singular
fibers. Furthermore each of these singular fibers has exactly one singular point, and this point is an ordinary double singularity.
We expect that $LG(Y)$ satisfies Homological Mirror Symmetry conjecture.
The structure of singular fibers of $LG(Y)$ confirms this expectation.
Indeed, \mbox{by~\cite[Cor.~10.3]{Kuz06}} there is a full exceptional collection of length $12$ on $Y$.
On the other hand, by %Homological Mirror Symmetry
Homological Mirror Symmetry
conjecture the
category $D^b(coh\,Y)$
is equivalent to the Fukaya--Seidel category for a dual Landau--Ginzburg model.
\end{example}

\subsection{Periods}
\label{section: Grass periods}

In this subsection we discuss period integrals for Laurent polynomials obtained in Theorem~\ref{theorem: CI in Grass Laurent}.

Recall the definition of Givental's integral in our case.

Given a torus $\TTT(\{x_1,\ldots,x_r\})$ we call a cycle $\{|x_i|=\varepsilon_{i}\mid i\in [1,r]\}$ depending on some
real numbers $\varepsilon_i$ \emph{standard}.

\begin{definition}[see~\cite{BCFKS97}]
\label{def:integral}
\emph{An (anticanonical) Givental's integral} for $Y$ is an integral
\begin{equation*}
\label{eq:restricted integral}
I_Y^0=\int_{\delta}%\limits_{|x_i|=\varepsilon_i}
\frac{
\Omega (\{\widetilde a_{i,j}\})
}{\prod_{j=1}^l\left(1-\widetilde F_j\right)\cdot \left(1-t\widetilde F_0\right)} \in \CC[[t]]
\end{equation*}
for a standard cycle $\delta=\{|\widetilde a_{i,j}|=\varepsilon_{i,j}\mid i\in [1,k], j\in [1,n], \varepsilon_{i,j}\in \R_+\}$,
whose orientation %order on $\{\widetilde a_{i,j}\}$
is chosen such that $\left.I_Y^0\right|_{t=0}=1$.
\end{definition}

In~$\mbox{\cite[Conjecture 5.2.3]{BCFKS97}}$ it is conjectured
that $\widetilde{I}^G_0=I_G^0$, and a formula for~$\widetilde{I}^G_0$ is provided.
This conjecture was
proved for $n=2$ in~\cite[Proposition 3.5]{BCK03} and for any~$n\ge 2$
in~\cite{MR13}.
%For a complete intersection $Y$ i
In discussion after Conjecture~5.2.1 in~\cite{BCFKS98} it is explained
that from the latter theorems and the Quantum Lefschetz Theorem it follows that Givental's integral $I^0_{Y}$ equals $\widetilde{I}^Y_0$.
We summarize the results mentioned above as follows.

\begin{theorem}
\label{theorem: periods of CI}
Let $Y=Y_1\cap \ldots \cap Y_l\subset \G(n,k+n)$
be a smooth Fano complete intersection. Denote $d_i= \deg Y_i$ and $d_0=k+n-\sum d_i$.
Then %the series
$$
\widetilde{I}_0^Y=I^0_Y=\sum_{d\ge 0}\sum_{s_{i,j}\ge 0} \frac{\prod_{i=0}^l (d_id)!}{(d!)^{k+n}} \prod_{i=1}^{k-1} \prod_{j=1}^{n-1} \binom{s_{i+1,j}}{s_{i,j}} \binom{s_{i,j+1}}{s_{i,j}} t^{d_0d},
$$
where we put $s_{k,j}=s_{i,n}=d$. %, is a period of BCFKS Landau--Ginzburg model for $Y$.
\end{theorem}

It turns out that changes of variables constructed in Theorem~\ref{theorem: CI in Grass Laurent} preserve this period.

\begin{proposition}
\label{theorem: Grass I-series}
The period condition holds for Laurent polynomials given by Theorem~\ref{theorem: CI in Grass Laurent}.
In other words, Theorem~\ref{theorem: CI in Grass Laurent} provides weak Landau--Ginzburg models for Fano complete intersections
in Grassmannians. %$\G(n, k+n)$.
\end{proposition}

\begin{proof}
We follow the notation from Theorem~\ref{theorem: CI in Grass Laurent}.
A toric change of variables $\varphi\psi$ change coordinates $\{\widetilde a_{i,j}\}$ by coordinates $\{w_i\}\cup V'$.
One gets
\begin{multline*}
I_Y^0=\int_{\delta}
\frac{
\Omega (\{\widetilde a_{i,j}\})
}{\prod_{j=1}^l\left(1-\widetilde F_j\right)\cdot \left(1-t\widetilde F_0\right)} =\\=
\int_{\delta'} \Omega(V')
\wedge\left(\bigwedge_{j=1}^l\left(\frac{1}{2\pi \sqrt{-1}}\frac{dw_j}{w_j\cdot\left(1-\bar F_{j}/w_j\right)}\right)\right)\cdot \frac{1}{1-t
\bar F}
\end{multline*}
for an appropriate choice of an orientation on $\delta'$, where $\bar F= \bar{F}_{A}+\mu(w)\cdot \bar{a}$.
Following the birational isomorphism $\tau$, consider variables $u_i=w_i-\bar F_{i}$ instead of $w_i$.
Then, after appropriate choice of cycle $\Delta'$ (cf.~\cite[proof of Proposition 10.5]{PSh17}) one gets
\begin{multline*}
I_Y^0=
\int_{\delta'} \Omega(V')
\wedge \left(\bigwedge_{j=1}^l\left(\frac{1}{2\pi \sqrt{-1}}\frac{dw_j}{w_j-\bar F_{j}}\right)\right)\cdot \frac{1}{1-t \bar F}=\\=
\int_{\Delta'} \Omega(V')
\wedge\left(\bigwedge_{j=1}^l\left(\frac{1}{2\pi \sqrt{-1}}\frac{du_j}{u_j}\right)\right)\cdot \frac{1}{1-t F_u}=
\int_{\Delta} \frac{
\Omega(V')
}{1-t f}=\sum [f^i]t^i,
\end{multline*}
where $\Delta$ is a projection of $\Delta'$ on $\TTT(V)$ and $F_u$ is a result of replacement of $w_i$ by
of $u_i+F_{B_i}$ in $\bar F$.
\end{proof}

\begin{problem}[{cf.~\cite[Problem 17]{Prz13}}]
\label{problem: toric LG}
Let $Y$ be a Fano complete intersection in $\G(n, k+n)$, and let $f_Y$ be the Laurent polynomial for $Y$ given by Theorem~\ref{theorem: CI in Grass Laurent}. Prove that $f_Y$ is a toric Landau--Ginzburg model. Prove that the number of components
of a central fiber of a Calabi--Yau compactification for $f_Y$ is equal to $h^{1,\mathrm{dim}\,Y-1}(Y)+1$ (cf. Conjecture~\ref{conjecture: Hodge-components}).
\end{problem}

\begin{remark}
In~\cite{DH15} it was shown by other methods that BCFKS Landau--Ginzburg models are birational to algebraic tori.
Moreover, these Laurent polynomials representing the superpotentials
are recovered from toric degenerations. Thus if one show that these polynomials satisfy the period condition,
then they are weak Landau--Ginzburg models satisfying the toric condition.
\end{remark}

\begin{example}\label{example:1121}
Let $Y$ be a smooth
intersection of the Grassmannian $\G(3,6)$ with
a quadric and three hyperplanes.
Put $l=4$, $d_1=d_2=d_4=1$, and $d_3=2$.
The BCFKS Landau--Ginzburg model in this case
is birational to a torus
$$
\Y\simeq \TTT\big(\{a_{1,2}, a_{1,3}, a_{2,1}, a_{2,2}, a_{2,3}\}\big)
$$
with the superpotential
\begin{multline*}
f_Y=
\left(a_{2,1}+\frac{a_{2,2}}{a_{1,2}}+\frac{a_{2,3}}{a_{1,3}}\right)  %\cdot\\
\cdot\left(\frac{1}{a_{2,1}}+\frac{a_{3,2}}{a_{2,2}}+
\frac{1}{a_{2,3}}+a_{1,2}+\frac{a_{2,2}}{a_{2,1}}+1\right)^2\cdot
\left(\frac{a_{1,3}}{a_{1,2}}+\frac{a_{2,3}}{a_{2,2}}+1\right)
\end{multline*}
given by Remark~\ref{remark: algorithm}.
By
Theorem~\ref{theorem: periods of CI} (see also~\cite[Example~5.2.2]{BCFKS97}) one has
\begin{multline}\label{eq:const-term-1121}
I^0_Y=\sum_{d, b_1,b_2,b_3,b_4}\frac{(2d)!}{(d!)^{2}} \binom{b_2}{b_1}  \binom{b_3}{b_1}  \binom{d}{b_2}  \binom{b_4}{b_2}
 \binom{b_4}{b_3}  \binom{d}{b_3}  \binom{d}{b_4}^2t^{d}=\\
 =1+12t+756t^2+78960t^3+10451700t^4+1587790512t^5+263964176784t^6+\\
 +46763681545152t^7+8685492699286260t^8+\cdots.
\end{multline}
One can check that
the first few terms we write down on the right
hand side of the formula~\eqref{eq:const-term-1121}
equal the first few terms of the series~\mbox{$\sum [f_Y^i]t^i$}.
\end{example}

\begin{example}\label{example:1112}
Let $Y$ be a smooth
intersection of the Grassmannian $\G(3,6)$ with
a quadric and three hyperplanes, i.\,e. the variety
that was already considered in Example~\ref{example:1121}.

Put $l=4$, $d_1=d_2=d_3=1$, and $d_4=2$.
One has
$$
\Y=\TTT\big(\{a_{1,2}, a_{1,3}, a_{2,2}, a_{2,3}, a_{3,1}\}\big).
$$
By Remark~\ref{remark: algorithm} we get
\begin{multline*}
f_Y=\left(1+\frac{a_{2,2}}{a_{1,2}}+\frac{a_{2,3}}{a_{1,3}}\right)\cdot
\left(a_{3,1}+\frac{1}{a_{2,2}}+\frac{1}{a_{2,3}}\right)
\cdot \left(a_{1,2}+a_{2,2}+\frac{1}{a_{3,1}}+\frac{a_{1,3}}{a_{1,2}}
+\frac{a_{2,3}}{a_{2,2}}+1\right)^2.
\end{multline*}
One can check that the first few constant terms $[f_Y^i]$
coincide with the first few terms of the series presented on the right
side of the formula~\eqref{eq:const-term-1121}.
Note that the Laurent polynomial
$f_Y$ can't be obtained from the polynomial from Example~\ref{example:1121}
by monomial change of variables (cf. Remark~\ref{remark:order}).
It would be interesting to find out if
these two Laurent polynomials
are mutational equivalent (cf.~\cite[Theorem~2.24]{DH15}).
\end{example}

%\thebibliography{XXX}

\end{document}